\documentclass[a4paper, 11pt]{article}
\usepackage[usenames, dvipsnames]{xcolor}
\usepackage{tikz}

\usepackage{mathpazo}
\usepackage{amssymb,hyperref,amsthm}
\usepackage{euscript}
\usepackage{flafter}
\usepackage{pstricks}
\usepackage[latin1]{inputenc}
\usepackage{amsmath}
\usepackage{amsfonts}
\usepackage{pstricks-add}
\usepackage{yfonts}
\usepackage{egothic}
\usepackage{dsfont}
\usepackage{bm}
\usepackage{variations}

\hypersetup{
    linktoc=page,
   linkcolor=red,          % color of internal links
  citecolor=blue,        % color of links to bibliography
    filecolor=blue,      % color of file links
   urlcolor=cyan,
    colorlinks=true           % color of external links
}

\usepackage[refpage]{nomencl}
\usepackage{makeidx}
\makeindex

\usepackage{enumitem}
\usepackage{mathrsfs} %%% Nice caligraphic fonts
\evensidemargin\oddsidemargin
\DeclareFontFamily{U}{mathx}{\hyphenchar\font45}
\DeclareFontShape{U}{mathx}{m}{n}{
      <5> <6> <7> <8> <9> <10>
      <10.95> <12> <14.4> <17.28> <20.74> <24.88>
      mathx10
      }{}
\DeclareSymbolFont{mathx}{U}{mathx}{m}{n}
\DeclareMathAccent{\widecheck}{\mathalpha}{mathx}{"71}
\newcommand{\eqnsection}{
\renewcommand{\theequation}{\thesection.\arabic{equation}}
   \makeatletter
   \csname  @addtoreset\endcsname{equation}{section}
   \makeatother}
\eqnsection
\def\e{{\mathbb E}}
\def\p{{\mathbb P}}
\def\P{{\bf P}}
\def\E{{\bf E}}
\def\Q{{\bf Q}}
\def\N{{\mathbb N}}
\def\T{{\mathbb T}}
\def\L{{\mathbb L}}
\def\1{{\mathds{1}}}
\def\en{\mathcal{E}}
\def\d{\, \mathrm{d}}
\def\d{\mathtt{d}}
\def\B{\mathfrak{B}}
\def\D{\mathfrak{D}}
\def\L{\mathfrak{L}}

\def\eL{{\overline{L}}}

\def\mS{{\underline{S}}}
\def\mV{{\underline{V}}}
\def\MV{{\overline{V}}}
\def\MS{{\overline{S}}}
\def\Hinf{\mathcal{H}}
\def\Ren{\mathcal{R}}

\newtheorem{theo}{Theorem}[section]
\newtheorem{prop}[theo]{Proposition}
\newtheorem{lem}[theo]{Lemma}

\newtheorem{cor}[theo]{Corollary}

\newtheorem{rem}[theo]{Remark}
\newtheorem{fact}[theo]{Fact}

\newcommand{\R}{\mathbb{R}}
\renewcommand{\T}{\mathbb{T}}
\newcommand{\calL}{\mathcal{L}}
\newcommand{\calF}{\mathcal{F}}
\newcommand{\ind}[1]{\mathbf{1}_{\left\{ #1 \right\}}}
\renewcommand{\epsilon}{\varepsilon}

\def\cb{{\bf c}}
\def\Cb{{\bf C}}
\def\cf{{\bf f}}

\newcommand{\V}{\mathbb{V}ar}
\usepackage{tikz}

\title{\bf Heavy range of the randomly biased walk on Galton-Watson trees in the slow movement regime}

\author{%Pierre Andreoletti\footnote{Laboratoire MAPMO - C.N.R.S. UMR 7349 - F\'ed\'eration Denis-Poisson, Universit\'e d'Orl\'eans(France). 
Xinxin Chen\footnote{Institut Camille Jordan - C.N.R.S. UMR 5208 - Universit\'e Claude Bernard Lyon 1 (France) Supported by ANR MALIN}}

\begin{document}

\baselineskip=17pt
\setcounter{page}{1}

\maketitle
%\listoftodos
\begin{abstract}
We consider the randomly biased random walk on trees in the slow movement regime as in \cite{HuShi16}, whose potential is given by a branching random walk in the boundary case. We study the heavy range up to the $n$-th return to the root, i.e., the number of edges visited more than $k_n$ times. For $k_n=n^\theta$ with $\theta\in(0,1)$, we obtain the convergence in probability of the rescaled heavy range, which improves one result of \cite{AD18+}. 
\bigskip

\end{abstract}

\section{Introduction}\label{intro}
Let $\T$ be a supercritical Galton-Watson tree rooted at $\rho$. And to any vertex $x\in\T\setminus\{\rho\}$, we assign a random bias $A_x\geq0$. For any vertex $x\in\T$, denote its parent by $x^*$ and denote its children by $x^1, x^2, \cdots, x^{N_x}$ where $N_x$ denotes the number of its children which could be $0$ if there is none. %For the root $\rho$, we add artificially a vertex $\rho^*$ to be its parent. 
Now for given $\en:=\{\T, (A_x)_{x\in\T\setminus\{\rho\}}\}$, let $(X_n)_{n\geq0}$ be a nearest-neighbour random walk on $\T$, started from $X_0=\rho$, with the biased transition probabilities: for any $x,y\in\T$,
\begin{equation}\label{transitionprob}
p^\en(x,y)=\begin{cases}
%\p^\mathcal{E}(X_{n+1}=x_j\vert X_n=x)=
\frac{A_{x^j}}{1+\sum_{i=1}^{N_x}A_{x^i}}, \textrm{ if } y=x^j \textrm{ for some } j\in\{1,\dots, N_x\}\\
%\p^\mathcal{E}(X_{n+1}=x^*\vert X_n=x)=
\frac{1}{1+\sum_{i=1}^{N_x}A_{x^i}};\textrm{ if } y=x^*.
\end{cases}
\end{equation}
For convenience, to the root $\rho$, we add artificially a vertex $\rho^*$ to be its parent and let \eqref{transitionprob} holds also for $x=\rho$ with $p^\en(\rho^*,\rho)=1$. Obviously, this is a random walk in random environment. In particular, when $A_x$ equals some constant $\lambda>0$ for any $x$, this is known as $\lambda$-biased random walk on Galton-Watson tree, which was introduced and deeply studied by Lyons \cite{Lyons90, Lyons92} and Lyons, Pemantle and Peres \cite{LPP95, LPP96}.

In our setting, we assume that $\{A_{x^1},\cdots, A_{x^{N_x}}\},\ x\in\T$ are i.i.d. copies of the point process $\mathcal{A}=\{A_1,\cdots, A_{N}\}$ where $N\in\mathbb{N}$ represents the offspring of the Galton-Watson tree $\T$. Let $\P$ denote the probability measure of the environment $\en$. Given the environment $\en$, denote the quenched probability by $\p^\en$. Then $\p(\cdot):=\int\p^\en(\cdot)\P(d\en)$ denotes the annealed probability. We always assume $\E[N]>1$ so that $\T$ is supercritical, i.e. $\T$ survives with positive probability. Let $\P^*(\cdot)=\P(\cdot\vert\T\textrm{ survives })$ and $\p^*(\cdot)=\p^*(\cdot\vert \T\textrm{ survives} )$.

In this setting, we could describe the environment $\en$ by a branching random walk. For any $x\in\T$, let $|x|$ be its generation, i.e., the graph distance between the root $\rho$ and $x$. For any $x,y\in\T$, we write $x\leq y$ if $x$ is an ancestor of $y$ and $x<y$ if $x\leq y$ and $x\neq y$. Let $[\![\rho, x]\!]$ be the ancestral line of $x$, that is, the set of vertices on the unique shortest path from $\rho$ to $x$. In this ancestral line, for any $0\leq i\leq |x|$, let $x_i$ be the ancestor of $x$ in the $i$-th generation; in particular, $x_0=\rho$ and $x_{|x|}=x$. Then, define
\[
V(x):=-\sum_{i=1}^{|x|}\log A_{x_i}, \forall x\in\T\setminus\{\rho\},
\]
with $V(\rho):=0$. Usually, $(V(x), x\in\T)$ is viewed as the potential of the random walk. Immediately, we see that $(V(x), x\in\T)$ is a branching random walk whose law is governed by that of $\mathcal{L}:=\{V(x), |x|=1\}$. Note that $\mathcal{A}$ is distributed as $\{e^{-V(x)},|x|=1\}$.

From now on, we write the environment by this branching random walk, i.e., $\en=(V(x), x\in\T)$. Then, the transition probabilities of the random walk $(X_n)_{n\geq0}$ can be written as follows
\begin{equation}
\begin{cases}
\p^\mathcal{E}(X_{n+1}=x^*\vert X_n=x)=\frac{e^{-V(x)}}{e^{-V(x)}+\sum_{y: y^*=x}e^{-V(y)}}\\
\p^\mathcal{E}(X_{n+1}=y\vert X_n=x)=\frac{e^{-V(y)}\ind{y^*=x}}{e^{-V(x)}+\sum_{z: z^*=x}e^{-V(z)}}.
\end{cases}
\end{equation}

Throughout the paper, we assume that the branching random walk is in the boundary case, that is,
\begin{equation}\label{boundarycase}
\E\left[\sum_{|x|=1}e^{-V(x)}\right]=1,\qquad \E\left[\sum_{|x|=1}V(x)e^{-V(x)}\right]=0.
\end{equation}
We also assume the following integrability condition: there exists certain $\delta_0>0$ such that
\begin{equation}\label{Integrability}
\E\left[\sum_{|x|=1}e^{-(1+\delta_0)V(x)}\right]+\E\left[\sum_{|x|=1}e^{\delta_0 V(x)}\right]<\infty.
\end{equation}
In addition, we assume that 
\begin{equation}\label{MomCond}
\E[N^2]+\E\left[\left(\sum_{|u|=1}(1+V_+(u))^2e^{-V(u)}\right)^2\right]<\infty,
\end{equation}
where $V_+(u):=\max\{V(u),0\}$. Immediately, one sees that $\sigma^2:=\E[\sum_{|u|=1}V(u)^2 e^{-V(u)}]\in(0,\infty)$. We take $\sigma=\sqrt{\sigma^2}$.

The criteria of recurrence/transience for random walks on trees is established by Lyons and Pemantle \cite{LyonPema}, which shows that the walk $(X_n)_{n\geq0}$ is recurrent under \eqref{boundarycase}. Further, Faraud \cite{Faraud} proved that the walk is null recurrent under \eqref{boundarycase} and \eqref{Integrability}. Hu and Shi studied the walk under these assumptions, and showed in \cite{HS07} that if $\T$ is regular tree, then a.s., asymptotically, $\max_{0\leq i\leq n}|X_i|=\Theta((\log n)^3)$.
%\begin{equation}
%0<\liminf_{n\rightarrow\infty}\frac{\max_{0\leq i\leq n}|X_i|}{(\log n)^3}\leq \limsup_{n\rightarrow\infty}\frac{\max_{0\leq i\leq n}|X_i|}{(\log n)^3}<\infty.
%\end{equation}
So the walk is called in a regime of \textbf{slow movement}. Later, under \eqref{boundarycase} and \eqref{Integrability}, Faraud, Hu and Shi proved in \cite{FHS11}, on the survival of $\T$, a.s. ,
\begin{equation}
\lim_{n\rightarrow\infty}\frac{\max_{0\leq i\leq n}|X_i|}{(\log n)^3}=Cst.
\end{equation}
Further, Hu and Shi obtained in \cite{HuShi16} that $\frac{|X_n|}{(\log n)^2}$ converges weakly. The spread and the range of this walk have been studied in \cite{AD14} and \cite{AC18}. In this paper, we study the heavy range of the walk in this slow regime.

Define the edge local time for the edge $(x^*,x)$ as follows
\[
\eL_x(n):=\sum_{k=0}^n\ind{X_{k-1}=x^*, X_k=x}, \forall n\geq1.
\]
 Let $\tau_0:=0$ and
\[
\tau_n:=\inf\{k>\tau_{n-1}: X_{k-1}=\rho^*, X_k=\rho\}, \forall n\geq1.
\]
Then $\eL_\rho(\tau_n)=n$. It can be seen from \cite{HuShi16+} that $\max_{x\in\T}\eL_x(\tau_n)$ is of order $n$ in probability. %Moreover, it is known in \cite{HS16} that 
%\begin{equation}\label{cvgtn}
%\frac{\tau_n}{n\log n}\xrightarrow[n\rightarrow\infty]{\textrm{ in } \p^*} \frac{4}{\sigma^2}D_\infty.
%\end{equation}
For any $\theta\in(0,1)$, define
\[
R^{\geqslant n^\theta}(\tau_n):=\sum_{x\in\T}\ind{\eL_x(\tau_n)\geq n^\theta}.
\]
We are interested in this so-called heavy range, which was first considered by Andreoletti and Diel \cite{AD18+}. They show that in any recurrent case, under $\p^*$, in probability, $R^{\geqslant n^\theta}(\tau_n)=n^{\xi_\theta+o(1)}$ where $\xi_\theta>0$ is a constant depending on the regimes and on $\theta$. In the sub-diffusive and diffusive regimes, our upcoming paper with de Raph\'elis \cite{CdR19+} will prove the convergence in law of $\frac{R^{\geqslant n^\theta}(\tau_n)}{n^{\xi_\theta}}$ under the annealed and quenched probability.  In the slow movement regime, it is given in \cite{AD18+} that $\xi_\theta=1-\theta$. We obtain the convergence in probability of $\frac{R^{\geqslant n^\theta}(\tau_n)}{n^{1-\theta}}$ under $\p^*$ in this paper.

Let us state the main result as follows.
\begin{theo}\label{main}
For any $\theta\in(0,1)$, the following convergence in probability holds:
\begin{equation}
\frac{R^{\geqslant n^\theta}(\tau_n)}{n^{1-\theta}}\xrightarrow[n\rightarrow\infty]{\textrm{ in }\p^*} \Lambda(\theta) D_\infty,
\end{equation}
where $D_\infty>0$ is the $\P^*$-a.s. limit of the derivative martingale $(D_n:=\sum_{|x|=n}V(x)e^{-V(x)})_{n\geq0}$
and $\Lambda(\theta)$ is a positive real number whose value is given in \eqref{cst} later.
\end{theo}

\begin{rem}
Note that for $\theta=0$, the total range up to $\tau_n$ has been studied in \cite{AC18} and is of order $n$ in probability $\p^*$.
\end{rem}
Under \eqref{boundarycase}, $D_n$ is a martingale with respect to the natural filtration $\{\mathcal{F}_n;  n\geq0\}$ with $\mathcal{F}_n:=\sigma(V(u); |u|\leq n)$. Under \eqref{Integrability}, it converges a.s. towards some non-degenerate limit according to Theorem of \cite{BigKyp}. Moreover, $\P(D_\infty>0)=\P(\T\textrm{ survives})$ under \eqref{Integrability}.

\subsection{Sketch of proofs and organisation of the paper}
Write $\eL_x^{(n)}$ for $\eL_x(\tau_n)$. In addition, up to the $n$-th return to $\rho^*$, define the number of excursions visiting $x$ by
\[
E_x^{(n)}:=\sum_{k=1}^n\ind{\exists j\in (\tau_{k-1}, \tau_k], X_j=x}, \forall n\geq1.
\]
Then define
\[
R^{\geqslant n^\theta}(\tau_n, j):=\sum_{x\in\T}\ind{\eL_x(\tau_n)\geq n^\theta, E_x^{(n)}=j}, \forall 1\leq j\leq n.
\]
It is clear that
\[
R^{\geqslant n^\theta}(\tau_n)=\sum_{j=1}^n R^{\geqslant n^\theta}(\tau_n, j)= \sum_{j= 2}^nR^{\geqslant n^\theta}(\tau_n, j)+ R^{\geqslant n^\theta}(\tau_n, 1).
\]
It is known in \cite{AD18+} that under $\p^*$, 
\[
\frac{\log R^{\geqslant n^\theta}(\tau_n)}{\log n}\xrightarrow[n\to\infty]{\p^*}1-\theta.
\]
We are going to treat $ \sum_{j=2}^nR^{\geqslant n^\theta}(\tau_n, j)$ and $ R^{\geqslant n^\theta}(\tau_n, 1)$ separately and show the convergences in probability of 
\[
\frac{1}{n^{1-\theta}}\sum_{j= 2}^nR^{\geqslant n^\theta}(\tau_n, j) \textrm{ and } \frac{1}{n^{1-\theta}}R^{\geqslant n^\theta}(\tau_n, 1),
\]
under the annealed probability $\p^*$. In fact, we have the following results.

\begin{prop}\label{mEx}
For any $\theta\in(0,1)$, the following convergence in probability holds:
\begin{equation}\label{Mex}
\frac{1}{n^{1-\theta}}\sum_{j= 2}^nR^{\geqslant n^\theta}(\tau_n, j) \xrightarrow[n\to\infty]{\p^*} \Lambda_0(\theta) D_\infty,
\end{equation}
where
\begin{equation}\label{Mexconstant}
\Lambda_0(\theta):=\frac{\sqrt{2}}{\sqrt{\pi}\sigma^2}\int_0^\infty\mathcal{C}_0(\frac{\theta}{\sqrt{u}},\frac{1-\theta}{\sqrt{u}})\frac{\d u}{u}%\cb_\Ren\cb_+\int_0^\infty \P(\overline{R}_1\leq \frac{1}{\sigma\sqrt{u}}\vert R_1=\frac{1-\theta}{\sigma\sqrt{u}})\varphi(\frac{1-\theta}{\sigma\sqrt{u}})\frac{\d u}{\sigma u}\in(0,\infty)
\end{equation}
with $\mathcal{C}_0$ defined in \eqref{defconstantC}.%$(R_t)_{t\geq0}$ a 3-dim Bessel process started from $0$, $\overline{R}_1:=\sup_{t\in[0,1]}R_t$ and $\varphi(t):=te^{-t^2/2}\ind{t\ge 0}$.
\end{prop}

\begin{prop}\label{1Ex}
For any $\theta\in(0,1)$, the following convergence in probability holds:
\begin{equation}
\frac{1}{n^{1-\theta}}R^{\geqslant n^\theta}(\tau_n, 1)\xrightarrow[n\to\infty]{\p^*} \Lambda_1(\theta) D_\infty,
\end{equation}
where
\begin{equation}
\Lambda_1(\theta):=\cb_\Ren\int_0^\infty \mathcal{G}(\frac{1}{\sqrt{s}},\frac{\theta}{\sqrt{s}})\frac{\d s}{s},
\end{equation}
with $\mathcal{G}(a,b)$ defined in \eqref{defconstantg} and $\cb_\Ren$ defined in \eqref{cvgrenewalf}.
\end{prop}

Theorem \ref{main} follows directly from Propositions \ref{mEx} and \ref{1Ex} with 
\begin{equation}\label{cst}
\Lambda(\theta)=\Lambda_0(\theta)+\Lambda_1(\theta).
\end{equation}
%\begin{align}\label{cst}
%\Lambda(\theta)=\Lambda_0(\theta)+\Lambda_1(\theta)=&\cb_\Ren\int_0^\infty \P(\overline{R}_1\leq \frac{1}{\sigma\sqrt{u}}\vert R_1=\frac{1-\theta}{\sigma\sqrt{u}})\varphi(\frac{1-\theta}{\sigma\sqrt{u}})\frac{\d u}{\sigma^2 u}\times(\int_0^\infty\frac{\mathcal{C}_{t^{-1/2}, t^{-1/2}}}{t}dt+1)\nonumber\\
%=&\Lambda_0(\theta)\times(\int_0^\infty\frac{\mathcal{C}_{t^{-1/2}, t^{-1/2}}}{t}dt+1),
%\end{align}
%where the finiteness of $\int_0^\infty\frac{\mathcal{C}_{t^{-1/2}, t^{-1/2}}}{t}dt$ has been verified in Lemma A.1 of  \cite{AC18}. 
Here the finiteness of $\Lambda_0(\theta)$ can be checked immediately as $\Lambda_0(\theta)\leq \frac{\sqrt2}{\sqrt{\pi}\sigma^2}\int_0^\infty \varphi(\frac{1-\theta}{\sigma\sqrt{u}})\frac{\d u}{\sigma u}<\infty$. For $\Lambda_1(\theta)$, by change of variables $r=s(1-u)$ and $t=su$, one sees that
\begin{align*}
\Lambda_1(\theta)=&\cb_\Ren\frac{\cb_-}{\sigma}\int_0^\infty \int_0^1 \mathcal{C}_{\frac{1}{\sqrt{su}},\frac{1}{\sqrt{su}}}\mathcal{C}_0(\frac{1-\theta}{\sqrt{s(1-u)}},\frac{\theta}{\sqrt{s(1-u)}})\frac{\d u}{(1-u)u} \frac{\d s}{s}\\
=&\cb_\Ren\frac{\cb_-}{\sigma}\int_0^\infty\mathcal{C}_0(\frac{1-\theta}{\sqrt{r}}, \frac{\theta}{\sqrt{r}})\frac{\d r}{r}\int_0^\infty \mathcal{C}_{\frac{1}{\sqrt{t}},\frac{1}{\sqrt{t}}}\frac{\d t}{t}=\frac{\cb_\Ren\cb_-\sqrt{\pi}\sigma}{\sqrt{2}}\Lambda_0(1-\theta)\int_0^\infty \mathcal{C}_{\frac{1}{\sqrt{t}},\frac{1}{\sqrt{t}}}\frac{\d t}{t},
\end{align*}
where the finiteness of $\int_0^\infty\frac{\mathcal{C}_{t^{-1/2}, t^{-1/2}}}{t}dt$ has been verified in Lemma A.1 of  \cite{AC18}. 

Let us do some basic calculations here. For any $x\in\T\cup\{\rho^*\}$, let $T_x$ be the first hitting time at $x$:
\[
T_x:=\inf\{k\ge 0: X_k=x\}.
\]
Then, it is known that
\begin{align}
a_x:=&\p^\mathcal{E}_\rho(T_x< T_{\rho^*})=\frac{1}{\sum_{y\in[\![\rho, x]\!]}e^{V(y)}}=\frac{e^{-V(x)}}{H_x},\\
b_x:=&\p^\mathcal{E}_{x^*}(T_x< T_{\rho^*})=1-\frac{1}{H_x},\\
\textrm{ where }H_x:=&\sum_{y\in[\![\rho, x]\!]}e^{V(y)-V(x)}.\nonumber
\end{align}
As a consequence, for any fixed $x\in\T$,
\[
\p^\en_\rho(\overline{L}_x(\tau_1)=0)=1-a_x\textrm{ and }\p^\en_\rho(\overline{L}_x(\tau_1)\geq k)=a_x b_x^{k-1}, \forall k\ge1.
\]
Then by Markov property, under $\p^\en_\rho$, $(\overline{L}_x(\tau_{n+1})-\overline{L}_x(\tau_{n}))_{n\geq1}$ are i.i.d. random variables distributed as $\overline{L}_x(\tau_1)$. Moreover, $E_x^{(n)}$ is a Binomial random variable with parameters $n$ and $a_x$. Let
\[
\MV(x):=\max_{y\in[\![\rho, x]\!]} V(y)\textrm{ and }\mV(x):=\min_{y\in[\![\rho, x]\!]} V(y), \forall x\in\T.
\]
Then $a_x\leq e^{-\MV(x)}$, $H_x\leq e^{\MV(x)-V(x)}$.

To get $\sum_{j= 2}^nR^{\geqslant n^\theta}(\tau_n, j)$, we need to consider the individuals $x\in\T$ such that $\{\eL_x(\tau_n)\geq n^\theta, E_x^{(n)}\geq2\}$. As $a_x\leq e^{-\MV(x)}$, the individuals with $\MV(x)\gtrsim \log n$ would be visited in at most one excursion with high probability. We thus take $\{x\in\T: \MV(x)\lesssim \log n\}$. On the other hand, $\e^\en[\eL_x(\tau_n)]=ne^{-V(x)}$. So, the event $\{V(x)\lesssim (1-\theta)\log n\}$ involves also. Actually, we could verify that with high probability under $\p^*$,
\[
\sum_{j= 2}^nR^{\geqslant n^\theta}(\tau_n, j)\approx \sum_{x\in\T}\ind{\MV(x)\lesssim \log n, V(x)\lesssim (1-\theta)\log n}.
\]
The asymptotic of the latter will be treated in Lemma \ref{largepart}.

To get $R^{\geqslant n^\theta}(\tau_n, 1)$, we are going to compare it with its quenched expectation. We see that 
\[
\e^\en[R^{\geqslant n^\theta}(\tau_n, 1)]=\sum_{x\in\T} na_x(1-a_x)^{n-1}b_x^{n^\theta-1}\approx n^{1-\theta}\sum_{x\in\T}e^{-V(x)}(\frac{n^\theta}{H_x}e^{-\frac{n^\theta}{H_x}})
\]
as we only need to count the individuals with $\{\MV(x)\gtrsim \log n\}$ so that they are visited only by one excursion with high probability. Here we take $\{\MV(x)-V(x)\approx \theta \log n\}$ so that $H_x=\Theta(n^\theta)$ as $H_x$ and $e^{\MV(x)-V(x)}$ are comparable. In additional, it is known in \cite{HuShi16} that up to $\tau_n$, with high probability, the walker has not visited the stopping line $\{x\in\T: \max_{\rho\leq y<x}H_y<\gamma_n\leq H_x\}$ with $\gamma_n=\frac{n}{(\log n)^\gamma}$ for any $\gamma>0$. By bounding the quenched variance of $R^{\geqslant n^\theta}(\tau_n, 1)$, we could verify that with high probability,
\[
R^{\geqslant n^\theta}(\tau_n, 1)\approx n^{1-\theta}\sum_{x\in\T}e^{-V(x)}(\frac{n^\theta}{H_x}e^{-\frac{n^\theta}{H_x}})\ind{\MV(x)\gtrsim \log n, \MV(x)-V(x)\approx \theta \log n, \max_{\rho\leq y\leq x}H_y<r_n}.
\]
The asymptotic of the latter will be given in Lemma \ref{SHnorm1ex}.

The rest of the paper is organised as follows. In section \ref{facts}, we state some basic facts on the branching random walk and the biased random walk. In section \ref{Exs}, we study $ \sum_{j=2}^nR^{\geqslant n^\theta}(\tau_n, j)$ and prove Proposition \ref{mEx} by choosing the suitable environment. In section \ref{Ex1}, we prove Proposition \ref{1Ex}. Next, Section \ref{SHnorm} is devoted to proving the generalised Seneta-Heyde norming results: Lemmas \ref{SHnorm1ex} and \ref{largepart}, by applying the new method introduced by \cite{BM19}. In Section \ref{lems}, we complete the proofs of the technical lemmas. 

In this paper, we use $(c_i)_{i\geq0}$ and $(C_i)_{i\ge0}$ for positive constants which may change from line to line. And we write $f(n)\sim g(n)$ when $\frac{f(n)}{g(n)}\rightarrow 1$ as $n\to\infty$.
%\section{Rough ideas}

\section{Preliminary results}\label{facts}
In this section, we state some facts and lemmas which will be used later. 

\subsection{Many-to-One Lemma}
Recall that $\P$ is the law of the branching random walk $(V(u),u\in\T)$ started from $V(\rho)=0$. Let $\P_a((V(u), u\in\T)\in\cdot)=\P((a+V(u), u\in\T)\in \cdot)$ for any $a\in\R$. Let $E_a$ be the corresponding expectation. Then the following lemma holds because of \eqref{boundarycase}.
\begin{lem}[Many-to-One]
For any $n\geq1$, $a\in\R$ and any measurable function $f: \mathbb{R}^n\rightarrow \mathbb{R}_+$, we have
\begin{equation}
\E_a\left[\sum_{|u|=n}e^{-V(u)}f(V(u_1), \cdots, V(u_n))\right]=e^{-a}\E\left[f(S_1+a,\cdots,S_n+a)\right],
\end{equation}
where $(S_n)_{n\geq0}$ is one dimensional centred random walk with i.i.d. increments and $S_0=0$. %Moreover, $\E[S_1]=0$ and $\E[S_1^2]=\sigma^2\in(0,\infty)$.
\end{lem}

Moreover, by \eqref{Integrability}, $\E[S_1^2]=\sigma^2\in(0,\infty)$ and
\begin{equation}\label{expmomS}
\E[e^{-\delta_0 S_1}]+\E[e^{(1+\delta_0)S_1}]<\infty.
\end{equation}
For any $n\geq0$, let $\MS_n:=\max_{0\leq k\leq n}S_k$ and $\mS_n:=\min_{0\leq k\leq n}S_k$. More estimates and rescaling results on the random walk $(S_n)_{n\in\N}$ can be found in Appendix \ref{A2}.

\subsection{Lyons' change of measure}
Define the additive martingale with respect to the natural filtration $\{\calF_n\}_{n\geq0}$ by
\[
W_n:=\sum_{|u|=1}e^{-V(u)},\forall n\geq0.
\]
Under \eqref{boundarycase}, this is a non-negative martingale  and it converges $\P$-a.s. to zero according to \cite{Lyons97}. By Kolmogorov extension theorem, for any $\R$, we can define a probability measure $\Q_a$ on $\calF_\infty:=\lor_{n\geq0}\calF_n$ such that
\[
\frac{d\Q_a}{d\P_a}\vert_{\mathcal{F}_n}:=e^a\sum_{|u|=n}e^{-V(u)}, \forall n\geq0.
\]
Let $\E_{\Q_a}$ denote the corresponding expectation and write $\Q$ for $\Q_0$.

Let us introduce a probability measure $\widehat{\Q}_a$ on the space of marked branching random walks so that its marginal distribution is exactly $\Q_a$. Recall that the reproduction law of the branching random walk $(V(x), x\in\T)$ is given by the point process $\calL=\{V(x), |x|=1\}$. Let $\widehat{\calL}$ be the point process having Radon-Nykodim derivative $\sum_{z\in\calL}e^{-z}$ with respect to the law of $\calL$.
%, i.e.,
%\[
%\widehat{Q}(\widehat{\calL}\in\cdot)=\E\left[\sum_{|x|=1}e^{-V(x)}\ind{(V(x), |x|=1)\in\cdot}\right].
%\]
 We start with $w_0$ the root, located at $V(w_0)=0$. At time $1$, it dies and reproduces a random number of individuals whose displacements with respect to $V(w_0)$, viewed as a point process, are distributed as $\widehat{\calL}$. All children of $w_0$ form the first generation, among which we choose $x$ to be $w_1$ with probability proportional to $e^{-V(x)}$. Then recursively, at time $n+1$, the individuals of the $n$-th generation die and reproduce independently their children according to the law of $\calL$ , except $w_n$ which gives birth to its children according to $\widehat{\calL}$. The $w_{n+1}$ is selected among the children of $w_n$ with probability proportional to $e^{-V(u)}$ for each child $u$ of $w_n$. This construction gives us a branching random walk with a marked ray $(w_n)_{n\geq0}$, which is called the spine. The law of this marked branching random walk $(V(x), x\in\T; (w_n)_{n\geq0})$ is denoted by $\widehat{\Q}_0$. Again, $\widehat{\Q}_a$ denotes the law of $(a+V(x), x\in\T; (w_n)_{n\geq0})$ under $\widehat{\Q}_0$. We use $\E_{\widehat{\Q}_a}$ to represent the corresponding expectation and use $\widehat{\Q}$ instead of $\widehat{\Q}_0$ for brevity.
 
It is known that the marginal law of $\widehat{\Q}_a$ on the branching random walk is the same as $\Q_a$ defined above. We also state the following proposition from \cite{Lyons97}, which gives some properties of $\widehat{\Q}_a$.

\begin{prop}\label{spinedecomp}
\begin{itemize}
\item[(i)] \[
\widehat{\Q}_a\left((V(w_0),\cdots, V(w_n))\in \cdot\right)=\P\left((a+S_0,\cdots, a+S_n)\in\cdot\right).
\]
\item[(ii)] For any $|u|=n$,
\[
\widehat{\Q}_a\left(w_n=u\vert \mathcal{F}_n\right)=\frac{e^{-V(u)}}{W_n}.
\]
\end{itemize}
\end{prop}

For the marked branching random walk $(V(x), x\in\T; (w_n)_{n\geq0})$, let $\Omega(w_j)=\{u\in\T: u^*=w_{j-1}, u\neq w_j\}$ be the collection of brothers of $w_j$ for any $j\geq1$. Let $\mathscr{G}$ be the sigma-field containing all information along the spine, that is,
\[
\mathscr{G}:=\sigma\{(w_k, V(w_k))_{k\geq0}, (u, V(u))_{u\in\cup_{k\geq0}\Omega(w_k)}\}.
\]

\subsection{Quenched probability for edge local times}
Recall that under $\p^\en_\rho$, $(\overline{L}_x(\tau_{n+1})-\overline{L}_x(\tau_{n}))_{n\geq1}$ are i.i.d. random variables distributed as $\overline{L}_x(\tau_1)$. Observe also that $E_x^{(n)}=\sum_{k=1}^n\ind{\overline{L}_x(\tau_{k})-\overline{L}_x(\tau_{k-1})\geq1}$. Let us state the following lemma on the joint distribution of $(\eL_x(\tau_n), E_x^{(n)})$ under the quenched probability.
%This means that Lemma \ref{sumGeo} can be applied for $\overline{L}_x(\tau_n)$ under $\p^\en_\rho$. The proof of Lemma \ref{sumGeo} will be postponed in Section \ref{tools}.

\begin{lem}\label{sumGeo}
Let $a, b\in(0,1)$. Suppose that $(\zeta_i)_{i\geq1}$ are i.i.d. random variables taking values in $\mathbb{N}$ such that
\[
\P(\zeta_1=0)=1-a,\textrm{ and } \P(\zeta_1\geq k)=ab^{k-1}, \forall k\geq1.
\]
\begin{enumerate}
\item If $n^\theta(1-b)\geq (1+\eta)na$ with some $\eta>0$, then there exists $c_\eta>0$ such that for any $n\geq 1$,
\begin{equation}\label{sumGeolarge}
\P\left(\sum_{i=1}^n\zeta_i\geq n^\theta\right)\leq 2 na e^{-c_\eta n^\theta(1-b)},
\end{equation}
and
\begin{equation}\label{sumGeolarge+}
\P\left(\sum_{i=1}^n\zeta_i\geq n^\theta; \sum_{i=1}^n\ind{\zeta_i\geq1}\geq 2\right)\leq 2 (na)^2 e^{-c_\eta n^\theta(1-b)}.
\end{equation}
\item For $A>0$, $0<\lambda<1$ and for any $n\geq1$,
\begin{equation}\label{sumGeosmall}
\P\left(\sum_{i=1}^n\zeta_i\leq A\right)\leq e^{-\lambda(\frac{na}{1+\lambda}-(1-b)A)}.
\end{equation}
\end{enumerate}

\end{lem}
The proof of Lemma \ref{sumGeo} will be postponed in Appendix \ref{A1}.

\section{Proof of Proposition \ref{mEx}}\label{Exs}
In this section, we study $\sum_{j=2}^n R^{\geqslant n^\theta}(\tau_n, j)$ and prove Proposition \ref{mEx}.

First, it is proved in \cite{FHS11} that $\max_{1\leq i\leq \tau_n}|X_i|=O((\log n)^3)$, $\p^*$-a.s. So, 
\[
\sum_{j\geq 2}R^{\geqslant n^\theta}(\tau_n, j)=\sum_{|x|\leq c_0(\log n)^3}\ind{\overline{L}_x(\tau_n)\geq n^\theta, E_x^{(n)}\geq 2}+o_n(1), \ \p^*\textrm{-a.s.},
\]
with some large and fixed constant $c_0>0$. On the other hand, it is known that $\P^*$-a.s.,
\[
0\geq \inf_{u\in\T}V(u)>-\infty.
\]
So, we only need to consider $\sum_{|x|\leq c_0(\log n)^3}\ind{\overline{L}_x(\tau_n)\geq n^\theta, E_x^{(n)}\geq 2}\ind{\mV(x)\geq-\alpha}$ for any fixed $\alpha>0$. Now for any $a,b\in\mathbb{R}$, let 
\[
A_n(a, b):=\{x\in\T: \MV(x)-V(x)\leq \theta \log n+a, V(x)\leq(1-\theta)\log n+b\}, \forall n\geq1,
\]
and
\[
A_n^+(a,b):=\{x\in\T: \MV(x)\leq \log n+a, V(x)\leq (1-\theta)\log n+b\}, \forall n\geq1.
\]
Then, we stress that for any $\alpha>0$, $b>0$, $a_n=a\log\log n$ with $a>3$,
\begin{multline*}
\sum_{|x|\leq c_0(\log n)^3}\ind{\mV(x)\ge-\alpha}\ind{x\in A_n(-a_n, -b)}+o_{\p}(n^{1-\theta})\leq 
\sum_{|x|\leq c_0(\log n)^3}\ind{\mV(x)\geq-\alpha}\ind{\overline{L}_x(\tau_n)\geq n^\theta, E_x^{(n)}\geq 2}\\
\leq \sum_{|x|\leq c_0(\log n)^3}\ind{\mV(x)\geq-\alpha}\ind{x\in A_n^+(a_n,b)}+o_{\p}(n^{1-\theta}),
\end{multline*}
because of the following lemma. 
\begin{lem}\label{smallpart}
Let $b>0$, $\alpha>0$. For $a_n=a\log\log n$ with $a>3$, we have
\begin{equation}\label{smallpartbadwalk}
\frac{1}{n^{1-\theta}}\sum_{|x|\leq c_0(\log n)^3}\ind{x\in A_n(-a_n, -b)}\ind{\overline{L}_x(\tau_n)<n^\theta\textrm{ or } E_x^{(n)}\leq 1}\xrightarrow[n\rightarrow\infty]{\textrm{ in }\p} 0,
\end{equation}
and
\begin{equation}\label{smallpartbadV}
\frac{1}{n^{1-\theta}}\sum_{|x|\leq c_0(\log n)^3}\ind{\mV(x)\ge-\alpha}\ind{x\notin A_n^+(a_n,b)}\ind{\overline{L}_x(\tau_n)\geq n^\theta, E_x^{(n)}\geq 2}\xrightarrow[n\rightarrow\infty]{\textrm{ in }\p}0.
\end{equation}
\end{lem}
It remains to study $\sum_{|x|\leq c_0(\log n)^3}\ind{x\in A_n(-a_n, -b)}$ and $\sum_{|x|\leq c_0(\log n)^3}\ind{x\in A_n^+(a_n, b)}$, which is done in the next lemma.
\begin{lem}\label{largepart}
Let $b>0$. For $a_n=o(\log n)$, we have
\begin{equation}\label{largepartlow}
\frac{\sum_{|x|\leq c_0(\log n)^3}\ind{x\in A_n(-a_n,-b)}}{n^{1-\theta}}\xrightarrow[n\rightarrow\infty]{\textrm{ in }\P^*} D_\infty \Lambda_0(\theta)e^{-b},
\end{equation}
and 
\begin{equation}\label{largepartup}
\frac{\sum_{|x|\leq c_0(\log n)^3}\ind{x\in A_n^+(a_n,b)}}{n^{1-\theta}}\xrightarrow[n\rightarrow\infty]{\textrm{ in }\P^*} D_\infty \Lambda_0(\theta)e^b.
\end{equation}
\end{lem}
The proof of Lemma \ref{smallpart} will be given later in Section \ref{lems}, and the proof of Lemma \ref{largepart} will be in Section \ref{SHnorm}. Now we are ready to prove Proposition \ref{mEx}.
\begin{proof}[Proof of Proposition \ref{mEx}]
Recall that $D_\infty>0$, $\p^*$-a.s. We only need to show that for any $\delta\in(0,1)$, as $n\rightarrow\infty$,
\[
\p^*\left(\frac{\sum_{j\geq 2}R^{\geqslant n^\theta}(\tau_n, j)}{n^{1-\theta}}\geq (1+\delta)\Lambda_0(\theta) D_\infty\textrm{ or }\leq (1-\delta)\Lambda_0(\theta) D_\infty\right)\rightarrow 0.
\]
Observe that for any $\alpha>0$ and $\beta\in(0,1)$,
\begin{align*}
&\p^*\left(\frac{\sum_{j\geq 2}R^{\geqslant n^\theta}(\tau_n, j)}{n^{1-\theta}}\geq (1+\delta)\Lambda_0(\theta) D_\infty\textrm{ or }\leq (1-\delta)\Lambda_0(\theta) D_\infty\right)\\
\leq &\P(\inf_{x\in\T}V(x)<-\alpha)+\p^*\left(\max_{1\leq i\leq \tau_n}|X_i|>c_0 (\log n)^3\right)+\P^*(D_\infty<\beta)\\
&+\p^*\left( \sum_{|x|\leq c_0(\log n)^3}\ind{\mV(x)\geq-\alpha}\ind{\overline{L}_x(\tau_n)\geq n^\theta, E_x^{(n)}\geq 2}\geq (1+\delta)n^{1-\theta}\Lambda_0(\theta) D_\infty\textrm{ or }\leq (1-\delta)n^{1-\theta}\Lambda_0(\theta) D_\infty; D_\infty\ge\beta\right)
\end{align*}
It is known (see \cite{Aid13}) that for any $\alpha>0$, $\P(\inf_{x\in\T}V(x)<-\alpha)\leq e^{-\alpha}$. Note also that $\P^*(D_\infty<\beta)=o_\beta(1)$ as $\beta\downarrow 0$. Therefore,
\begin{align}\label{Obs}
&\p^*\left(\frac{\sum_{j\geq 2}R^{\geqslant n^\theta}(\tau_n, j)}{n^{1-\theta}}\geq (1+\delta)\Lambda D_\infty\textrm{ or }\leq (1-\delta)\Lambda D_\infty\right)\\
\leq & c_2e^{-\alpha}+o_n(1)+o_\beta(1)+\p^*\left( \sum_{|x|\leq c_0(\log n)^3}\ind{\mV(x)\geq-\alpha}\ind{\overline{L}_x(\tau_n)\geq n^\theta, E_x^{(n)}\geq 2}\geq (1+\delta)n^{1-\theta}\Lambda_0(\theta) D_\infty; D_\infty\ge\beta\right)\nonumber\\
&+\p^*\left( \sum_{|x|\leq c_0(\log n)^3}\ind{\mV(x)\geq-\alpha}\ind{\overline{L}_x(\tau_n)\geq n^\theta, E_x^{(n)}\geq 2}\leq (1-\delta)n^{1-\theta}\Lambda_0(\theta) D_\infty; D_\infty\ge\beta, \inf_{x\in\T} V(x)\geq-\alpha\right).\nonumber
%\p^*\left(\frac{1}{n^{1-\theta}}\sum_{|x|\leq c_0(\log n)^3}\ind{x\in A_n(-a_n, -b)}\ind{\overline{L}_x(\tau_n)<n^\theta\textrm{ or } E_x^{(n)}\leq 1}\geq \frac{\delta\Lambda_1\beta}{2}\right)\\
%&+\p^*\left(\frac{1}{n^{1-\theta}}\sum_{|x|\leq c_0(\log n)^3}\ind{x\notin A_n^+(a_n,b)}\ind{\mV(x)\ge-\alpha}\ind{\overline{L}_x(\tau_n)\geq n^\theta, E_x^{(n)}\geq 2}\geq \frac{\delta\Lambda_1 \beta}{2}\right)
\end{align}
On the one hand, for any $b>0$ and $a_n=a\log\log n$ with $a>3$, one has
\begin{align*}
&\p^*\left( \sum_{|x|\leq c_0(\log n)^3}\ind{\mV(x)\geq-\alpha}\ind{\overline{L}_x(\tau_n)\geq n^\theta, E_x^{(n)}\geq 2}\geq (1+\delta)n^{1-\theta}\Lambda_0(\theta) D_\infty; D_\infty>\beta\right)\\
\leq &\p^*\left(\frac{1}{n^{1-\theta}}\sum_{|x|\leq c_0(\log n)^3}\ind{x\notin A_n^+(a_n,b)}\ind{\mV(x)\ge-\alpha}\ind{\overline{L}_x(\tau_n)\geq n^\theta, E_x^{(n)}\geq 2}\geq \frac{\delta\Lambda_0(\theta) \beta}{2}\right)\\
&+\p^*\left( \sum_{|x|\leq c_0(\log n)^3}\ind{\mV(x)\geq-\alpha}\ind{x\in A_n^+(a_n,b)}\ind{\mV(x)\ge-\alpha}\ind{\overline{L}_x(\tau_n)\geq n^\theta, E_x^{(n)}\geq 2}\geq (1+\delta/2)n^{1-\theta}\Lambda_0(\theta) D_\infty\right)\\
\leq&o_n(1)+\p^*\left( \sum_{|x|\leq c_0(\log n)^3}\ind{x\in A_n^+(a_n,b)}\geq (1+\delta/2)n^{1-\theta}\Lambda_0(\theta) D_\infty\right),
\end{align*}
where the last line follows from \eqref{smallpartbadV}. For the second term on the righthand side, taking $b>0$ small so that $e^b<1+\delta/2$ and using \eqref{largepartup} yields that
\[
\p^*\left( \sum_{|x|\leq c_0(\log n)^3}\ind{x\in A_n^+(a_n,b)}\geq (1+\delta/2)n^{1-\theta}\Lambda_0(\theta) D_\infty\right)\rightarrow 0,
\]
as $n$ goes to infinity. On the other hand, observe that
\begin{align*}
&\p^*\left( \sum_{|x|\leq c_0(\log n)^3}\ind{\mV(x)\geq-\alpha}\ind{\overline{L}_x(\tau_n)\geq n^\theta, E_x^{(n)}\geq 2}\leq (1-\delta)n^{1-\theta}\Lambda_0(\theta) D_\infty; D_\infty\ge\beta, \inf_{x\in\T}V(x)\geq-\alpha\right)\\
\leq &\p^*\left(\frac{1}{n^{1-\theta}}\sum_{|x|\leq c_0(\log n)^3}\ind{x\in A_n(-a_n, -b)}\ind{\overline{L}_x(\tau_n)<n^\theta\textrm{ or } E_x^{(n)}\leq 1}\geq \frac{\delta\Lambda_0(\theta)\beta}{2}\right)\\
&+\p^*\left(\sum_{|x|\leq c_0(\log n)^3}\ind{x\in A_n(-a_n, -b)}\leq (1-\delta/2)\Lambda_0(\theta)D_\infty\right)\\
=&o_n(1)
\end{align*}
by \eqref{smallpartbadwalk} and \eqref{largepartlow} with $b>0$ small enough so that $e^{-b}>1-\delta/2$. Going back to \eqref{Obs}, one sees that
\[
\p^*\left(\frac{\sum_{j\geq 2}R^{\geqslant n^\theta}(\tau_n, j)}{n^{1-\theta}}\geq (1+\delta)\Lambda_0(\theta) D_\infty\textrm{ or }\leq (1-\delta)\Lambda_0(\theta) D_\infty\right)\leq c_2e^{-\alpha}+o_\beta(1)+o_n(1).
\]
Letting $n\rightarrow\infty$ then $\alpha\uparrow\infty$ and $\beta\downarrow0$ concludes \eqref{Mex}.
\end{proof}

\section{Proof of Proposition \ref{1Ex}}\label{Ex1}
This section is devoted to proving Proposition \ref{1Ex}. Similarly as above, we have $\p^*$-a.s.,
\[
R_{n^\theta}(\tau_n, 1)=\sum_{\ell=1}^{c_0(\log n)^3}\sum_{|x|=\ell}\ind{\eL_x(\tau_n)\geq n^\theta, E_x^{(n)}=1}+o_n(1).
\]
For $a_n=a\log\log n$ with $a>3$, set
\[
\B_n^\pm:=\{x\in\T:  \MV(x)\geq \log n\pm a_n\}, \textrm{ and } \D_n:=\{x\in\T:\MV(x)-V(x)\in[\theta\log n-a_n,\theta\log n+a_n]\}.
\]
We first show that with high probability, $R_{n^\theta}(\tau_n, 1)\approx\sum_{\ell=1}^{c_0(\log n)^3}\sum_{|x|=\ell}\ind{\eL_x(\tau_n)\geq n^\theta, E_x^{(n)}=1}\ind{z\in \B_n^-}\ind{z\in \D_n}$. This comes from the following lemma whose proof is stated in Section \ref{lems}.
\begin{lem}\label{smallpart1ex}
As $n\uparrow\infty$, we have
\begin{align}
\e\left[\sum_{\ell=1}^{c_0(\log n)^3}\sum_{|x|=\ell}\ind{\eL_x(\tau_n)\geq n^\theta, E_x^{(n)}=1}\ind{\MV(x)< \log n-a_n}\right]=&o(n^{1-\theta}),\\
\e\left[\sum_{\ell=1}^{c_0(\log n)^3}\sum_{|x|=\ell}\ind{\eL_x(\tau_n)\geq n^\theta, E_x^{(n)}=1}\ind{x\notin \D_n}\right]=&o(n^{1-\theta}).
\end{align}
\end{lem}
Here we introduce the stopping line
\[
\mathcal{L}_r:=\{x\in\T: \max_{y< x}H_y< r\leq H_x\}, \forall r>1.
\]
It is known that in \cite{HuShi16} that 
\[
\p\left(\exists k\leq \tau_n: X_k\in\mathcal{L}_n\right)\rightarrow 0.
\] 
This means that $\p(\{X_k, k\le \tau_n\}\subset\{x\in\T: x<\mathcal{L}_n\}\cup\{\rho^*\})\rightarrow1$. For any $r>1$, define
\[
\L_{r}:=\{x\in\T: \max_{y< x}H_y< r\}.
\]
So, we only need to study $\sum_{\ell=1}^{c_0(\log n)^3}\sum_{|x|=\ell}\ind{\eL_x(\tau_n)\geq n^\theta, E_x^{(n)}=1}\ind{z\in \B_n^-}\ind{z\in \D_n}\ind{z\in\L_n}$. In fact, only the generations of order $(\log n)^2$ should be counted and $\B_n^-$ can be replaced by $\B_n^+$, in view of the following lemma. 
\begin{lem}\label{smallpart1ex+}
As $ \varepsilon\downarrow0$, we have
\begin{align}
\limsup_{n\rightarrow\infty}\frac{1}{n^{1-\theta}}\e\left[\sum_{\ell=1}^{\varepsilon(\log n)^2}\sum_{|x|=\ell}\ind{\eL_x(\tau_n)\geq n^\theta, E_x^{(n)}=1}\ind{x\in \B_n^-}\ind{x\in \D_n, \mV(x)\geq-\alpha}\right]=&o_\varepsilon(1),\label{smallgeneration1ex+}\\
\limsup_{n\rightarrow\infty}\frac{1}{n^{1-\theta}}\e\left[\sum_{\ell=(\log n)^2/\varepsilon}^{c_0(\log n)^3}\sum_{|x|=\ell}\ind{\eL_x(\tau_n)\geq n^\theta, E_x^{(n)}=1}\ind{x\in \B_n^-}\ind{x\in \D_n, \mV(x)\geq-\alpha}\ind{x\in\L_n}\right]=&o_\varepsilon(1).\label{largegeneration1ex+}
\end{align}
For any $\varepsilon\in(0,1)$,
\begin{equation}\label{smallMV1ex+}
\frac{1}{n^{1-\theta}}\e\left[\sum_{\ell=\varepsilon (\log n)^2}^{(\log n)^2/\varepsilon}\sum_{|x|=\ell}\ind{\eL_x(\tau_n)\geq n^\theta, E_x^{(n)}=1}\ind{ \MV(x)\in[\log n-a_n, \log n+a_n], \mV(x)\geq-\alpha, x\in\D_n}\right]=o_n(1)
\end{equation}
\end{lem}
Instead of $\L_n$, we are going to use $\L_{r_n}$ with $r_n=\frac{n}{(\log n)^\gamma}$ to control the quenched variance of $\sum_{\ell=\varepsilon (\log n)^2}^{(\log n)^2/\varepsilon}\sum_{|x|=\ell}\ind{\eL_x(\tau_n)\geq n^\theta, E_x^{(n)}=1}\ind{z\in \B_n^+}\ind{z\in \D_n}\ind{z\in\L_{r_n}}$.
\begin{lem}\label{smallpart1exL}
For any $\varepsilon\in(0,1)$ fixed, $\alpha>0$ and for $\gamma_n=\frac{n}{(\log n)^{\gamma}}$ with fixed $\gamma>0$, we have
\begin{equation}\label{badvalley1ex}
\e\left[\sum_{\ell=\varepsilon (\log n)^2}^{(\log n)^2/\varepsilon}\sum_{|z|=\ell}\ind{\eL_z(\tau_n)\geq n^\theta, E_z^{(n)}=1}\ind{z\in \B^+_n}\ind{z\in\D_n}\ind{\mV(z)\geq-\alpha, \gamma_n\leq \max_{x\leq z} H_z< n}\right]=o(n^{1-\theta}).
\end{equation}
Let $\D_n^K:=\{x\in\T: \MV(x)-V(x)\in[\theta\log n-K, \theta\log n+K]\}$ with large constant $K\geq1$. Then, as $K\rightarrow\infty$,
\begin{equation}\label{badend1ex}
\e\left[\sum_{\ell=\varepsilon (\log n)^2}^{(\log n)^2/\varepsilon}\sum_{|z|=\ell}\ind{\eL_z(\tau_n)\geq n^\theta, E_z^{(n)}=1}\ind{z\in \B^+_n}\ind{z\in\D_n\setminus\D_n^K}\ind{\mV(z)\geq-\alpha, \max_{x\leq z} H_z< n}\right]=o_K(1)n^{1-\theta}.
\end{equation}
\end{lem}
Let
\[
\Xi_n(\ell, \B_n^+\cap \D_n\cap \L_{\gamma_n},\alpha):=\sum_{|x|=\ell}\ind{\eL_x(\tau_n)\geq n^\theta, E_x^{(n)}=1}\ind{x\in \B^+_n}\ind{x\in \D_n}\ind{x\in\L_{\gamma_n}}\ind{\mV(x)\geq-\alpha}.
\]
It immediately follows that
\[
\e^\en\left[\Xi_n(\ell, \B^+_n\cap \D_n\cap \L_{\gamma_n},\alpha)\right]=\sum_{|x|=\ell}\p^\en(\eL_x(\tau_n)\geq n^\theta, E_x^{(n)}=1)\ind{x\in \B^+_n}\ind{x\in \D_n}\ind{x\in\L_{\gamma_n}}\ind{\mV(x)\geq-\alpha},
\]
where $\p^\en(\eL_x(\tau_n)\geq n^\theta, E_x^{(n)}=1)=(1+o_n(1))n^{1-\theta} e^{-V(x)}\cf(\frac{n^\theta}{H_x})$ with $\cf(u)=ue^{-u}$. 

Let $\V^\en$ denote the quenched variance. We state the following estimate.
\begin{lem}\label{1exvariance}
Let $0<A<B<\infty$. For $\ell\in[A(\log n)^2, B(\log n)^2]\cap\mathbb{N}$, one has
\begin{equation}
\E[\V^\en(\Xi_n(\ell, \B^+_n\cap \D_n\cap \L_{\gamma_n},\alpha))]\leq c_1\frac{n^{2-2\theta}}{(\log n)^{a\wedge \gamma-4}}.
\end{equation}
\end{lem}
All these previous lemmas will be proved in Section \ref{lems}. The following lemma states the asymptotic behaviour of the quenched expectation $\e^\en\left[\Xi_n(\ell, \B^+_n\cap \D_n\cap \L_{\gamma_n},\alpha)\right]$. 
\begin{lem}\label{SHnorm1ex}
For any $0<A<B<\infty$ and $a+\gamma>6$, one has
\begin{equation*}
\sum_{\ell=A (\log n)^2}^{B (\log n)^2} \sum_{|x|=\ell}e^{-V(x)}\cf(\frac{n^\theta}{H_x})\ind{x\in \B^+_n\cap\D_n\cap\L_{\gamma_n}}\xrightarrow[n\to\infty]{\P^*} D_\infty \times \int_A^B \mathcal{G}(\frac{1}{\sqrt{u}},\frac{\theta}{\sqrt{u}})\frac{du}{u}.
\end{equation*}
\end{lem}
In fact, because of \eqref{badend1ex}, we only need to prove that 
\begin{equation}\label{SHnorm1exK}
\sum_{\ell=A (\log n)^2}^{B (\log n)^2} \sum_{|x|=\ell}e^{-V(x)}\cf(\frac{n^\theta}{H_x})\ind{x\in \B^+_n\cap\D_n^K\cap\L_{\gamma_n}}\xrightarrow[n\to\infty]{\P^*} C_0(A,B,K)D_\infty, 
\end{equation}   
where $C_0(A,B,K)\in (0,\infty)$ and $\lim_{K\to\infty}C_0(A,B,K)=\int_A^B \mathcal{G}(\frac{1}{\sqrt{u}},\frac{1}{\sqrt{u}},\frac{\theta}{\sqrt{u}})\frac{du}{u}$. The proof of \eqref{SHnorm1exK} is postponed in Section \ref{SHnorm}.  

Let us prove Proposition \ref{1Ex} by use of these lemmas.
\begin{proof}[Proof of Proposition \ref{1Ex}]
Note that for any $\delta>0$ and $\beta>0$,
\begin{align*}
&\p^*\left(|\frac{R_{n^\theta}(\tau_n, 1)}{n^{1-\theta}}- \Lambda_1(\theta)D_\infty|\geq \delta D_\infty \right)\\
\leq &\P^*(\inf_{x\in\T}V(x)<-\alpha)+\p^*\left(\max_{1\leq i\leq \tau_n}|X_i|>c_0 (\log n)^3\right)+\P^*(D_\infty<\beta)+\p^*(\exists k\leq \tau_n, X_k\in\mathcal{L}_n)\\
&+\p^*\left(|\frac{\sum_{\ell=1}^{c_0(\log n)^3}\sum_{|x|=\ell}\ind{\eL_x(\tau_n)\geq n^\theta, E_x^{(n)}=1}\ind{x\in\L_{n}}\ind{\mV(x)\geq-\alpha}}{n^{1-\theta}}- \Lambda_1(\theta)D_\infty|\geq \delta D_\infty , D_\infty\ge\beta\right).
\end{align*}
Here $\p^*(\exists k\leq \tau_n, X_k\in\mathcal{L}_n)=o_n(1)$ according to \cite{HuShi16}. 
By Lemmas \ref{smallpart1ex}, \ref{smallpart1ex+} and \ref{smallpart1exL}, one has
\begin{multline}\label{1Exgoodsite}
\p^*\left(|\frac{R_{n^\theta}(\tau_n, 1)}{n^{1-\theta}}- \Lambda_1(\theta)D_\infty|\geq \delta D_\infty \right)
\leq c_2e^{-\alpha}+o_n(1)+o_\beta(1)\\
+\p^*\left(|\frac{\sum_{\ell=\varepsilon(\log n)^2}^{(\log n)^2/\varepsilon}\Xi_n(\ell, \B^+_n\cap\D_n\cap\L_{\gamma_n},\alpha)}{n^{1-\theta}}- \Lambda_1(\theta)D_\infty|\geq \frac{\delta}{2} D_\infty , D_\infty\ge\beta\right).
\end{multline}
Here,  one sees that
\begin{align*}
&\p^*\left(|\frac{\sum_{\ell=\varepsilon(\log n)^2}^{(\log n)^2/\varepsilon}\Xi_n(\ell, \B^+_n\cap\D_n\cap\L_{\gamma_n},\alpha)}{n^{1-\theta}}- \Lambda_1(\theta)D_\infty|\geq \frac{\delta}{2} D_\infty , D_\infty\ge\beta\right)\\
\leq &\p^*\left(|\frac{\sum_{\ell=\varepsilon(\log n)^2}^{(\log n)^2/\varepsilon}\Xi_n(\ell, \B^+_n\cap\D_n\cap\L_{\gamma_n},\alpha)-\e^\en[\Xi_n(\ell, \B_n^+\cap \D_n\cap \L_{\gamma_n},\alpha)]}{n^{1-\theta}}|\ge \delta\beta/4\right)\\
&+\p^*\left(|\sum_{\ell=\varepsilon(\log n)^2}^{(\log n)^2/\varepsilon}\e^\en[\Xi_n(\ell, \B^+_n\cap\D_n\cap\L_{\gamma_n},\alpha)]-\Lambda_1(\theta)D_\infty|\geq\frac{\delta}{4}D_\infty, D_\infty\geq\beta\right)
\end{align*}
By Chebyshev's inequality and then Cauchy-Schwartz inequality,
\begin{align*}
&\p\left(|\frac{\sum_{\ell=\varepsilon(\log n)^2}^{(\log n)^2/\varepsilon}\Xi_n(\ell, \B^+_n\cap\D_n\cap\L_{\gamma_n},\alpha)-\e^\en[\Xi_n(\ell, \B_n^+\cap \D_n\cap \L_{\gamma_n},\alpha)]}{n^{1-\theta}}|\ge \delta\beta/4\right)\\
\leq& \frac{16}{(\delta\beta)^2n^{2-2\theta}}\e\left[\left(\sum_{\ell=\varepsilon(\log n)^2}^{(\log n)^2/\varepsilon}\Xi_n(\ell, \B^+_n\cap\D_n\cap\L_{\gamma_n},\alpha)-\e^\en[\Xi_n(\ell, \B_n^+\cap \D_n\cap \L_{\gamma_n},\alpha)]\right)^2\right]\\
\leq &\frac{16}{(\delta\beta)^2n^{2-2\theta}}\sum_{\ell=\varepsilon(\log n)^2}^{(\log n)^2/\varepsilon} 1 \sum_{\ell=\varepsilon(\log n)^2}^{(\log n)^2/\varepsilon}\E\left[\V^\en(\Xi_n(\ell, \B^+_n\cap\D_n\cap\L_{\gamma_n},\alpha))\right],
\end{align*}
which is $o_n(1)$ by Lemma \ref{1exvariance} as long as $a\wedge\gamma>8$. On the other hand, 
\begin{align*}
&\p^*\left(|\sum_{\ell=\varepsilon(\log n)^2}^{(\log n)^2/\varepsilon}\e^\en[\Xi_n(\ell, \B^+_n\cap\D_n\cap\L_{\gamma_n},\alpha)]-\Lambda_1(\theta)D_\infty|\geq\frac{\delta}{4}D_\infty, D_\infty\geq\beta\right)\\
\leq &\p^*\left(|\sum_{\ell=\varepsilon(\log n)^2}^{(\log n)^2/\varepsilon}\sum_{|x|=\ell}(1+o_n(1))e^{-V(x)}\cf(\frac{n^\theta}{H_x})\ind{x\in \B^+_n\cap\D_n\cap\L_{\gamma_n}}-\Lambda_1(\theta)D_\infty|\geq\frac{\delta}{4}D_\infty, D_\infty\geq\beta\right)\\
&+\P^*(\inf V(u)< -\alpha).
\end{align*}
We thus deduce from Lemma \ref{SHnorm1ex} that 
$$\limsup_{\varepsilon\to0}\limsup_{n\to\infty}\p^*\left(|\frac{\sum_{\ell=\varepsilon(\log n)^2}^{(\log n)^2/\varepsilon}\Xi_n(\ell, \B^+_n\cap\D_n\cap\L_{\gamma_n},\alpha)}{n^{1-\theta}}- \Lambda_1(\theta)D_\infty|\geq \frac{\delta}{2} D_\infty , D_\infty\ge\beta\right)\le c_2 e^{-\alpha}.$$ Going back to \eqref{1Exgoodsite} and letting $\alpha\to\infty$ and $\beta\downarrow0$, we therefore conclude that for any $\delta>0$.
\[
\limsup_{n\to\infty}\p^*\left(|\frac{R_{n^\theta}(\tau_n, 1)}{n^{1-\theta}}- \Lambda_1(\theta)D_\infty|\geq \delta D_\infty \right)=0.
\]
\end{proof}

%Recall that
%\begin{equation}\label{keyprob}
%\p^\en(\eL_x(\tau_n)\geq n^\theta, E_x^{(n)}=1)=na_xb_x^{n^\theta-1}(1-a_x)^{n-1}.
%\end{equation}
\section{Generalised Seneta-Heyde scaling: proof of Lemmas \ref{largepart} and \ref{SHnorm1ex}}\label{SHnorm}
In this section, we study the following sum: for any $0<A<B<\infty$,
\begin{equation}
\chi_i(A,B, r):=\sum_{Ar^2\leq m\leq Br^2}\sum_{|z|=m}e^{-V(z)}F_i(z,r), \textrm{ for } i=1,2,3;
\end{equation}
where
\begin{align}
F_1(z,r):=&e^{V(z)-(1-\theta)r-b}\ind{\MV(z)-V(z)\leq \theta r+t_r, V(z)\leq (1-\theta)r+b}\\
F_2(z,r):=&e^{V(z)-(1-\theta)r-b}\ind{\MV(z)\leq r+t_r, V(z)\leq (1-\theta)r+b}\\
F_3(z,r):=&\cf(\frac{e^{\theta r}}{H_z})\ind{\MV(z)\geq r+t_r, \max_{y\leq z}(\MV(y)-V(y))\leq r+s_r, \MV(z)-V(z)\in[\theta r-K, \theta r+K]}
\end{align}
with $t_r=o(r)$, $s_r=o(r)$, $K>0$ and $b\in\mathbb{R}$ such that $s_r+6\log r< t_r$. We are going to show that as $r\to\infty$.
\begin{equation}\label{keycvginp}
\chi_i(A,B,r)\xrightarrow{\textrm{ in } \P^*} \Cb_i(A,B)D_\infty, \textrm{ for } i=1,2,3,
\end{equation}
where $\Cb_i(A,B)$ are positive constants which will be determined later. %We mainly follow the idea of \cite{BM19}. %In the following, we only treat $\chi_2$ and $\chi_3$ as the convergence of $\chi_1$ can be obtained from that of $\chi_2$. We will explain it later. 

One can see immediately that Lemma \ref{largepart} is mainly based on the convergences of $\chi_1$ and $\chi_2$ and that Lemma \ref{SHnorm1ex} is based on the convergence of $\chi_3$ with $r=\log n$. To complete the proof of Lemma \ref{largepart}, as $F_1\leq F_2$, we still need to check the following estimate. 
\begin{lem}\label{largepartbadtime}
For any $\alpha>0$, as $\varepsilon\downarrow0$, we have
\begin{align}
\limsup_{n\to\infty}\E\left[\sum_{m=1}^{\varepsilon (\log n)^2}\sum_{|z|=m}e^{-V(z)}F_2(z,\log n)\ind{\mV(z)\geq-\alpha}\right]=&o_\varepsilon(1);\label{largepartsmalltime}\\
\limsup_{n\to\infty}\E\left[\sum_{(\log n)^2/\varepsilon}^{c_0(\log n)^3}\sum_{|z|=m}e^{-V(z)}F_2(z,\log n)\ind{\mV(z)\geq -\alpha}\right]=&o_\varepsilon(1).\label{largepartlargetime}
\end{align}
\end{lem}
To conclude Lemma \ref{SHnorm1ex}, in other words, to get \eqref{SHnorm1exK}, we need to compare $\{z\in \L_{\gamma_n}\}=\{\max_{y\le z}H_y\leq \frac{n}{(\log n)^\gamma}\}$ with $\{\max_{y\leq z}(\MV(y)-V(y))\leq r+s_r\}$. In fact, note that $e^{\MV(y)-V(y)}\le H_y\leq |z| e^{\MV(y)-(y)}$. Thus for $|z|\le B(\log n)^2$ with $n\gg1$, 
\[
\ind{\max_{y\leq z}(\MV(y)-V(y))\leq \log n-(\gamma+3)\log \log n}\le \ind{z\in \L_{\gamma_n}}\leq \ind{\max_{y\leq z}(\MV(y)-V(y))\leq \log n-\gamma \log\log n}
\] 
Note also that in Lemma \ref{SHnorm1ex}, $t_r=a\log\log n$ with $a+\gamma>6$. Therefore, we can deduce Lemma \ref{SHnorm1ex} from \eqref{keycvginp} for $i=3$.

In the following, we prove \eqref{keycvginp} and check Lemma \ref{largepartbadtime} in Section \ref{lems}. Our proof of \eqref{keycvginp} mainly follows the idea of \cite{BM19}.

\textbf{Outline of proof of \eqref{keycvginp}.} It is known that for any $\varepsilon\in(0,1)$, there exists $k_0\geq1$ such that 
\begin{equation}\label{infV}
\P\left(\inf_{n\geq k_0}\inf_{|z|=n}V(z)\geq0\right)\geq 1-\varepsilon,
\end{equation}
with the convention that $\inf\emptyset=\infty$.
For any $r$ such that $Ar^2\geq 2k_0$, let 
\[
\widetilde{\chi}_i(A,B, r, k_0):=\sum_{Ar^2\leq m\leq Br^2}\sum_{|z|=m}\widetilde{F}_i(z,r, k_0)
\]
where $\widetilde{F}_i(z,r,k_0):=F_i(z,r)\ind{\min_{z_0\leq y\leq z}V(y)\geq0}$
%\begin{align*}
%\widetilde{F}_i(z,r):=&e^{V(z)-(1-\theta)r+b}\ind{\mV_{z_0}(z)+v(z_0)\geq0, \MV_{z_0}(z)-V_{z_0}(z)\leq \theta r+t_r, V_{z_0}(z)+V(z_0)\leq (1-\theta)r-b}\\
%\widetilde{F}_2(z,r):=&e^{V(z)-(1-\theta)r-b}\ind{\mV_{z_0}(z)+v(z_0)\geq0, \MV_{z_0}(z)+V(z_0)\leq r+t_r, V_{z_0}(z)+V(z_0)\leq (1-\theta)r+b}\\
%\widetilde{F}_3(z,r):=&f(\frac{e^{\theta r}}{H_z})\ind{\mV_{z_0}(z)+v(z_0)\geq0, \MV(z)\geq r+t_r, \max_{\rho \leq y\leq z}(\MV_{z_0}(y)-V_{z_0}(y))\leq r+s_r, \MV_{z_0}(z)-V_{z_0}(z)\in[\theta r+y, \theta r+y+h]}
%\end{align*}
with $z_0:=z_{k_0}$. It then follows from \eqref{infV} that for any $\varepsilon>0$ and $i=1,2,3$, there exists $k_0\geq1$ such that for any $k\geq k_0$,
\begin{equation}\label{infVchi}
\P\left(\forall r\geq 1, \chi_i(A,B,r)\neq \widetilde{\chi}_i(A,B,r, k)\right)\le 2\varepsilon.
\end{equation}
So, according to \cite{BM19}, it suffices to show that for any $\lambda>0$ and $i=1,2,3$, a.s.,
\begin{equation}\label{keycvg+}
\lim_{k_0\to\infty}\limsup_{r\to\infty}\E[e^{-\lambda \tilde{\chi}_i(A,B,r, k_0 )}\vert \mathcal{F}_{k_0}]=\lim_{k_0\to\infty}\liminf_{r\to\infty}\E[e^{-\lambda \tilde{\chi}_i(A,B,r, k_0 )}\vert \mathcal{F}_{k_0}]=\exp\{-\lambda \Cb_i(A,B)D_\infty\}.
\end{equation}
By Lemma B.1 of \cite{BM19} and a Cantor diagonal extraction argument, this yields the convergence in probability of $\tilde{\chi}_i(A,B, r, k_0(r))$ towards $\Cb_i(A,B)D_\infty$. Then \eqref{keycvginp} follows immediately from \eqref{infVchi}.

Let us check \eqref{keycvg+}. Observe that by Jensen's inequality,
\begin{align}\label{lowerbdSH}
\E[e^{-\lambda \tilde{\chi}_i(A,B,r, k_0)}\vert \mathcal{F}_{k_0}]=&\prod_{|u|=k_0}\E\left[\exp\{-\lambda \sum_{Ar^2\leq m\leq Br^2}\sum_{|z|=m}\ind{z_0=u}e^{-V(z)}\widetilde{F}_i(z,r)\}\Big\vert \calF_{k_0}\right]\nonumber\\
\geq & \exp\left\{-\lambda \sum_{|u|=k_0}\sum_{Ar^2\leq m\leq Br^2}\E\left[\sum_{|z|=m}\ind{z_0=u}e^{-V(z)}\widetilde{F}_i(z,r)\}\vert \calF_{k_0}\right]\right\}\nonumber\\
\geq&\exp\{-\lambda\sum_{|u|=k_0}(1+o_r(1))\E_{V(u)}[\hat{\chi}_i]\}\ind{\max_{|u|=k_0}\MV(u)\leq r^{1/3}, \min_{|u|=k_0}\mV(u)\geq -r^{1/3}},
\end{align}
where $\hat{\chi}_i=\hat{\chi}_i(A,B,r,k_0):=\sum_{Ar^2-k_0\leq m\leq Br^2-k_0}\sum_{|z|=m}e^{-V(z)}\widehat{F}_i(z,r,k_0)$ with
\begin{align*}
\widehat{F}_1(z,r,k_0):=&e^{V(z)-(1-\theta)r-b}\ind{\mV(z)\geq0, \MV(z)-V(z)\leq \theta r+t_r, V(z)\leq (1-\theta)r+b};\\
\widehat{F}_2(z,r,k_0):=&e^{V(z)-(1-\theta)r-b}\ind{\mV(z)\geq0, \MV(z)\leq r+t_r, V(z)\leq (1-\theta)r+b};\\
\widehat{F}_3(z,r,k_0):=&f(\frac{e^{\theta r}}{H_z})\ind{\mV(z)\geq0, \MV(z)\geq r+t_r, \max_{y\leq z}(\MV(y)-V(y))\leq r+s_r, \MV(z)-V(z)\in[\theta r-K,\theta r+K]}.
%\textrm{ and }\widehat{F}_3^*(z,r,k_0):=&f(\frac{e^{\theta r}}{H_z})\ind{\mV(z)\geq0, \MV(z)\geq r+t_r, \max_{y\leq z}(\MV(y)-V(y))\leq r+s_r, \MV(z)-V(z)\in[\theta r-K,\theta r+K]}.
\end{align*}
Let us explain a little the last inequality in \eqref{lowerbdSH}. Note that if $\{\max_{|u|=k_0}\MV(u)\leq r^{1/3}\}$, one has $\MV(z)=\max_{z_0\le y\le z}V(y)$. Thus, for $i=1,2$,
\[
\E\left[\sum_{|z|=m}\ind{z_0=u}e^{-V(z)}\widetilde{F}_i(z,r)\}\vert \calF_{k_0}\right]= \E_{V(u)}\left[\sum_{|z|=m-k_0}e^{-V(z)}\widehat{F}_i(z,r,k_0)\right].
\]
For $i=3$, one can see that given $\{\max_{|u|=k_0}\MV(u)\leq r^{1/3}, \min_{|u|=k_0}\mV(u)\geq-r^{1/3}\}$ and $\{\MV(z)\geq r+t_r, \MV(z)-V(z)\in[\theta r+ d,\theta r+ d+h]\}$, we have moreover $\{\max_{y\leq z}(\MV(y)-V(y))\leq r+s_r\}=\{\max_{z_0\leq y\leq z}(\MV(y)-V(y))\leq r+s_r\}$ and $f(\frac{e^{\theta r}}{H_z})=(1+o_r(1))f(\frac{e^{\theta r}}{\sum_{z_0\le y\le z}e^{V(y)-V(z)}})$ as
\[
\frac{|H_z-\sum_{z_0\le y\le z}e^{V(y)-V(z)}|}{H_z}\leq k_0e^{r^{1/3}-r-t_r}=o_r(1).
\]
This leads to 
\[
\E\left[\sum_{|z|=m}\ind{z_0=u}e^{-V(z)}\widetilde{F}_3(z,r)\}\vert \calF_{k_0}\right]= (1+o_r(1))\E_{V(u)}\left[\sum_{|z|=m-k_0}e^{-V(z)}\widehat{F}_3(z,r,k_0)\right].
\]
We next turn to the upper bound of $\E[e^{-\lambda \tilde{\chi}_i(A,B,r, k_0)}\vert \mathcal{F}_{k_0}]$. For any $\delta\in(0,1)$, let $\lambda_\delta:=\lambda e^{-\lambda \delta}$ and 
\[
\hat{\chi}^{(\delta)}_i=:\sum_{Ar^2-k_0\leq m\leq Br^2-k_0} \sum_{|z|=m}e^{-V(z)}\widehat{F}_i(z,r,k_0)\ind{\sum_{|z|=m}e^{-V(z)}\widehat{F}_i(z,r,k_0)\leq \frac{\delta}{Br^2}}
\]
As a consequence of the fact $e^{-\lambda t}\leq 1-\lambda_\delta t$ for any $t\in[0,\delta]$,
\begin{align}\label{upperbdSH}
&\E[e^{-\lambda \tilde{\chi}_i(A,B,r, k_0)}\vert \mathcal{F}_{k_0}]\nonumber\\
\leq&\prod_{|u|=k_0}\E\left[\exp\{-\lambda \sum_{Ar^2\leq m\leq Br^2}\sum_{|z|=m}\ind{z_0=u}e^{-V(z)}\widetilde{F}_i(z,r)\ind{\sum_{|z|=m}\ind{z_0=u}e^{-V(z)}\widetilde{F}_i(z,r)\leq \frac{\delta}{Br^2}}\}\Big\vert \calF_{k_0}\right]\nonumber\\
\leq &\prod_{|u|=k_0}\left(1-\lambda_\delta \E\left[ \sum_{Ar^2\leq m\leq Br^2}\sum_{|z|=m}\ind{z_0=u}e^{-V(z)}\widetilde{F}_i(z,r)\ind{\sum_{|z|=m}\ind{z_0=u}e^{-V(z)}\widetilde{F}_i(z,r)\leq \frac{\delta}{Br^2}}\Big\vert \calF_{k_0}\right]\right)\nonumber\\
\leq &\exp\left\{-\lambda_\delta\sum_{|u|=k_0}\E\left[ \sum_{Ar^2\leq m\leq Br^2}\sum_{|z|=m}\ind{z_0=u}e^{-V(z)}\widetilde{F}_i(z,r)\ind{\sum_{|z|=m}\ind{z_0=u}e^{-V(z)}\widetilde{F}_i(z,r)\leq \frac{\delta}{Br^2}}\Big\vert \calF_{k_0}\right]\right\}
\end{align}
which as explained above, for $r$ large enough, is bounded by
\[
\exp\left\{-\lambda_\delta\sum_{|u|=k_0}(1+o_r(1))\E_{V(u)}[\hat{\chi}_i^{(\delta/2)}]\right\}+\ind{\max_{|u|=k_0}\MV(u)>r^{1/3}}+\ind{\min_{|u|=k_0}\mV(u)<-r^{1/3}}.
\]
For \eqref{lowerbdSH} and \eqref{upperbdSH}, letting $r\to\infty$ brings out that
\begin{multline}
\liminf_{r\to\infty}\exp\left\{-\lambda\sum_{|u|=k_0}(1+o_r(1))\E_{V(u)}[\hat{\chi}_i]\right\}\leq \liminf_{r\to\infty}\E[e^{-\lambda \tilde{\chi}_i(A,B,r, k_0)}\vert \mathcal{F}_{k_0}]\\\leq\limsup_{r\to\infty}\E[e^{-\lambda \tilde{\chi}_i(A,B,r, k_0)}\vert \mathcal{F}_{k_0}]\leq \limsup_{r\to\infty}\exp\left\{-\lambda_\delta\sum_{|u|=k_0}(1+o_r(1))\E_{V(u)}[\hat{\chi}_i^{(\delta/2)}]\right\}
\end{multline}
Next, we claim that for any $x\geq0$, $\lim_{r\to\infty}\E_{x}[\hat{\chi}_i]=C_i(A,B)\Ren(x)e^{-x}$ with $C_i(A,B)$ a positive constant and $\Ren(\cdot)$ is the renewal function defined in \eqref{renewalf}. Moreover, we stress that for $\delta>0$ and $x\gg1$,
\[
\limsup_{r\to\infty}\E_{x}[\hat{\chi}_i-\hat{\chi}_i^{(\delta)}]=o_{x}(1)\Ren(x)e^{-x}.
\]
These are stated in the following lemma.
\begin{lem}\label{firstmomentchi} 
For any $x\geq0$, $\delta>0$, as $r\to\infty$,
\begin{align}
\lim_{r\to\infty}\E_x\left[\sum_{Ar^2-k_0\leq m\leq Br^2-k_0} \sum_{|z|=m}e^{-V(z)}\widehat{F}_i(z,r,k_0)\right]=&C_i(A,B)\Ren(x)e^{-x},\label{1momchi+}\\
\limsup_{r\to\infty}\E_x\left[\sum_{Ar^2-k_0\leq m\leq Br^2-k_0} \sum_{|z|=m}e^{-V(z)}\widehat{F}_i(z,r,k_0)\ind{\sum_{|z|=m}e^{-V(z)}\widehat{F}_i(z,r,k_0)> \frac{\delta}{Br^2}}\right]=&o_x(1)\Ren(x)e^{-x},\label{1momchi-}
\end{align}
where 
\[
C_1(A,B)=C_2(A,B)=\frac{\cb_+}{\sigma}\int_A^B\mathcal{C}_0(\frac{\theta}{\sqrt{u}},\frac{1-\theta}{\sqrt{u}})\frac{du}{u},
\]
and 
\[
C_3(A,B)=\int_{-K}^{K}\E[\cf(\frac{e^{-s}}{\Hinf_\infty+\Hinf^{(-)}_\infty-1})]ds \int_A^B \mathcal{G}(\frac{1}{\sqrt{u}},\frac{\theta}{\sqrt{u}})\frac{du}{u}.
\]
\end{lem}
By \eqref{cvgrenewalf}, $\Ren(u)\sim \cb_\Ren u$ as $u\to\infty$. Recall also that the derivative martingale $D_{k_0}=\sum_{|u|=k_0}V(u)e^{-V(u)}$ converges a.s. to some non-negative limit $D_\infty$. As a result, we obtain
\[
\lim_{k_0\to\infty}\lim_{r\to\infty}\E[e^{-\lambda \tilde{\chi}_i(A,B,r, k_0)}\vert \mathcal{F}_{k_0}]=\exp\{-\lambda \cb_\Ren C_i(A,B)D_\infty\}.
\]
By Lemma B.1 of \cite{BM19} and a Cantor diagonal extraction argument, this yields convergence in probability of $\tilde{\chi}_i(A,B,r, k_0(r))$ towards $\cb_\Ren C_i(A,B)D_\infty$. In view of \eqref{infVchi}, we obtain the convergence in probability of $\chi_i(A,B,r)$ towards $\Cb_i(A,B)D_\infty$ under $\P$ (hence under $\P^*$) with $\Cb_i(A,B)=\cb_\Ren C_i(A,B)$. Note that $\int_{\R}\E[\cf(\frac{e^{-s}}{ \Hinf_\infty+\Hinf^{(-)}_\infty-1})]ds=1$. So Lemma \ref{SHnorm1ex} holds and finally Proposition \ref{1Ex} holds with
\[
\Lambda_1(\theta)=\cb_\Ren \int_0^\infty \mathcal{G}(\frac{1}{\sqrt{u}},\frac{\theta}{\sqrt{u}})\frac{\d u}{u}.
\]
And Proposition \ref{mEx} holds with
\[
\Lambda_0(\theta)=\cb_\Ren\frac{\cb_+}{\sigma}\int_0^\infty\mathcal{C}_0(\frac{\theta}{\sqrt{u}},\frac{1-\theta}{\sqrt{u}})\frac{\d u}{u}=\frac{\sqrt{2}}{\sqrt{\pi}\sigma^2}\int_0^\infty\mathcal{C}_0(\frac{\theta}{\sqrt{u}},\frac{1-\theta}{\sqrt{u}})\frac{\d u}{u},
\]
because of \eqref{prodc}.

In order to conclude \eqref{keycvginp}, we only need to prove Lemma \ref{firstmomentchi} mainly for $i=2,3$. 

\begin{proof}[Proof of Lemma \ref{firstmomentchi}]

\textbf{Proof of \eqref{1momchi+}.} By Many-to-one lemma, we have
\[
\E_x[\hat{\chi}_2]=e^{-x}\sum_{Ar^2-k_0\leq m\leq Br^2-k_0}\E_x\left[e^{S_m-(1-\theta)r-b}; \mS_m\geq0, \MS_m\leq r+t_r, S_m\leq (1-\theta)r+b\right]
\]
%\begin{align*}
%\E_x[\hat{\chi}_2]=&e^{-x}\sum_{Ar^2-k_0\leq m\leq Br^2-k_0}\E_x\left[e^{S_m-(1-\theta)r-b}; \mS_m\geq0, \MS_m\leq r+t_r, S_m\leq (1-\theta)r+b\right];\\
%\E_x[\hat{\chi_3}]=&e^{-x}\sum_{Ar^2-k_0\leq m\leq Br^2-k_0}\E_x\left[f(\frac{e^{\theta r}}{H_m^S})\ind{\mS_m\geq0, \MS_m\geq r+t_r, \max_{k\leq m}(\MS_k-S_k)\leq r+s_r, \MS_m-S_m\in[\theta r-t_r, \theta r+t_r]}\right].
%\end{align*}
By \eqref{eSmScvg}, as $r\to\infty$, 
\begin{align*}
\E_x[\hat{\chi}_2]=&\Ren(x)e^{-x}\sum_{Ar^2-k_0\leq m\leq Br^2-k_0}\frac{1+o_r(1)}{m}\mathcal{C}_0(\frac{r}{\sqrt{m}},\frac{(1-\theta)r}{\sqrt{m}})\\
=&\Ren(x)e^{-x}(1+o_r(1))\int_{A-\frac{k_0}{r^2}}^{B-\frac{k_0}{r^2}}\frac{\cb_+}{\sigma}\mathcal{C}_0(\frac{\theta}{\sqrt{u}},\frac{1-\theta}{\sqrt{u}})\frac{du}{u}
\end{align*}
which converges to $\Ren(x)e^{-x}\frac{\cb_+}{\sigma}\int_A^B\mathcal{C}_0(\frac{\theta}{\sqrt{u}},\frac{1-\theta}{\sqrt{u}})\frac{du}{u}$. By \eqref{eSmSMScvg} instead of \eqref{eSmScvg}, we get \eqref{1momchi+} for $i=1$. Moreover, we get that
\[
C_1(A,B)=C_2(A,B)=\int_A^B\frac{\cb_+}{\sigma}\mathcal{C}_0(\frac{\theta}{\sqrt{u}},\frac{1-\theta}{\sqrt{u}})\frac{du}{u}.
\]

For $i=3$, by \eqref{cvgeSMS}, as $r\to\infty$,
\begin{align*}
\E_x[\hat{\chi}_3]=&e^{-x}\Ren(x)\sum_{Ar^2-k_0\leq m\leq Br^2-k_0}\int_{-K}^{K}\E[\cf(\frac{e^{-s}}{\Hinf_\infty+\Hinf^{(-)}_\infty}-1)]ds \frac{1+o_r(1)}{m}\mathcal{G}(\frac{r}{\sqrt{m}}, \frac{\theta r}{\sqrt{m}})\\
\rightarrow&\Ren(x)e^{-x}\int_{-K}^{K}\E[\cf(\frac{e^{-s}}{\Hinf_\infty+\Hinf^{(-)}_\infty-1})]ds \int_A^B \mathcal{G}(\frac{1}{\sqrt{u}},\frac{\theta}{\sqrt{u}})\frac{du}{u}
\end{align*}

\textbf{Proof of \eqref{1momchi-}.} First, by Markov inequality,
\begin{align*}
\E_x[\widehat{\chi}_i-\widehat{\chi}_i^{(\delta)}]=& \sum_{Ar^2-k_0\leq m\leq Br^2-k_0}\E_x\left[\sum_{|z|=m}e^{-V(z)}\widehat{F}_i(z,r,k_0)\ind{\sum_{|z|=m}e^{-V(z)}\widehat{F}_i(z,r,k_0)> \frac{\delta}{Br^2}}\right]\\
\leq & \sum_{m=Ar^2-k_0}^{Br^2-k_0}\E_x\left[\sum_{|z|=m}e^{-V(z)}\widehat{F}_i(z,r,k_0)\left(\frac{Br^2\sum_{|z|=m}e^{-V(z)}\widehat{F}_i(z,r,k_0)}{\delta}\wedge 1\right)\right].
\end{align*}
Note that $\widehat{F}_1\leq \widehat{F}_2$. So, we only need to treat it for $i=2, 3$.
By Lyons' change of measure and Proposition \ref{spinedecomp}, we then get that
\begin{align}
&\E_x[\widehat{\chi}_i-\widehat{\chi}_i^{(\delta)}]\nonumber\\%\leq \E_x\left[\sum_{Ar^2-k_0\leq m\leq Br^2-k_0} \sum_{|z|=m}e^{-V(z)}\widehat{F}_i(z,r,k_0)\left(\frac{Br^2}{\delta}\sum_{|z|=m}e^{-V(z)}\widehat{F}_i(z,r,k_0) \wedge 1\right)\right]\nonumber\\
\leq& \sum_{m=Ar^2-k_0}^{Br^2-k_0}e^{-x}\E_{\widehat{\Q}_x}\left[\widehat{F}_i(w_m,r,k_0)\left([\frac{Br^2}{\delta}\sum_{j=1}^m\sum_{u\in\Omega(w_j)}\sum_{|z|=m, z\geq u}e^{-V(z)}\widehat{F}_i(z,r,k_0)+e^{-V(w_m)}\widehat{F}_i(w_m,r,k_0)]\wedge 1\right)\right]\nonumber\\
\leq &UB_1(A,B,r,i)+UB_2(A,B,r,i),% \sum_{Ar^2-k_0\leq m\leq Br^2-k_0}e^{-x}\E_{\Q_x}\left[\widehat{F}_i(w_m,r,k_0)\left(\frac{Br^2}{\delta}e^{-V(w_m)}\widehat{F}_i(w_m,r,k_0) \wedge 1\right)\right]\\
%&+ \sum_{Ar^2-k_0\leq m\leq Br^2-k_0}e^{-x}\Q_x\left[\widehat{F}_i(w_m,r,k_0)\left(\frac{Br^2}{\delta}(\sum_{j=1}^m\sum_{u\in\Omega(w_j)}\sum_{|z|=m, z\geq u}e^{-V(z)}\widehat{F}_i(z,r,k_0)) \wedge 1\right)\right].
\end{align}
where 
\begin{align*}
UB_1(A,B,r,i):=& \sum_{Ar^2-k_0\leq m\leq Br^2-k_0}e^{-x}\E_{\widehat{\Q}_x}\left[\widehat{F}_i(w_m,r,k_0)\left(\frac{Br^2}{\delta}e^{-V(w_m)}\widehat{F}_i(w_m,r,k_0) \wedge 1\right)\right],\\
UB_2(A,B,r,i):=& \sum_{Ar^2-k_0\leq m\leq Br^2-k_0}e^{-x}\E_{\widehat{\Q}_x}\left[\widehat{F}_i(w_m,r,k_0)\left([\frac{Br^2}{\delta}\sum_{j=1}^m\sum_{u\in\Omega(w_j)}\sum_{|z|=m, z\geq u}e^{-V(z)}\widehat{F}_i(z,r,k_0)]\wedge 1\right)\right].
\end{align*}
Observe that for $i=3$, by Proposition \ref{spinedecomp} and \eqref{mSbd},
\begin{align*}
UB_1(A,B,r,3)\leq &  \sum_{Ar^2-k_0\leq m\leq Br^2-k_0}\frac{Br^2 e^{-x}}{\delta}\E_x\left[e^{-S_m}\ind{\mS_m\geq0, \MS_m\geq r+t_r, \MS_m-S_m\in[\theta r-K,\theta r+K]}\right]\\
\leq & \sum_{Ar^2-k_0\leq m\leq Br^2-k_0}\frac{Br^{2}e^{-(1-\theta)r-t_r+K}}{\delta\sqrt{m}}c_3(1+x)e^{-x}=o_r(1)\Ren(x)e^{-x}.
\end{align*}
Note also that as $\widehat{F}_2\leq 1$, by \eqref{mSSbd},
\begin{align*}
UB_1(A,B,r,2)\leq &\sum_{Ar^2-k_0\leq m\leq Br^2-k_0}e^{-x}\frac{Br^2}{\delta}\E_{\widehat{\Q}_x}\left[\widehat{F}_2(w_k,r,k_0)e^{-V(w_m)}\right]\\
\leq &\sum_{Ar^2-k_0\leq m\leq Br^2-k_0}e^{-x}\frac{Br^2e^{-(1-\theta)r-b}}{\delta}\P_x(\mS_m\geq0,S_m\le (1-\theta)r+b)\\
\leq &\sum_{Ar^2-k_0\leq m\leq Br^2-k_0}e^{-x}\frac{Br^2e^{-(1-\theta)r-b}}{\delta}\frac{c_4(1+x)(1+r)^2}{m^{3/2}}=o_r(1)\Ren(x)e^{-x}.
\end{align*}
Recall that $\mathscr{G}=\sigma\{(w_k, V(w_k))_{k\geq0}, (u, V(u))_{u\in\cup_{k\geq0}\Omega(w_k)}\}$. So,
\begin{multline}\label{sumOmegaF3}
UB_2(A,B,r,i)\leq\\
\sum_{Ar^2-k_0\leq m\leq Br^2-k_0}e^{-x}\E_{\widehat{\Q}_x}\left[\widehat{F}_i(w_m,r,k_0)\left((\frac{Br^2}{\delta}\sum_{j=1}^m\sum_{u\in\Omega(w_j)}\E_{\widehat{\Q}_x}\left[\sum_{|z|=m, z\geq u}e^{-V(z)}\widehat{F}_i(z,r,k_0)\Big\vert \mathscr{G}\right]) \wedge 1\right)\right],
\end{multline}
where for $i=2$ and $u\in\Omega(w_j)$, by branching property at $u$ and then by \eqref{mSSbd},
\begin{align}\label{ub2for2}
&\E_{\widehat{\Q}_x}\left[\sum_{|z|=m, z\geq u}e^{-V(z)}\widehat{F}_2(z,r,k_0)\Big\vert \mathscr{G}\right]\nonumber\\
\leq &e^{-V(u)}\ind{\mV(u)\geq0}\E_{V(u)}[e^{S_{m-j}-(1-\theta)r-b}; \mS_{m-j}\geq0, S_{m-j}\leq (1-\theta)r-b]\nonumber\\
\leq &e^{-V(u)}\ind{\mV(u)\geq0}\frac{c_5(1+V(u))(1+r)}{(m-j+1)^{3/2}}\ind{j<m/2}+e^{-V(u)}\ind{\mV(u)\geq0}\ind{j\ge m/2},
\end{align}
and for $i=3$ and $u\in\Omega(w_j)$,
\begin{align*}
&\E_{\widehat{\Q}_x}\left[\sum_{|z|=m, z\geq u}e^{-V(z)}\widehat{F}_3(z,r,k_0)\vert \mathscr{G}\right]\\
=&\E_{\widehat{\Q}_x}\left[\sum_{|z|=m, z\geq u}e^{-V(z)}\widehat{F}_3(z,r,k_0)\ind{\MV(z)=\MV(u)}\vert \mathscr{G}\right]+\E_{\widehat{\Q}_x}\left[\sum_{|z|=m, z\geq u}e^{-V(z)}\widehat{F}_3(z,r,k_0)\ind{\MV(z)>\MV(u)}\vert \mathscr{G}\right]\\
\leq &e^{-V(u)}\ind{\mV(u)\geq0, \MV(u)-V(u)\leq r+s_r, \MV(u)\geq r+t_r}\P_{V(u)}(\mS_{m-j}\geq0, x-S_{m-j}\in[\theta r-K,\theta r+K])\vert_{x=\MV(u)}\\
+&e^{-V(u)}\ind{\mV(u)\geq0}\P_{V(u)}(\mS_{m-j}\geq0, \MS_{m-j}-S_{m-j}\in[\theta r-K,\theta r+K], \max_{k\leq m-j}(\MS_k-S_k)\leq r+s_r, \MS_{m-j}\geq r+t_r),
\end{align*}
where by \eqref{mSMSMS-Sbd} for $j< m/2$ and $V(u)\leq r/2$, one has
\begin{align}\label{mSMSMS-Sbd+}
&\P_{V(u)}(\mS_{m-j}\geq0, \MS_{m-j}-S_{m-j}\in[\theta r-K,\theta r+K], \max_{k\leq m-j}(\MS_k-S_k)\leq r+s_r, \MS_{m-j}\geq r+t_r)\nonumber\\
\leq & \ind{j\ge m/2}+\ind{j< m/2, V(u)\ge r/2}+c_6(1+V(u))\frac{(1+K^2)(1+r)}{(m-j)^{3/2}}\ind{j< m/2, V(u)\le r/2}.
\end{align}
%Note that if $V(u)\leq r/2$ and $j\le m/2$, one has
%\begin{align*}
%&\P_{V(u)}(\mS_{m-j}\geq0, \MS_{m-j}-S_{m-j}\in[\theta r-K,\theta r+K], \max_{k\leq m-j}(\MS_k-S_k)\leq r+s_r, \MS_{m-j}\geq r+t_r)\\
%\leq & (1+V(u))\frac{(1+2K)(1+\theta r+K)}{(m-j)^{3/2}}.
%\end{align*}
Moreover, by \eqref{mSSbd}, one sees that
\begin{align}\label{ub2for3}
&\E_{\widehat{\Q}_x}\left[\sum_{|z|=m, z\geq u}e^{-V(z)}\widehat{F}_3(z,r,k_0)\vert \mathscr{G}\right]\nonumber\\
\leq &e^{-V(u)}\ind{\mV(u)\geq0, \MV(u)-V(u)\leq r+s_r, \MV(u)\geq r+t_r}[\frac{c_7(1+V(u))(\MV(u)-\theta r+K)(1+2K)}{(m-j)^{3/2}}\wedge 1]\\
&+e^{-V(u)}\ind{\mV(u)\geq0}\ind{j\geq m/2}+e^{-V(u)}\ind{\mV(u)\geq0}\ind{j<m/2, V(u)\ge r/2}\nonumber\\
&+c_6(1+V(u))e^{-V(u)}\ind{\mV(u)\geq0, V(u)\leq r/2, j<m/2}\frac{(1+K^2)(1+ r)}{(m-j)^{3/2}}\nonumber\\
\leq & c_8(1+V(u))^2e^{-V(u)}\ind{\mV(u)\geq0}\frac{(1+K^2)(1+r)}{(m-j)^{3/2}}\ind{j< m/2}+2e^{-V(u)}\ind{\mV(u)\geq0}\ind{j\geq m/2}\nonumber\\
&+e^{-r/4}e^{-V(u)/2}\ind{\mV(u)\geq0}\ind{j<m/2}.
\end{align}
Plugging \eqref{ub2for2} or \eqref{ub2for3} to \eqref{sumOmegaF3} yields that%$\sum_{ar^2-k_0\leq m\leq Ar^2-k_0}e^{-x}\Q_x\left[\widehat{F}_3(z,r,k_0)\left(\frac{r^2}{\delta}(\sum_{j=1}^m\sum_{u\in\Omega(w_j)}\sum_{|z|=m, z\geq u}e^{-V(z)}\widehat{F}_3(z,r,k_0)) \wedge 1\right)\right]$ is bounded by 
\begin{equation}\label{sumOmegaF2bd}
UB_2(A,B,r,2)\leq UB_2^<(A,B,r,2)+UB_2^{\ge}(A,B,r,2)
\end{equation}
where
\begin{align*}
UB_2^<(A,B,r,2):= &\sum_{Ar^2-k_0\leq m\leq Br^2-k_0}e^{-x}\frac{Br^3}{\delta m^{3/2}}\sum_{j=1}^{m/2}\E_{\widehat{\Q}_x}\left[\widehat{F}_2(w_m,r,k_0)\sum_{u\in\Omega(w_j)}(1+V(u))e^{-V(u)}\ind{\mV(u)\geq0}\right]\nonumber\\
UB_2^{\ge}(A,B,r,2):=&\sum_{Ar^2-k_0\leq m\leq Br^2-k_0}e^{-x} \E_{\widehat{\Q}_x}\left[\widehat{F}_2(w_m,r,k_0)\left(\sum_{j=m/2}^m\sum_{u\in\Omega(w_j)}\frac{Br^2}{\delta}e^{-V(u)}\ind{\mV(u)\geq0}\right)\wedge 1\right]
\end{align*}
and that
\begin{equation}\label{sumOmegaF3bd}
UB_2(A,B,r,3)\leq UB_2^{(1)}(A,B,r,3)+UB_2^{(2)}(A,B,r,3)+UB_2^{(3)}(A,B,r,3),
\end{equation}
where
\begin{align*}
%&\Q_x\left[\widehat{F}_3(z,r,k_0)\left(\frac{r^2}{\delta}(\sum_{j=1}^m\sum_{u\in\Omega(w_j)}\sum_{|z|=m, z\geq u}e^{-V(z)}\widehat{F}_3(z,r,k_0)) \wedge 1\right)\right]\\
UB_2^{(1)}(A,B,r,3):=&\sum_{m=Ar^2-k_0}^{Br^2-k_0}e^{-x}\frac{2B(1+K^2)r^3}{\delta m^{3/2}} \E_{\widehat{\Q}_x}\left[\widehat{F}_3(w_m,r,k_0)\left(\sum_{j=1}^{m/2}\sum_{u\in\Omega(w_j)}(1+V(u))^2e^{-V(u)}\ind{\mV(u)\geq0}\right)\right]\nonumber\\
UB_2^{(2)}(A,B,r,3):=&\sum_{Ar^2-k_0\leq m\leq Br^2-k_0}e^{-x}\frac{Br^2}{\delta}e^{-r/4}\E_{\widehat{\Q}_x}\left[\widehat{F}_3(w_m,r,k_0)\left(\sum_{j=1}^{m/2}\sum_{u\in\Omega(w_j)}e^{-V(u)/2}\ind{\mV(u)\geq0}\right)\right]\nonumber\\
UB_2^{(3)}(A,B,r,3):=&\sum_{Ar^2-k_0\leq m\leq Br^2-k_0}e^{-x} \E_{\widehat{\Q}_x}\left[\widehat{F}_3(w_m,r,k_0)\left(\sum_{j=m/2}^m\sum_{u\in\Omega(w_j)}2\frac{Br^2}{\delta}e^{-V(u)}\ind{\mV(u)\geq0}\right)\wedge 1\right].
%&+\Q_x\left[\widehat{F}_3(z,r,k_0)\left(\frac{r^2}{\delta}(\sum_{j=1}^{m/2}\sum_{u\in\Omega(w_j)}(1+K^2)(1+V(u))e^{-V(u)}\ind{\mV(u)\geq0, V(u)\leq r/2}\frac{1+\theta r}{(m-j)^{3/2}}\right)\right]
\end{align*}
In the rest part, we will check that all these terms are $o_x(1)\Ren(x)e^{-x}$ for $r\to\infty$ and then $x\gg1$.

We will first treat $UB_2^<(A,B,r,2)$, $UB_2^{(1)}(A,B,r,3)$ and $UB_2^{(2)}(A,B,r,3)$ in the similar way. For any $u\in\T$, let $\Delta V(u)=V(u)-V(u^*)$ be its displacement. Write $\Delta_+V(u)$ for $\Delta V(u) \vee 0$. Then, 
\[
\sum_{u\in\Omega(w_j)}(1+V(u))e^{-V(u)}\ind{\mV(u)\geq0}\leq\sum_{u\in\Omega(w_j)}(1+V(u))^2e^{-V(u)}\ind{\mV(u)\geq0}\leq e^{-V(w_{j-1})/2}\ind{\mV(w_{j-1})\geq0} V_j^+,
\]
with $V_j^+:=\sum_{u\in\Omega(w_j)}e^{-\Delta V(u)/2}$. Consequently,
\begin{multline*}
UB_2^<(A,B,r,2)\\\leq 
\sum_{m=Ar^2-k_0}^{Br^2-k_0}e^{-x}\frac{Br^3}{\delta m^{3/2}}\sum_{j=1}^{m/2}\E_{\widehat{\Q}_x}\left[e^{V(w_m)-(1-\theta)r-b}\ind{\mV(w_m)\geq0, \MV(w_m)\leq r+t_r, V(w_m)\leq (1-\theta)r+b}e^{-V(w_{j-1})/2}V_j^{+}\right],
\end{multline*}
which by Markov property at time $j$ and then by \eqref{eSmSSsmallbd}, is bounded by 
\begin{align*}
%&UB_2^<(A,B,r,2)\\
%\leq &\sum_{m=Ar^2-k_0}^{Br^2-k_0}e^{-x}\frac{Br^3}{\delta m^{3/2}}\sum_{j=1}^{m/2}\E_{\widehat{\Q}_x}\left[e^{V(w_m)-(1-\theta)r-b}\ind{\mV(w_m)\geq0, \MV(w_m)\leq r+t_r, V(w_m)\leq (1-\theta)r+b}e^{-V(w_{j-1})/2}V_j^{+}\right]\\
&\sum_{m=Ar^2-k_0}^{Br^2-k_0}e^{-x}\frac{Br^3}{\delta m^{3/2}}\sum_{j=1}^{m/2}\E_{\widehat{\Q}_x}\left[\ind{\mV(w_j)\geq0}e^{-V(w_{j-1})/2}V_j^+\E_{V(w_j)}[e^{S_{m-j}-(1-\theta)r+b}\ind{\mS_{m-j}\geq0, S_{m-j}\leq (1-\theta)r-b}]\right]\\
\leq &\sum_{m=Ar^2-k_0}^{Br^2-k_0}e^{-x}\frac{Br^3}{\delta m^{3/2}}\sum_{j=1}^{m/2}\E_{\widehat{\Q}_x}\left[\ind{\mV(w_j)\geq0}e^{-V(w_{j-1})/2}V_j^+(1+V(w_j))\right]\frac{c_9r}{(m-j)^{3/2}}.
%\leq &\sum_{ar^2-k_0\leq m\leq Ar^2-k_0}e^{-x}\frac{Ar}{\delta m^{3/2}}\sum_{j=1}^{m/2}E_x\left[\ind{\mS_{j-1}\geq0}e^{-S_{j-1}/4}\right]\leq e^{-x}\E_x\left[\sum_{j\geq0}e^{-S_j/4}\ind{\mS_j\geq0}\right]=o_x(1)\Ren(x)e^{-x},
\end{align*}
Here $(1+V(w_j))\ind{\mV(w_j)\geq0}\le (1+V(w_{j-1}))\ind{\mV(w_{j-1})\geq0}(1+\Delta_+ V(w_j))$
and then Markov property at time $j-1$ implies that
\begin{multline*}
\E_{\widehat{\Q}_x}\left[\ind{\mV(w_j)\geq0}e^{-V(w_{j-1})/2}V_j^+(1+V(w_j))\right]\\
\leq \E_{\widehat{\Q}_x}\left[\ind{\mV(w_{j-1})\geq0}(1+V(w_{j-1}))e^{-V(w_{j-1})/2}\right]\E_{\widehat{\Q}}[V_1^+(1+V_+(w_1))],
\end{multline*}
where by Proposition \ref{spinedecomp},
\[
\E_{\widehat{\Q}}[V_1^+(1+V_+(w_1))]=\E\left[\sum_{|u|=1}(1+V_+(u))e^{-V(u)}\sum_{|v|=1,v\neq u}e^{-V(v)/2}\right].%\leq \E\left[(\sum_{|u|=1}(1+V_+(v))e^{-V(v)})^2\right].
\]
By Cauchy-Schwartz inequality and \eqref{MomCond}, 
\begin{align*}
\E_{\widehat{\Q}}[V_1^+(1+V_+(w_1))]^2\leq &\E\left[(\sum_{|u|=1}(1+V_+(u))e^{-V(u)})^2\right]\E\left[(\sum_{|u|=1}e^{-V(u)/2})^2\right]\\
\leq &\E\left[(\sum_{|u|=1}(1+V_+(u))e^{-V(u)})^2\right]\E\left[N \sum_{|u|=1}e^{-V(u)}\right]\\
\leq &\E\left[(\sum_{|u|=1}(1+V_+(u))e^{-V(u)})^2\right]\E[N^2]\E\left[(\sum_{|u|=1}e^{-V(u)})^2\right]<\infty.
\end{align*}
Similarly, we also have $\E_{\widehat{\Q}}[V_1^+(1+V_+(w_1))^2]<\infty$.
It follows that
\begin{align*}
&UB_2^<(A,B,r,2)\\
\leq &\sum_{m=Ar^2-k_0}^{Br^2-k_0}e^{-x}\frac{Br^3}{\delta m^{3/2}}\sum_{j=1}^{m/2}\E_{\widehat{\Q}_x}\left[\ind{\mV(w_{j-1})\geq0}(1+V(w_{j-1}))e^{-V(w_{j-1})/2}\right]\frac{c_9r}{(m-j)^{3/2}}\\
\leq &\sum_{Ar^2-k_0\leq m\leq Br^2-k_0}e^{-x}\frac{c_{10}r^4}{\delta m^{3}}\sum_{j=1}^{m/2}\E_x\left[\ind{\mS_{j-1}\geq0}e^{-S_{j-1}/4}\right]\leq e^{-x}\E_x\left[\sum_{j\geq0}e^{-S_j/4}\ind{\mS_j\geq0}\right],%=o_x(1)\Ren(x)e^{-x}
\end{align*}
which by \eqref{sumeSmSbd} shows that $UB_2^<(A,B,r,2)=o_x(1)\Ren(x)e^{-x}$. For $UB_2^{(1)}(A,B,r,3)$, as $\cf(t)\leq 1$, we have
\begin{align}\label{elementUB2for3}
& \E_{\widehat{\Q}_x}\left[\widehat{F}_3(w_m,r,k_0)\left(\sum_{j=1}^{m/2}\sum_{u\in\Omega(w_j)}(1+V(u))^2e^{-V(u)}\ind{\mV(u)\geq0}\right)\right]\\
\leq &\sum_{j=1}^{m/2}\E_{\widehat{\Q}_x}\left[\ind{\mV(w_m)\geq0, \MV(w_m)\geq r+t_r, \max_{k\leq m}(\MV(k)-V(k))\leq r+s_r, \MV(w_m)-V(w_m)\in[\theta r-K,\theta r+K]}e^{-V(w_{j-1})/2}V_{j}^+\right]\nonumber\\
=&\sum_{j=1}^{m/2}\sum_{i=1}^{m}\E_{\widehat{\Q}_x}\left[\ind{\mV(w_m)\geq0, \tau_m^V=i, \MV(w_m)\geq r+t_r, \max_{k\leq m}(\MV(k)-V(k))\leq r+s_r, \MV(w_m)-V(w_m)\in[\theta r-K,\theta r+K]}e^{-V(w_{j-1})/2}V_{j}^+\right]\nonumber
\end{align}
where $\tau^V_m:=\inf\{i\leq m: V(w_i)=\MV(w_m)\}$. 

On the one hand, if $\tau^V_m\ge j+1$, by Markov property at time $j$, one sees that
\begin{align*}
&\sum_{j=1}^{m/2}\sum_{i=j+1}^{m}\E_{\widehat{\Q}_x}\left[\ind{\mV(w_m)\geq0, \tau_m^V=i, \MV(w_m)\geq r+t_r, \max_{k\leq m}(\MV(k)-V(k))\leq r+s_r, \MV(w_m)-V(w_m)\in[\theta r-K,\theta r+K]}e^{-V(w_{j-1})/2}V_{j}^+\right]\\
 &\leq \sum_{j=1}^{m/2}\E_{\widehat{\Q}_x}\left[\ind{\mV(w_{j})\geq0}e^{-V(w_{j-1})/2}V_{j}^+\E_{V(w_{j})}\left[\ind{\mS_{m-j}\geq0, \MS_{m-j}\geq r+t_r, \max_{k\leq m-j}(\MS_k-S_k)\leq r+s_r, \MS_{m-j}-S_{m-j}\in[\theta r-K,\theta r+K]}\right]\right],
%\leq & e^{-r/8}\Q_x\left[\ind{\mV(w_{j})\geq0, \max_{k\leq j}(\MV(w_k)-V(w_k))\leq r+s_r}e^{-V(w_{j})/4}\right]\E[\sum_{|u|=1}e^{-\frac{3}{4}V(u)}\sum_{|v|=1,v\neq u}(1+V_+(v))e^{-V(v)}]\\
%&+\Q_x\left[\ind{\mV(w_{j})\geq0, \max_{k\leq j}(\MV(w_k)-V(w_k))\leq r+s_r}e^{-V(w_{j})/4}\right]\E[(\sum_{|v|=1}(1+V_+(v))e^{-V(v)})^2]\frac{(1+K^2)r}{m^{3/2}}
\end{align*}
which by \eqref{mSMSMS-Sbd} and \eqref{sumeSmSbd} is bounded by
\begin{align*}
&c_{11}\sum_{j=1}^{m/2}\E_{\widehat{\Q}_x}\left[\ind{\mV(w_{j})\geq0}e^{-V(w_{j-1})/2}V_{j}^+(1+V(w_j))\frac{(1+K^2)(1+r)}{(m-j)^{3/2}}\right]\\
\leq &\frac{c_{11}(1+K^2)(1+r)}{m^{3/2}}\E_x\left[\sum_{j\geq1}(1+S_{j-1})e^{-S_{j-1}/2}\ind{\mS_{j-1}\geq0}\right]\E_{\widehat{\Q}}\left[(1+V_+(w_1)V_1^+\right]\\
=&\frac{c_{11}(1+K^2)(1+r)}{m^{3/2}}o_x(1)\Ren(x).
\end{align*}
On the other hand, if $\tau^V_m\leq j$, again by Markov property at time $j$,
\begin{align*}
&\sum_{j=1}^{m/2}\sum_{i=1}^{j}\E_{\widehat{\Q}_x}\left[\ind{\mV(w_m)\geq0, \tau_m^V=i, \MV(w_m)\geq r+t_r, \max_{k\leq m}(\MV(k)-V(k))\leq r+s_r, \MV(w_m)-V(w_m)\in[\theta r-K,\theta r+K]}e^{-V(w_{j-1})/2}V_{j}^+\right]\\
\leq &\sum_{j=1}^{m/2}\E_{\widehat{\Q}_x}\left[\ind{\mV(w_{j})\geq0}e^{-V(w_{j-1})/2}V_{j}^+\P_{V(w_j)}(\mS_{m-j}\geq0, x-S_{m-j}\in[\theta r-K,\theta r+K])\vert_{x=\MV(w_j)}\right]
%\leq & \Q_x\left[\ind{\mV(w_{j+1})\geq0, \tau_{j+1}^V=i, \MV(w_{j+1})\geq r+t_r, \max_{k\leq j+1}(\MV(w_k)-V(w_k))\leq r+s_r}e^{-V(w_j)/2}V_{j+1}^+\frac{(1+K^2)(1+V(w_{j+1}))(1+V(w_i)-\theta r)}{(m-j)^{3/2}}\right]\\ 
\end{align*}
where by \eqref{mSSbd},
\[
\P_{V(w_j)}(\mS_{m-j}\geq0, x-S_{m-j}\in[\theta r-K,\theta r+K])\vert_{x=\MV(w_j)}\leq \frac{c_{12}(1+K^2)(1+V(w_{j}))(1+\MV(w_j)-\theta r)}{(m-j)^{3/2}}
\]
which is bounded by $\frac{c_{13}(1+K^2)(1+V(w_{j}))^2(1+r)}{(m-j)^{3/2}}$ because $\MV(w_{j})-V(w_{j})\leq r+s_r$. Again by Markov property at time $j-1$ and \eqref{sumeSmSbd}, we get that
\begin{multline*}
\sum_{j=1}^{m/2}\sum_{i=1}^{j}\E_{\widehat{\Q}_x}\left[\ind{\mV(w_m)\geq0, \tau_m^V=i, \MV(w_m)\geq r+t_r, \max_{k\leq m}(\MV(k)-V(k))\leq r+s_r, \MV(w_m)-V(w_m)\in[\theta r-K,\theta r+K]}e^{-V(w_{j-1})/2}V_{j}^+\right]\\
\leq \frac{c_{13}(1+K^2)(1+r)}{m^{3/2}}\E_x\left[\sum_{j\geq1}(1+S_{j-1})^2e^{-S_{j-1}/2}\ind{\MS_{j-1}\geq0}\right]\E_{\widehat{\Q}}[V_1^+(1+V_+(w_1))^2]\\
=\frac{c_{13}(1+K^2)(1+r)}{m^{3/2}}o_x(1)\Ren(x).
\end{multline*}
Combining these inequalities and going back to \eqref{elementUB2for3},  we have 
\[
\E_{\widehat{\Q}_x}\left[\widehat{F}_3(w_m,r,k_0)\left(\sum_{j=1}^{m/2}\sum_{u\in\Omega(w_j)}(1+V(u))^2e^{-V(u)}\ind{\mV(u)\geq0}\right)\right]\leq \frac{c_{14}(1+K^2)(1+r)}{m^{3/2}}o_x(1)\Ren(x).
\] This implies that
\begin{equation}\label{UB2for31}
UB_2^{(1)}(A,B,r,3)
\leq \sum_{m=Ar^2-k_0}^{Br^2-k_0}e^{-x}\frac{2B(1+K^2)r^3}{\delta m^{3/2}} \frac{c_{14}(1+K^2)(1+r)}{m^{3/2}}o_x(1)\Ren(x)=o_x(1)\Ren(x)e^{-x}.
\end{equation}

Note that $\sum_{u\in\Omega(w_j)}e^{-V(u)/2}\leq e^{-V(w_{j-1})/2}V_j^+$. So similarly as above, 
\begin{equation}\label{UB2for32}
UB_2^{(2)}(A,B,r,3)=o_x(1)\Ren(x)e^{-x}.
\end{equation}
Let us turn to bound $UB_2^{(3)}(A,B,r,3)$ in \eqref{sumOmegaF3bd}. Let $\mV(w_{[j,m]}):=\min_{j\leq k\le m}V(w_k)$ and $S_{[j,m]}:=\min_{j\le k\le m}S_{k}$. Observe that
\begin{align}\label{UB2for33}
&UB_2^{(3)}(A,B,r,3)=\sum_{Ar^2-k_0\leq m\leq Br^2-k_0}e^{-x}\E_{\widehat{\Q}_x}\left[\widehat{F}_3(w_m,r,k_0)\left(\sum_{j=m/2}^m\sum_{u\in\Omega(w_j)}2 \frac{Br^2}{\delta}e^{-V(u)}\ind{\mV(u)\geq0}\right)\wedge 1\right]\nonumber\\
\leq &\sum_{Ar^2-k_0\leq m\leq Br^2-k_0}e^{-x}\E_{\widehat{\Q}_x}\left[\widehat{F}_3(w_m,r,k_0)\ind{\mV(w_{[m/2-1,m]})\leq 6\log r}\right]\nonumber\\
&+\sum_{Ar^2-k_0\leq m\leq Br^2-k_0}e^{-x} \frac{2Br^2}{\delta}\sum_{j=m/2}^m \E_{\widehat{\Q}_x}\left[\ind{\mV(w_n)\geq 0,\mV(w_{[m/2-1,m]})\geq 6\log r}e^{-V(w_{j-1})}\sum_{u\in\Omega(w_j)}e^{-\Delta V(u)}\right].
%\sum_{Ar^2-k_0\leq m\leq Br^2-k_0}e^{-x} \frac{2Ar^2}{\delta}\sum_{j=m/2}^m \P_x(\mS_j\geq0)e^{-6\log r}.
\end{align}
On the one hand, by Proposition \ref{spinedecomp},
\begin{align*}
&\E_{\widehat{\Q}_x}\left[\widehat{F}_3(z,r,k_0)\ind{\mV(w_{[m/2-1,m]})\leq 6\log r}\right]\\
\leq &\E_{\widehat{\Q}_x}\left[\ind{\mV(w_n)\ge0, \MV(w_n)\ge r+t_r, \max_{k\leq n}(\MV(w_k)-V(w_k))\le r+s_r, \MV(w_n)-V(w_n)\in[\theta r-K,\theta r+K]}\ind{\mV(w_{[m/2-1,m]})\leq 6\log r}\right]\\
\leq &\sum_{j=m/2-1}^{m-1}\P_x(\mS_m\geq0, \MS_m\ge r+t_r, \mS_{[m/2-1,m]}=S_j\le 6\log r, \MS_m-S_m\in[\theta r-K,\theta r+K], \max_{k\leq m}(\MS_k-S_k)\leq r+s_r)
%\leq & \sum_{j=m/2-1}^{m-1}\P_x(\mS_j\ge 0, S_j\le 6\log r)\P(\mS_{m-j}\ge0, \max_{k\le m-j}(\MS_k-S_k)\le r+s_r, \MS_{m-j}-S_{m-j}\in[\theta r-K,\theta r+K])
\end{align*}
Recall that $t_r>s_r+6\log r$. So $\MS_m>\MS_j$. By Markov property at time $j$, this leads to the following inequality:
\begin{align*}
&\sum_{j=m/2-1}^{m-1}\P_x(\mS_m\geq0, \MS_m\ge r+t_r, \mS_{[m/2-1,m]}=S_j\le 6\log r, \MS_m-S_m\in[\theta r-K,\theta r+K], \max_{k\leq m}(\MS_k-S_k)\leq r+s_r)\\
\leq & \sum_{j=m/2-1}^{m-1}\P_x(\mS_j\ge 0, S_j\le 6\log r)\P(\mS_{m-j}\ge0, \max_{k\le m-j}(\MS_k-S_k)\le r+s_r, \MS_{m-j}-S_{m-j}\in[\theta r-K,\theta r+K])\\
\leq &\frac{c_{15}(1+x)(6\log r)^2}{m^{3/2}} \sum_{j=m/2-1}^{m-1}\P(\mS_{m-j}\ge0, \max_{k\le m-j}(\MS_k-S_k)\le r+s_r, \MS_{m-j}-S_{m-j}\in[\theta r-K,\theta r+K])
\end{align*}
where the last inequality comes from \eqref{mSSbd}. Then by \eqref{summSMSMS-Sbd}, one gets that
\[
\E_{\widehat{\Q}_x}\left[\widehat{F}_3(z,r,k_0)\ind{\mV(w_{[m/2-1,m]})\leq 6\log r}\right]\leq \frac{c_{16}(1+K)(1+x)(6\log r)^2}{m^{3/2}}
\]
which ensures that $\sum_{Ar^2-k_0\leq m\leq Br^2-k_0}e^{-x}\E_{\widehat{\Q}_x}\left[\widehat{F}_3(z,r,k_0)\ind{\mV(w_{[m/2-1,m]})\leq 6\log r}\right]=o_r(1)\Ren(x)e^{-x}$.

On the other hand, by Markov property at time $j$,
\begin{align*}
&\sum_{Ar^2-k_0\leq m\leq Br^2-k_0}e^{-x} \frac{2Br^2}{\delta}\sum_{j=m/2}^m \E_{\widehat{\Q}_x}\left[\ind{\mV(w_n)\geq 0,\mV(w_{[m/2-1,m]})\geq 6\log r}e^{-V(w_{j-1})}\sum_{u\in\Omega(w_j)}e^{-\Delta V(u)}\right]\\
\leq &\sum_{Ar^2-k_0\leq m\leq Br^2-k_0}e^{-x} \frac{2Br^{-2}}{\delta}\sum_{j=m/2}^m \E_{\widehat{\Q}_x}\left[\ind{\mV(w_{j-1})\ge 0}e^{-V(w_{j-1})/3}\right]\E_{\widehat{\Q}_x}\left[\sum_{u\in\Omega(w_1)}e^{-V(u)}\right]
\end{align*}
where by Proposition \ref{spinedecomp} and \eqref{sumeSmSbd}, 
\[
\sum_{j=m/2}^m \E_{\widehat{\Q}_x}\left[\ind{\mV(w_{j-1})\ge 0}e^{-V(w_{j-1})/3}\right]=\sum_{j=m/2}^m \E_{x}\left[\ind{\mS_{j-1}\ge 0}e^{-S_{j-1}/3}\right]\leq \E_x[\sum_{j\geq0}e^{-S_j/4}]=o_x(1)\Ren(x).
\]
Moreover by Proposition \ref{spinedecomp} and \eqref{MomCond}, 
\begin{equation*}
\E_{\widehat{\Q}_x}\left[\sum_{u\in\Omega(w_1)}e^{-V(u)}\right]\leq \E\left[(\sum_{|u|=1}e^{-V(u)})^2\right]<\infty.
\end{equation*}
We thus deduce that
\begin{equation}\label{UB2for3c}
\sum_{Ar^2-k_0\leq m\leq Br^2-k_0}e^{-x} \frac{2Br^2}{\delta}\sum_{j=m/2}^m \E_{\widehat{\Q}_x}\left[\ind{\mV(w_n)\geq 0,\mV(w_{[m/2-1,m]})\geq 6\log r}e^{-V(w_{j-1})}\sum_{u\in\Omega(w_j)}e^{-\Delta V(u)}\right]=o_x(1)\Ren(x)e^{-x}.
\end{equation}
Going back to \eqref{UB2for33}, we obtain that $UB_2^{(3)}(A,B,r,3)=o_x(1)\Ren(x)e^{-x}$.

It remains to bound $UB_2^{\ge}(A,B,r,2)$ in \eqref{sumOmegaF2bd}. Similarly as above, observe that
\begin{align*}
&UB_2^{\ge}(A,B,r,2)=\sum_{Ar^2-k_0\leq m\leq Br^2-k_0}e^{-x} \E_{\widehat{\Q}_x}\left[\widehat{F}_2(w_k,r,k_0)\left(\sum_{j=m/2}^m\sum_{u\in\Omega(w_j)}\frac{Br^2}{\delta}e^{-V(u)}\ind{\mV(u)\geq0}\right)\wedge 1\right]\\
\leq &\sum_{Ar^2-k_0\leq m\leq Br^2-k_0}e^{-x} \E_{\widehat{\Q}_x}\left[\widehat{F}_2(w_k,r,k_0)\ind{\mV(w_{[m/2-1,m]})\le 6\log r}\right]\\
&+\sum_{Ar^2-k_0\leq m\leq Br^2-k_0}e^{-x} \sum_{j=m/2}^m\frac{Br^2}{\delta}\E_{\widehat{\Q}_x}\left[\ind{\mV(w_n)\geq 0,\mV(w_{[m/2-1,m]})\geq 6\log r}e^{-V(w_{j-1})}\sum_{u\in\Omega(w_j)}e^{-\Delta V(u)}\right]\\
=&\sum_{Ar^2-k_0\leq m\leq Br^2-k_0}e^{-x} \E_{\widehat{\Q}_x}\left[\widehat{F}_2(w_k,r,k_0)\ind{\mV(w_{[m/2-1,m]})\le 6\log r}\right]+o_x(1)\Ren(x)e^{-x},
\end{align*}
where the last line comes from \eqref{UB2for3c}.

For the first term on the right hand side, by Proposition \ref{spinedecomp},
\begin{align*}
&\sum_{Ar^2-k_0\leq m\leq Br^2-k_0}e^{-x} \E_{\widehat{\Q}_x}\left[\widehat{F}_2(w_k,r,k_0)\ind{\mV(w_{[m/2-1,m]})\le 6\log r}\right]\\
\leq &\sum_{Ar^2-k_0\leq m\leq Br^2-k_0}e^{-x}\sum_{j=m/2-1}^{m} \E_x\left[e^{S_m-(1-\theta)r-b}\ind{\mS_m\geq0, S_m\le (1-\theta )r+b, S_j=\mS_{[m/2-1,m]}\le 6\log r}\right]
\end{align*}
which by Markov property at time $j$, is bounded by
\begin{align*}
&\sum_{Ar^2-k_0\leq m\leq Br^2-k_0}e^{-x}\sum_{j=m/2-1}^{m} \E_x\left[\ind{\mS_j\ge0, S_j\le 6\log r}\E[e^{S_{m-j}-[(1-\theta)r+b-v]}\ind{\mS_{m-j}\ge0, S_{m-j}\leq (1-\theta)r+b-v}]\vert_{v=S_j}\right]\\
\leq &\sum_{Ar^2-k_0\leq m\leq Br^2-k_0}e^{-x}\sum_{k=0}^{6\log r}\sum_{j=m/2-1}^{m}\P_x(\mS_j\ge0,S_j\in[k,k+1])e\E[e^{S_{m-j}-[(1-\theta)r+b-k]}\ind{\mS_{m-j}\ge0, S_{m-j}\leq (1-\theta)r+b-k}].
\end{align*}
Then by \eqref{mSSbd} and by \eqref{sumeSmSSsmallbd}, we have
\begin{align*}
&\sum_{Ar^2-k_0\leq m\leq Br^2-k_0}e^{-x} \E_{\widehat{\Q}_x}\left[\widehat{F}_2(w_k,r,k_0)\ind{\mV(w_{[m/2-1,m]})\le 6\log r}\right]\\
\leq &\sum_{Ar^2-k_0\leq m\leq Br^2-k_0}e^{1-x}\sum_{k=0}^{6\log r}\frac{c_{17}(1+x)(2+k)}{m^{3/2}}\sum_{j=m/2-1}^{m}\E[e^{S_{m-j}-[(1-\theta)r+b-k]}\ind{\mS_{m-j}\ge0, S_{m-j}\leq (1-\theta)r+b-k}]\\
\leq &c\sum_{Ar^2-k_0\leq m\leq Br^2-k_0}e^{1-x}\frac{c_{18}(1+x)(6\log r)^2}{m^{3/2}}=o_r(1)\Ren(x)e^{-x}.
\end{align*}
We hence completes the proof of \eqref{1momchi-}.
%%%%%%%%%%%%%%%%%%%%%%%%%%%%%%%%%%%%%%%

\end{proof}
\section{Proof of Lemmas \ref{smallpart}, \ref{largepartbadtime}, \ref{smallpart1ex}, \ref{smallpart1ex+}, \ref{smallpart1exL} and \ref{1exvariance}}\label{lems}

\begin{proof}[Proof of Lemma \ref{smallpart}]
 It suffices to show that 
\begin{align}\label{meansmallpartbadwalk}
\e_{\eqref{meansmallpartbadwalk}}:=\e\left[\sum_{k=1}^{c_0(\log n)^3}\sum_{|x|=k}\ind{x\in A_n(-a_n, -b)}\ind{\overline{L}_x(\tau_n)<n^\theta\textrm{ or } E_x^{(n)}\leq 1}\right]=&o(n^{1-\theta}).\\
\e_{\eqref{meansmallpartbadV}}:=\e\left[\sum_{k=1}^{c_0(\log n)^3}\sum_{|x|=k}\ind{x\notin A_n^+(a_n,b)}\ind{\mV(x)\geq-\alpha}\ind{\overline{L}_x(\tau_n)\geq n^\theta, E_x^{(n)}\geq 2}\right]=&o(n^{1-\theta}).\label{meansmallpartbadV}
\end{align}
\textbf{Proof of \eqref{meansmallpartbadwalk}.} Observe that
\begin{align*}
%&\e\left[\sum_{k=1}^{c_0(\log n)^3}\sum_{|x|=k}\ind{x\in A_n(-a_n, -b)}\ind{\overline{L}_x(\tau_n)<n^\theta\textrm{ or } E_x^{(n)}\leq 1}\right]\\
\e_{\eqref{meansmallpartbadwalk}}&=\e\left[\sum_{k=1}^{c_0(\log n)^3}\sum_{|x|=k}\ind{x\in A_n(-a_n, -b)}\left(\ind{\overline{L}_x(\tau_n)<n^\theta}+\ind{\overline{L}_x(\tau_n)\geq n^\theta, E_x^{(n)}\leq 1}\right)\right]\\
&=\E\left[\sum_{k=1}^{c_0(\log n)^3}\sum_{|x|=k}\ind{x\in A_n(-a_n, -b)}\left\{\p^\en\left(\overline{L}_x(\tau_n)<n^\theta\right)+\p^\en\left(\overline{L}_x(\tau_n)\geq n^\theta, E_x^{(n)}=1\right)\right\}\right]
\end{align*}
So \eqref{meansmallpartbadwalk} follows the following convergences:
\begin{align}
\e_{\eqref{badwalkmean1}}:=\E\left[\sum_{k=1}^{c_0(\log n)^3}\sum_{|x|=k}\ind{x\in A_n(-a_n, -b)}\p^\en\left(\overline{L}_x(\tau_n)<n^\theta\right)\right]=&o(n^{1-\theta});\label{badwalkmean1}\\
\e_{\eqref{badwalkmean2}}:=\E\left[\sum_{k=1}^{c_0(\log n)^3}\sum_{|x|=k}\ind{x\in A_n(-a_n, -b)}\p^\en\left(\overline{L}_x(\tau_n)\geq n^\theta, E_x^{(n)}=1\right)\right]=&o(n^{1-\theta}).\label{badwalkmean2}.
\end{align}
Note that for $x\in A_n(-a_n,-b)$, we have $n^\theta(1-b_x)=\frac{n^\theta}{H_x}\leq e^{-b}\frac{n e^{-V(x)}}{H_x}=e^{-b} na_x$. So, by \eqref{sumGeosmall} with $\lambda=b$,
\begin{equation}\label{upLsmall}
\p^\en(\overline{L}_x(\tau_n)<n^\theta)\leq e^{-\lambda(\frac{na_x}{1+\lambda}-b_xn^\theta)}\leq e^{-c_{19}(b) \frac{n^\theta}{H_x}},
\end{equation}
with $c_{19}(b):=b(\frac{e^b}{1+b}-1)>0$. Moreover, we see that for $x\in A_n(-a_n,-b)$ with $|x|\leq c_0(\log n)^3$,
\begin{equation}\label{upH}
H_x\leq |x|e^{\MV(x)-V(x)}\leq c_0(\log n)^3 n^{\theta} e^{-a_n}.
\end{equation}
Plugging it into \eqref{upLsmall} yields that for $a_n=a\log\log n$ with $a>3$,
\[
\p^\en(\overline{L}_x(\tau_n)<n^\theta)\leq e^{-c_{20} (\log n)^{a-3}}.
\]
This implies that
\[
%\e\left[\sum_{k=1}^{ c_0(\log n)^3}\sum_{|x|=k}\ind{x\in A_n(-a_n, -b)}\p^\en\left(\overline{L}_x(\tau_n)<n^\theta\right)\right]
\e_{\eqref{badwalkmean1}}\leq  e^{-c_2 (\log n)^{a-3}}\E\left[\sum_{k=1}^{c_0(\log n)^3}\sum_{|x|=k}\ind{x\in A_n(-a_n, -b)}\right]
\]
Then by Many-to-One Lemma, one sees that
\begin{align*}
%&\e\left[\sum_{k=1}^{ c_0(\log n)^3}\sum_{|x|=k}\ind{x\in A_n(-a_n, -b)}\p^\en\left(\overline{L}_x(\tau_n)<n^\theta\right)\right]\\
%\leq  e^{-c_1 (\log n)^{a-3}}\e\left[\sum_{|x|\leq c_0(\log n)^3}\ind{x\in A_n(-a_n, -b)}\right]\\
\e_{\eqref{badwalkmean1}}\leq&e^{-c_2 (\log n)^{a-3}}\sum_{k=1}^{c_0(\log n)^3}\E\left[e^{S_k}\ind{\MS_k-S_k\leq \theta \log n-a_n, S_k\leq (1-\theta)\log n-b}\right]\\
\leq & e^{-c_{20} (\log n)^{a-3}} n^{1-\theta} c_0 (\log n)^3=o(n^{1-\theta}),
\end{align*}
which shows \eqref{badwalkmean1}. On the other hand, for $n\geq 2$ and $x\in A_n(-a_n,-b)$, we could get that
\[
\p^\en\left(\overline{L}_x(\tau_n)\geq n^\theta, E_x^{(n)}=1\right)=n a_x b_x^{n^\theta-1}(1-a_x)^{n-1}\leq n^{1-\theta} e^{-V(x)}e^{-c_{21}\frac{n^\theta}{H_x}}.
\]
In fact, for $c_{21}=1-\frac{1}{2^\theta}$,
\[
n a_x b_x^{n^\theta-1}(1-a_x)^{n-1}=n^{1-\theta}e^{-V(x)} (1-\frac{1}{H_x})^{n^\theta-1} \frac{n^\theta}{H_x}(1-\frac{e^{-V(x)}}{H_x})^{n-1}\leq n^{1-\theta}e^{-V(x)} e^{-c_{21} \frac{n^\theta}{H_x}} \frac{n^\theta}{H_x}e^{-\frac{n}{2}e^{-V(x)}/H_x}
\]
where $ne^{-V(x)}\geq e^b n^\theta\geq n^\theta$ for $x\in A_n(-a_n,-b)$. So $\frac{n^\theta}{H_x}e^{-\frac{n}{2}e^{-V(x)}/H_x}\leq \frac{n^\theta}{H_x}e^{-\frac{n^\theta}{2H_x}}\leq\sup_{t\geq0}te^{-t/2}<1$. Therefore, $\p^\en\left(\overline{L}_x(\tau_n)\geq n^\theta, E_x^{(n)}=1\right)\leq n^{1-\theta} e^{-V(x)}e^{-c_{21}\frac{n^\theta}{H_x}}$. Further, for $x\in A_n(-a_n,-b)$, by \eqref{upH}, we get that
\[
\p^\en\left(\overline{L}_x(\tau_n)\geq n^\theta, E_x^{(n)}=1\right)\leq n^{1-\theta} e^{-V(x)}e^{-c_{22}(\log n)^{a-3}}.
\]
Consequently, 
\begin{equation*}
%\e\left[\sum_{k=1}^{c_0(\log n)^3}\sum_{|x|=k}\ind{x\in A_n(-a_n, -b)}\p^\en\left(\overline{L}_x(\tau_n)\geq n^\theta, E_x^{(n)}=1\right)\right]\\
\e_{\eqref{badwalkmean2}}\leq  n^{1-\theta }e^{-c_{22}(\log n)^{a-3}}\E\left[\sum_{k=1}^{ c_0(\log n)^3}\sum_{|x|=k}\ind{x\in A_n(-a_n, -b)}e^{-V(x)}\right],
\end{equation*}
which, by the Many-to-One Lemma, leads to
\begin{align*}
\e_{\eqref{badwalkmean2}}\leq& n^{1-\theta }e^{-c_{22}(\log n)^{a-3}}\sum_{k=1}^{c_0(\log n)^3}\P\left(\MS_k-S_k\leq \theta \log n-a_n, S_k\leq (1-\theta)\log n-b\right)\\
\leq & n^{1-\theta }e^{-c_{22}(\log n)^{a-3}}c_0(\log n)^3=o(n^{1-\theta}),
\end{align*}
which concludes \eqref{badwalkmean2}.

\noindent\textbf{Proof of \eqref{meansmallpartbadV}.} Observe that
\begin{align*}
%&\e\left[\sum_{k=1}^{c_0(\log n)^3}\sum_{|x|=k}\ind{x\notin A_n^+(a_n,b)}\ind{\mV(x)\geq-\alpha}\ind{\overline{L}_x(\tau_n)\geq n^\theta, E_x^{(n)}\geq 2}\right]\\
\e_{\eqref{meansmallpartbadV}}\leq&\E\left[\sum_{k=1}^{c_0(\log n)^3}\sum_{|x|=k}\ind{\MV(x)>\log n+a_n}\p^\en(\overline{L}_x(\tau_n)\geq n^\theta, E_x^{(n)}\geq 2)\right]\\
&\qquad+\E\left[\sum_{k=1}^{c_0(\log n)^3}\sum_{|x|=k}\ind{\mV(x)\geq-\alpha, \MV(x)\leq \log n+a_n, V(x)>(1-\theta)\log n+b}\p^\en(\overline{L}_x(\tau_n)\geq n^\theta, E_x^{(n)}\geq 2)\right].
\end{align*}
So, to get \eqref{meansmallpartbadV}, we only need to show that 
\begin{align}
\e_{\eqref{badVmean1}}:=\E\left[\sum_{k=1}^{c_0(\log n)^3}\sum_{|x|=k}\ind{\MV(x)>\log n+a_n}\p^\en(\overline{L}_x(\tau_n)\geq n^\theta, E_x^{(n)}\geq 2)\right]=&o(n^{1-\theta});\label{badVmean1}\\
\e_{\eqref{badVmean2}}:=\E\left[\sum_{k=1}^{c_0(\log n)^3}\sum_{|x|=k}\ind{\mV(x)\geq-\alpha,\MV(x)\leq \log n+a_n, V(x)>(1-\theta)\log n+b}\p^\en(\overline{L}_x(\tau_n)\geq n^\theta, E_x^{(n)}\geq 2)\right]=&o(n^{1-\theta}).\label{badVmean2}
\end{align}
Let us begin with \eqref{badVmean1}. For $x\in\T$ such that $\MV(x)>\log n+a_n$ with $a_n=a\log\log n$, we could show that
\[
\p^\en\left(\overline{L}_x(\tau_n)\geq n^\theta, E_x^{(n)}\geq 2\right)\leq n^{1-\theta}(\log n)^{-a} e^{-V(x)}.
\]
In fact, by Markov inequality, one sees that
\begin{align*}
\p^\en\left(\overline{L}_x(\tau_n)\geq n^\theta, E_x^{(n)}\geq 2\right)\leq &\frac{\e^\en[\overline{L}_x(\tau_n); E_x^{(n)}\geq 2]}{n^\theta}=\frac{H_x}{n^\theta}\e^\en\left[E_x^{(n)}; E_x^{(n)}\geq 2\right]\\
=& \frac{H_x}{n^\theta}\left[\e^\en[E_x^{(n)}]-\p^\en(E_x^{(n)}=1)\right]=\frac{H_x}{n^\theta} na_x\left(1-(1-a_x)^{n-1}\right),
\end{align*} where $\frac{H_x}{n^\theta} na_x=n^{1-\theta}e^{-V(x)}$ and $1-(1-a_x)^{n-1}\leq (n-1)a_x$. So,
\[
\p^\en\left(\overline{L}_x(\tau_n)\geq n^\theta, E_x^{(n)}\geq 2\right)\leq n^{1-\theta}e^{-V(x)}(n-1)a_x.
\]
Note that as $\MV(x)>\log n+a_n$ with $a_n=a\log\log n$, we have $(n-1)a_x\leq ne^{-\MV(x)}\leq (\log n)^{-a}$. So, $\p^\en\left(\overline{L}_x(\tau_n)\geq n^\theta, E_x^{(n)}\geq 2\right)\leq n^{1-\theta}(\log n)^{-a} e^{-V(x)}$. It hence follows that for $a>3$,
\begin{align*}
%\E\left[\sum_{k=1}^{c_0(\log n)^3}\sum_{|x|=k}\ind{\MV(x)>\log n+a_n}\p^\en(\overline{L}_x(\tau_n)\geq n^\theta, E_x^{(n)}\geq 2)\right]
\e_{\eqref{badVmean1}}\leq & n^{1-\theta}(\log n)^{-a} \E\left[\sum_{k=1}^{ c_0(\log n)^3}\sum_{|x|=k}e^{-V(x)}\right]\\
=&c_0(\log n)^{3-a} n^{1-\theta}=o(n^{1-\theta}),
\end{align*}
where the last line comes from Many-to-One Lemma. This proves \eqref{badVmean1}.

It remains to prove \eqref{badVmean2}. Observe that
%\e_{\eqref{badVmean2}}\leq \e_{\eqref{badVmean21}}+\e_{\eqref{badVmean22}}
%\end{equation}
\begin{multline*}
%&\E\left[\sum_{k=1}^{c_0(\log n)^3}\sum_{|x|=k}\ind{\mV(x)\geq-\alpha}\ind{\MV(x)\leq \log n+a_n, V(x)>(1-\theta)\log n+b}\p^\en(\overline{L}_x(\tau_n)\geq n^\theta, E_x^{(n)}\geq 2)\right]\\
\e_{\eqref{badVmean2}}\leq \E\left[\sum_{k=1}^{c_0(\log n)^3}\sum_{|x|=k}\ind{\MV(x)\leq \log n-a_n, V(x)>(1-\theta)\log n+b}\p^\en(\overline{L}_x(\tau_n)\geq n^\theta)\right]\\
+\E\left[\sum_{k=1}^{c_0(\log n)^3}\sum_{|x|=k}\ind{\mV(x)\geq-\alpha, \MV(x)\in[\log n-a_n, \log n+a_n], V(x)>(1-\theta)\log n+b}\p^\en(\overline{L}_x(\tau_n)\geq n^\theta, E_x^{(n)}\geq 2)\right].
\end{multline*}
Therefore, we only need to check that
\begin{align}
\e_{\eqref{badV2smallMV}}:=\E\left[\sum_{k=1}^{c_0(\log n)^3}\sum_{|x|=k}\ind{\MV(x)\leq \log n-a_n, V(x)>(1-\theta)\log n+b}\p^\en(\overline{L}_x(\tau_n)\geq n^\theta)\right]\label{badV2smallMV}\\
\E_{\eqref{badV2bdMV}}:=\e\left[\sum_{k=1}^{c_0(\log n)^3}\sum_{|x|=k}\ind{\mV(x)\geq-\alpha, \MV(x)\in[\log n-a_n, \log n+a_n], V(x)>(1-\theta)\log n+b}\p^\en(\overline{L}_x(\tau_n)\geq n^\theta, E_x^{(n)}\geq 2)\right]\label{badV2bdMV}
\end{align}
For $\e_{\eqref{badV2smallMV}}$, as $n^\theta(1- b_x)=\frac{n^\theta}{H_x}\geq e^b na_x$ if $V(x)>(1-\theta)\log n+b$, by \eqref{sumGeolarge} with $\eta=b$, 
\[
\p^\en\left(\overline{L}_x(\tau_n)\geq n^\theta\right)\leq 2 na_x e^{-c_\eta n^\theta b_x}=2 n^{1-\theta}e^{-V(x)} \frac{n^\theta}{H_x}e^{-c_\eta\frac{n^\theta}{H_x}}\leq c_{23} n^{1-\theta} e^{-V(x)} e^{-\frac{c_\eta}{2} \frac{n^\theta}{H_x}},%\leq 2 n^{1-\theta} e^{-V(x)}e^{-c_5 (\log n)^{a-3}}
\]
where we use the fact that $te^{-c_\eta t}\leq c_{24} e^{-c_\eta t/2}$ for any $t>0$ and $c_{24}:=\sup_{t\geq0}te^{-c_\eta/2 t}$. In addition, given $\{\MV(x)\leq \log n-a_n, V(x)>(1-\theta)\log n+b\}$, we get $H_x\leq |x|e^{\MV(x)-V(x)}\leq c_0(\log n)^{3-a}n^\theta$. Thus,
\[
\p^\en\left(\overline{L}_x(\tau_n)\geq n^\theta\right)\leq c_{23}n^{1-\theta} e^{-V(x)}e^{-c_{25} (\log n)^{a-3}}.
\]
This combined with Many-to-One Lemma implies that
\begin{align*}
%&\e\left[\sum_{k=1}^{ c_0(\log n)^3}\sum_{|x|=k}\ind{\MV(x)\leq \log n-a_n, V(x)>(1-\theta)\log n+b}\p^\en(\overline{L}_x(\tau_n)\geq n^\theta)\right]\\
\e_{\eqref{badV2smallMV}}\leq & c_{23} n^{1-\theta}e^{-c_{25} (\log n)^{a-3}} \E\left[\sum_{k=1}^{ c_0(\log n)^3}\sum_{|x|=k}e^{-V(x)}\right]\\
= &c_{23}c_0(\log n)^3e^{-c_{25} (\log n)^{a-3}} n^{1-\theta}=o(n^{1-\theta}),
\end{align*}
which shows \eqref{badV2smallMV}. 

For $\e_{\eqref{badV2bdMV}}$, again, as $n^\theta (1-b_x)=\frac{n^\theta}{H_x}\geq e^b na_x$, by \eqref{sumGeolarge+} with $\eta=b$, one has
\[
\p^\en\left(\overline{L}_x(\tau_n)\geq n^\theta, E_x^{(n)}\geq 2\right)\leq 2(na_x)^2 e^{-c_\eta n^\theta(1-b_x)}= 2n^{2(1-\theta)}e^{-2V(x)} (\frac{n^\theta}{H_x})^2e^{-c_\eta \frac{n^\theta}{H_x}},
\]
which is less than $c_{26} n^{2(1-\theta)}e^{-2V(x)}$ since $(\frac{n^\theta}{H_x})^2e^{-c_\eta \frac{n^\theta}{H_x}}\leq \sup_{t\geq0}t^2e^{-c_\eta t}<\infty$. 
As a result, 
\begin{align*}
%&\e\left[\sum_{k=1}^{ c_0(\log n)^3}\sum_{|x|=k}\ind{\mV(x)\geq-\alpha, \MV(x)\in[\log n-a_n, \log n+a_n], V(x)>(1-\theta)\log n+b}\p^\en(\overline{L}_x(\tau_n)\geq n^\theta, E_x^{(n)}\geq 2)\right]\\
\e_{\eqref{badV2bdMV}}\leq & c_{26} n^{1-\theta}\E\left[\sum_{k=1}^{ c_0(\log n)^3}\sum_{|x|=k}\ind{\mV(x)\geq-\alpha, \MV(x)\in[\log n-a_n, \log n+a_n], V(x)>(1-\theta)\log n+b}e^{-V(x)}e^{(1-\theta)\log n-V(x)} \right]\\
=&c_{26} n^{1-\theta}\sum_{k=1}^{c_0(\log n)^3}\E\left[e^{(1-\theta)\log n-S_k}; \mS_k\geq-\alpha, \MS_k\in[\log n-a_n,\log n+a_n], S_k>(1-\theta)\log n+b\right],
%=& 2n^{1-\theta}\sum_{k=1}^{c_0(\log n)^3}\E\left[e^{(1-\theta)\log n-S_k}f(\frac{n^\theta}{H_k(S)}); \mS_k\geq-\alpha, \MS_k\in[\log n-a_n,\log n+a_n], S_k>(1-\theta)\log n\right],
\end{align*}
where the last equality follows from Many-to-One Lemma. So it suffices to show that
\begin{equation}\label{badV2bdMVmean}
\E_{\eqref{badV2bdMVmean}}:=\sum_{k=1}^{c_0(\log n)^3}\E\left[e^{(1-\theta)\log n-S_k}\ind{\mS_k\geq-\alpha, \MS_k\in[\log n-a_n,\log n+a_n], S_k\geq(1-\theta)\log n}\right]=o_n(1).
\end{equation}
Apparently, $e^{(1-\theta)\log n-S_k}\leq e^{-a_n}$ if $S_k>(1-\theta)\log n+a_n$. As a result, for $a_n=a\log \log n$ with $a>3$,
\[
\E_{\eqref{badV2bdMVmean}}\leq o_n(1)+ \sum_{k=1}^{c_0(\log n)^3}\E_{\eqref{badV2bdMVbdS}}(k),
\]
where 
\begin{equation}\label{badV2bdMVbdS}
\E_{\eqref{badV2bdMVbdS}}(k):=\E\left[e^{(1-\theta)\log n-S_k}\ind{\mS_k\geq-\alpha, \MS_k\in[\log n-a_n,\log n+a_n], S_k\in[(1-\theta)\log n,(1-\theta)\log n+a_n]}\right].
\end{equation}
We only need to show that $\sum_{k=1}^{c_0(\log n)^3}\E_{\eqref{badV2bdMVbdS}}(k)=o_n(1)$. For $1\leq k\leq \varepsilon (\log n)^2$ with $\varepsilon\in(0,1)$ small, by \eqref{summSSbd},
\begin{align*}
\sum_{k=1}^{\varepsilon(\log n)^2}\E_{\eqref{badV2bdMVbdS}}(k)%&\sum_{k=1}^{\varepsilon(\log n)^2}\E\left[e^{(1-\theta)\log n-S_k}\ind{\mS_k\geq-\alpha, \MS_k\in[\log n-a_n,\log n+a_n], S_k\in[(1-\theta)\log n,(1-\theta)\log n+a_n]}\right]\\
\leq &\sum_{r=(1-\theta)\log n}^{(1-\theta)\log n+a_n}e^{(1-\theta)\log n-r}\sum_{k=1}^{\varepsilon(\log n)^2}\P\left(\mS_k\geq-\alpha, S_k\in[r, r+1]\right)\\
\leq & c_{27}(1+\alpha)\varepsilon=o_\varepsilon(1).
\end{align*}
For $k\geq (\log n)^2/\varepsilon$, by \eqref{mSSbd}, one has
\begin{align*}
%&\sum_{k=(\log n)^2/\varepsilon}^{c_0(\log n)^3}\E\left[e^{(1-\theta)\log n-S_k}\ind{\mS_k\geq-\alpha, \MS_k\in[\log n-a_n,\log n+a_n], S_k\in[(1-\theta)\log n,(1-\theta)\log n+a_n]}\right]\\
\sum_{k=(\log n)^2/\varepsilon}^{c_0(\log n)^3}\E_{\eqref{badV2bdMVbdS}}(k)\leq &\sum_{r=(1-\theta)\log n}^{(1-\theta)\log n+a_n}e^{(1-\theta)\log n-r}\sum_{k=(\log n)^2/\varepsilon}^{c_0(\log n)^3}\P\left(\mS_k\geq-\alpha, S_k\in[r, r+1]\right)\\
\leq & c_{28}(1+\alpha)^2\sqrt{\varepsilon}=o_\varepsilon(1).
\end{align*}

It remains to check that $\limsup_{n\to\infty}\sum_{k=\varepsilon(\log n)^2}^{(\log n)^2/\varepsilon}\E_{\eqref{badV2bdMVbdS}}(k)=o_\varepsilon(1)$. By considering the first time that $(S_i)_{0\leq i\leq k}$ hits $\MS_k$, we get that
\begin{align*}
&\sum_{k=\varepsilon(\log n)^2}^{(\log n)^2/\varepsilon}\E_{\eqref{badV2bdMVbdS}}(k)\\%\E\left[e^{(1-\theta)\log n-S_k}\ind{\mS_k\geq-\alpha, \MS_k\in[\log n-a_n,\log n+a_n], S_k\in[(1-\theta)\log n,(1-\theta)\log n+a_n]}\right]\nonumber\\
=&\sum_{k=\varepsilon(\log n)^2}^{(\log n)^2/\varepsilon}\sum_{j=1}^k\E\left[e^{(1-\theta)\log n-S_k}\ind{\mS_k\geq-\alpha, \MS_{j-1}<S_j=\MS_k\in[\log n-a_n,\log n+a_n], S_k\in[(1-\theta)\log n,(1-\theta)\log n+a_n]}\right]\\
\leq &\sum_{k=\varepsilon(\log n)^2}^{(\log n)^2/\varepsilon}\sum_{s=-a_n}^{a_n}\sum_{t=0}^{a_n}e^{-t}\sum_{j=1}^k\E\left[\ind{\mS_k\geq-\alpha, \MS_{j-1}<S_j=\MS_k\in[\log n+s,\log n+s+1], S_k\in[(1-\theta)\log n+t,(1-\theta)\log n+t+1]}\right]
\end{align*}
By Markov property at time $j$, one sees that
\begin{align*}
&%\sum_{s=-a_n}^{a_n}\sum_{t=0}^{a_n}e^{-t}\sum_{j=1}^k
\E[\ind{\mS_k\geq-\alpha, \MS_{j-1}<S_j=\MS_k\in[\log n+s,\log n+s+1], S_k\in[(1-\theta)\log n+t,(1-\theta)\log n+t+1]}]\\
\leq&%\sum_{s=-a_n}^{a_n}\sum_{t=0}^{a_n}e^{-t}\sum_{j=1}^{k-1}
\P(\mS_j\geq-\alpha, \MS_j=S_j\in[\log n+s,\log n+s+1])\P(\MS_{k-j}\leq0, S_{k-j}+\theta\log n\in[t-s-1,t-s+1]).
%\leq & \sum_{j=1}^{(\varepsilon\log n)^2}\P(\mS_j\geq-\alpha, \MS_j=S_j\ge\log n-a_n)\sup_{0\leq t\leq a_n, |s|\leq a_n}\P(\MS_{k-j}\leq0, S_{k-j}+\theta\log n\in[t-s-1,t-s+1])\\
%&+2a_n\sum_{j=(\varepsilon\log n)^{2}}^{k-(\varepsilon\log n)^2}\sup_{0\leq t\leq a_n, |s|\leq a_n}\P(\mS_j\geq-\alpha, \MS_j=S_j-\log n \in[s,s+1])\P(\MS_{k-j}\leq0, S_{k-j}+\theta\log n\in[t-s-1,t-s+1])\\
%&+2a_n\sup_{0\leq t\leq a_n, |s|\leq a_n}\sum_{j=k-(\varepsilon\log n)^2}^{k-1}
\end{align*}
So,
\begin{equation}\label{goodkbadMS}
\sum_{k=\varepsilon(\log n)^2}^{(\log n)^2/\varepsilon}\E_{\eqref{badV2bdMVbdS}}(k)\leq \sum_{k=\varepsilon(\log n)^2}^{(\log n)^2/\varepsilon}\sum_{j=1}^{k-1}\P_{\eqref{badV2cutMS}}(j,k),
\end{equation}
with 
\begin{multline}\label{badV2cutMS}
\P_{\eqref{badV2cutMS}}(j,k)\\
=\sum_{s=-a_n}^{a_n}\sum_{t=0}^{a_n}e^{-t}\P(\mS_j\geq-\alpha, \MS_j=S_j\in[\log n+s,\log n+s+1])\P(\MS_{k-j}\leq0, S_{k-j}+\theta\log n\in[t-s-1,t-s+1])
\end{multline}
Observe that $\sum_{j=1}^{k-1}\P_{\eqref{badV2cutMS}}(j,k)\le\sum_{j=1}^{\varepsilon^2(\log n)^2}+\sum_{j=(\varepsilon \log n)^2}^{k-(\varepsilon \log n)^2}+\sum_{j=k-(\varepsilon \log n)^2}^{k-1}\P_{\eqref{badV2cutMS}}(j,k)$. We bound the three sums separately. First, by \eqref{mSMSlargebd} and \eqref{mSSbd} for $j\leq (\varepsilon \log n)^2\le \varepsilon k$, we have
\begin{align*}
&\sum_{j=1}^{\varepsilon^2(\log n)^2}\P_{\eqref{badV2cutMS}}(j,k)\\
\leq&\sum_{j=1}^{\varepsilon^2(\log n)^2}\P(\mS_j\geq-\alpha, \MS_j=S_j\ge\log n-a_n)\sup_{0\leq t\leq a_n, |s|\leq a_n}\P(\MS_{k-j}\leq0, S_{k-j}+\theta\log n\in[t-s-1,t-s+1])\\
\leq & c_{29}\frac{\log n}{k^{3/2}}\sum_{j=1}^{\varepsilon^2(\log n)^2}\frac{1+\alpha}{j^{1/2} \log n}\leq c_{30} \frac{(1+\alpha)\varepsilon\log n}{k^{3/2}}
\end{align*}
For $(\varepsilon\log n)^{2}\leq j\leq k-(\varepsilon\log n)^{2}$, by \eqref{mSSbd} and \eqref{mSMS-Sbd}, one sees that 
\begin{align*}
\sum_{j=(\varepsilon\log n)^{2}}^{k-(\varepsilon\log n)^2}\P_{\eqref{badV2cutMS}}(j,k)
\leq&2a_n\sum_{j=(\varepsilon\log n)^{2}}^{k-(\varepsilon\log n)^2}\sup_{0\leq t\leq a_n, |s|\leq a_n}\P(\mS_j\geq-\alpha, \MS_j=S_j-\log n \in[s,s+1])\frac{(1+\log n)}{(k-j)^{3/2}}\\
\leq & 2a_n \sum_{j=(\varepsilon\log n)^{2}}^{k-(\varepsilon\log n)^2}\frac{c_{31}(1+\alpha)^4 (\log n)^4}{j^3(k-j)^{3/2}}\leq \frac{c_{32}a_n(1+\alpha)^4}{\varepsilon^4 k^{3/2}}+c_{32}(1+\alpha)^4\frac{(\log n)^3a_n}{\varepsilon k^{3}},
\end{align*}
As $k\geq \varepsilon (\log n)^2$, we get that $\sum_{j=(\varepsilon\log n)^{2}}^{k-(\varepsilon\log n)^2}\P_{\eqref{badV2cutMS}}(j,k) \leq c_{33}(1+\alpha)^4 \frac{a_n}{\varepsilon^4 k^{3/2}}$.

For $k-(\varepsilon \log n)^2\leq j<k$, by \eqref{mSMS-Sbd} and \eqref{summSSbd}, one sees that
\begin{align*}
&\sum_{j=k-(\varepsilon \log n)^2}^{k-1}\P_{\eqref{badV2cutMS}}(j,k)\\
%\leq&2a_n\sup_{t,s}\sum_{j=k-(\varepsilon \log n)^2}^{k-1}\P(\mS_j\geq-\alpha, \MS_j=S_j\in[\log n+s,\log n+s+1])\P(\MS_{k-j}\leq0, S_{k-j}+\theta\log n\in[t-s-1,t-s+1])\\
\leq& \frac{c_{34}a_n(1+\alpha)^4(\log n)^3}{k^3}\sup_{0\leq t\leq a_n, |s|\leq a_n}\sum_{j=1}^{(\varepsilon\log n)^2}\P(\MS_j\leq0, -S_j\in[\theta\log n+s-t-1, \theta\log n+s-t+1])\\
\leq & \frac{c_{35}a_n(1+\alpha)^4(\log n)^3}{k^3}\varepsilon^2.
\end{align*}
As a consequence, 
\[
\sum_{j=1}^{k-1}\P_{\eqref{badV2cutMS}}(j,k)\leq c_{30} \frac{(1+\alpha)\varepsilon\log n}{k^{3/2}}+c_{33}(1+\alpha)^4 \frac{a_n}{\varepsilon^4 k^{3/2}}+\frac{c_{35}a_n(1+\alpha)^4(\log n)^3}{k^3}\varepsilon^2.
\]
Plugging it  into \eqref{goodkbadMS} yields that
\[
\sum_{k=\varepsilon(\log n)^2}^{(\log n)^2/\varepsilon}\E_{\eqref{badV2bdMVbdS}}(k)=o_n(1)+o_\varepsilon(1),
\]
which completes the proof of \eqref{badV2bdMVmean}. We thus conclude \eqref{badV2bdMV} and \eqref{meansmallpartbadV}.
\end{proof}

\begin{proof}[Proof of Lemma \ref{largepartbadtime}]
Let 
\begin{align*}
\E_{\eqref{largepartsmalltime}}:=\E\left[\sum_{m=1}^{\varepsilon (\log n)^2}\sum_{|z|=m}e^{-V(z)}F_2(z,\log n)\ind{\mV(z)\geq-\alpha}\right],\\
\E_{\eqref{largepartlargetime}}:=\E\left[\sum_{m= (\log n)^2/\varepsilon}^{c_0(\log n)^3}\sum_{|z|=m}e^{-V(z)}F_2(z,\log n)\ind{\mV(z)\geq-\alpha}\right].
\end{align*}
Let us bound $\E_{\eqref{largepartsmalltime}}$ first. By Many-to-One Lemma, 
\begin{align*}
E_{\eqref{largepartsmalltime}}=&\sum_{k=1}^{\varepsilon (\log n)^2}\E\left[e^{S_k-(1-\theta)\log n-b}\ind{\mS_k\geq-\alpha, \MS_k\leq \log n+a_n, S_k\leq (1-\theta)\log n+b}\right]\\
\leq &\sum_{k=1}^{\varepsilon (\log n)^2} e^{-\frac{1-\theta}{2}\log n-b}+\sum_{k=1}^{\varepsilon (\log n)^2}\E\left[e^{S_k-(1-\theta)\log n-b}\ind{\mS_k\geq-\alpha, \MS_k\leq \log n+a_n, S_k\in[\frac{1-\theta}{2}\log n, (1-\theta)\log n+b]}\right]\\
\le &o_n(1)+\sum_{t=\frac{1-\theta}{2}\log n}^{(1-\theta)\log n+b}e^{t-(1-\theta)\log n-b}\sum_{k=1}^{\varepsilon (\log n)^2}\P(\mS_k\geq-\alpha, S_k\in [t,t+1]).
\end{align*}
We then deduce from \eqref{summSSbd} that $\E_{\eqref{largepartsmalltime}}=o_n(1)+o_\varepsilon(1)$. This suffices to conclude \eqref{largepartsmalltime}. 

On the other hand, by Many-to-One Lemma,
\begin{align*}
\E_{\eqref{largepartlargetime}}=&\sum_{k= (\log n)^2/\varepsilon}^{c_0(\log n)^3}\E\left[e^{S_k-(1-\theta)\log n-b}\ind{\mS_k\geq-\alpha, \MS_k\leq \log n+a_n, S_k\leq (1-\theta)\log n+b}\right]\\
\leq &\sum_{k= (\log n)^2/\varepsilon}^{c_0(\log n)^3}e^{-\frac{1-\theta}{2}\log n-b}+\sum_{k= (\log n)^2/\varepsilon}^{c_0(\log n)^3}\sum_{t=\frac{1-\theta}{2}\log n}^{(1-\theta)\log n+b}e^{t-(1-\theta)\log n-b}\P(\mS_k\geq-\alpha, S_k\in [t,t+1]).
\end{align*}
By use of \eqref{mSSbd}, we obtain that
\begin{align*}
\E_{\eqref{largepartlargetime}}\leq & o_n(1)+\sum_{k= (\log n)^2/\varepsilon}^{c_0(\log n)^3}\frac{c_{36}(1+\alpha)(1+(1-\theta)\log n+b+\alpha)}{k^{3/2}}=o_n(1)+o_\varepsilon(1).
\end{align*}
This ends the proof of Lemma \ref{largepartbadtime}.
\end{proof}

\begin{proof}[Proof of Lemma \ref{smallpart1ex}]
In fact, as $\p^\en(\eL_x(\tau_n)\geq n^\theta, E_x^{(n)}=1)=na_x(1-a_x)^{n-1}b_x^{n^\theta-1}$, we are going to show that
\begin{align}
\E_{\eqref{1exnotinB-}}:=\sum_{\ell=1}^{c_0(\log n)^3}\E\left[\sum_{|x|=\ell}na_x(1-a_x)^{n-1}b_x^{n^\theta-1}\ind{\MV(x)\leq \log n-a_n}\right]=&o(n^{1-\theta});\label{1exnotinB-}\\
\E_{\eqref{1exnotinD}}:=\sum_{\ell=1}^{c_0(\log n)^3}\E\left[\sum_{|x|=\ell}na_x(1-a_x)^{n-1}b_x^{n^\theta-1}\ind{x\not\in\D_n}\right]=&o(n^{1-\theta}).\label{1exnotinD}
\end{align}
First, observe that if $\MV(x)\leq \log n-a_n$ with $a_n=a\log\log n$ and $|x|\leq c_0(\log n)^3$, then 
\[
na_x(1-a_x)^{n-1}b_x^{n^\theta-1}\leq n^{1-\theta} e^{-V(x)}\frac{n^\theta}{H_x}e^{-\frac{n^\theta}{2H_x}}e^{-(n-1)a_x}\leq n^{1-\theta}e^{-V(x)}e^{-c_{37}(\log n)^{a-3}}
\]
as $a_x\geq \frac{1}{|x|e^{\MV(x)}}$. This follows that
\begin{equation*}
%\sum_{\ell=1}^{c_0(\log n)^3}\E\left[\sum_{|x|=\ell}na_x(1-a_x)^{n-1}b_x^{n^\theta-1}\ind{\MV(x)\leq \log n-a_n}\right]
\E_{\eqref{1exnotinB-}}\leq  n^{1-\theta}e^{-c_{37}(\log n)^{a-3}}\sum_{\ell=1}^{c_0(\log n)^3}\E\left[\sum_{|x|=\ell}e^{-V(x)} \right]
\end{equation*}
which by Many-to-One lemma, is bounded by $n^{1-\theta} c_0(\log n)^3e^{-c_{37}(\log n)^{a-3}}=o(n^{1-\theta})$. This proves \eqref{1exnotinB-}. 

Next, let
\begin{align}
\E_{\eqref{1exnotinD+}}:=&\sum_{\ell=1}^{c_0(\log n)^3}\E\left[\sum_{|x|=\ell}na_x(1-a_x)^{n-1}b_x^{n^\theta-1}\ind{\MV(x)-V(x)>\theta\log n+a_n}\right]\label{1exnotinD+}\\
\E_{\eqref{1exnotinD-}}:=&\sum_{\ell=1}^{c_0(\log n)^3}\E\left[\sum_{|x|=\ell}na_x(1-a_x)^{n-1}b_x^{n^\theta-1}\ind{\MV(x)-V(x)<\theta\log n-a_n}\right]\label{1exnotinD-}.
\end{align}
It is immediate that
\begin{equation}\label{1exbadvalley}
%&\sum_{\ell=1}^{c_0(\log n)^3}\E\left[\sum_{|x|=\ell}na_x(1-a_x)^{n-1}b_x^{n^\theta-1}\ind{x\not\in\D_n}\right]\nonumber\\
\E_{\eqref{1exnotinD}}=\E_{\eqref{1exnotinD+}}+\E_{\eqref{1exnotinD-}}.
%=&\sum_{\ell=1}^{c_0(\log n)^3}\E\left[\sum_{|x|=\ell}na_x(1-a_x)^{n-1}b_x^{n^\theta-1}\ind{\MV(x)-V(x)>\theta\log n+a_n}\right]\nonumber\\
%&+\sum_{\ell=1}^{c_0(\log n)^3}\E\left[\sum_{|x|=\ell}na_x(1-a_x)^{n-1}b_x^{n^\theta-1}\ind{\MV(x)-V(x)<\theta\log n-a_n}\right].
\end{equation}
So, we only need to check that $\E_{\eqref{1exnotinD+}}=o(n^{1-\theta})$ and $\E_{\eqref{1exnotinD-}}=o(n^{1-\theta})$. 

On the one hand, note that $|x|e^{\MV(x)-V(x)}\geq H_x \geq e^{\MV(x)-V(x)}$. If $\MV(x)-V(x)\geq \theta \log n+a_n$, then $\frac{n^\theta}{H_x}\leq (\log n)^{-a}$ and
\[
na_x(1-a_x)^{n-1}b_x^{n^\theta-1}\leq n^{1-\theta} e^{-V(x)}\frac{n^\theta}{H_x}\le n^{1-\theta} e^{-V(x)} (\log n)^{-a}.
\]
This brings out that 
\begin{align}\label{1exlargevalley}
%&\sum_{\ell=1}^{c_0(\log n)^3}\E\left[\sum_{|x|=\ell}na_x(1-a_x)^{n-1}b_x^{n^\theta-1}\ind{\MV(x)-V(x)>\theta\log n+a_n}\right]\nonumber\\
\E_{\eqref{1exnotinD+}}\leq & n^{1-\theta}(\log n)^{-a}\sum_{\ell=1}^{c_0(\log n)^3}\E\left[\sum_{|x|=\ell}e^{-V(x)} \right]=o(n^{1-\theta}).
\end{align}
On the other hand, if $\MV(x)-V(x)\leq \theta\log n-a_n$ and $|x|\le c_0(\log n)^3$, one has $\frac{n^\theta}{H_x}\geq \frac{1}{c_0}(\log n)^{a-3}$ and
\[
na_x(1-a_x)^{n-1}b_x^{n^\theta-1}\leq n^{1-\theta} e^{-V(x)}e^{-\frac{n^\theta}{2H_x}}\leq n^{1-\theta}e^{-V(x)} e^{-c_{38}(\log n)^{a-3}}.
\]
As a consequence,
\begin{align}\label{1exsmallvalley}
%&\sum_{\ell=1}^{c_0(\log n)^3}\E\left[\sum_{|x|=\ell}na_x(1-a_x)^{n-1}b_x^{n^\theta-1}\ind{\MV(x)-V(x)<\theta\log n-a_n}\right]\nonumber\\
\E_{\eqref{1exnotinD-}}\leq &n^{1-\theta}e^{-c(\log n)^{a-3}}\sum_{\ell=1}^{c_0(\log n)^3} \E\left[\sum_{|x|=\ell}e^{-V(x)} \right]=o(n^{1-\theta}).
\end{align}
We then deduce \eqref{1exnotinD} from \eqref{1exbadvalley}, \eqref{1exlargevalley} and \eqref{1exsmallvalley}.
\end{proof}

\begin{proof}[Proof of Lemma \ref{smallpart1ex+}]
Again, as $\p^\en(\eL_x(\tau_n)\geq n^\theta, E_x^{(n)}=1)=na_x(1-a_x)^{n-1}b_x^{n^\theta-1}$, we are going to show that
\begin{align}
\limsup_{n\to\infty}\frac{1}{n^{1-\theta}}\sum_{\ell=1}^{\varepsilon(\log n)^2}\E\left[\sum_{|x|=\ell}na_x(1-a_x)^{n-1}b_x^{n^\theta-1}\ind{x\in\B_n^-\cap \D_n, \mV(x)\geq-\alpha}\right]=&o_\varepsilon(1);\label{smallgeneration1ex++}\\
\limsup_{n\to\infty}\frac{1}{n^{1-\theta}}\sum_{\ell=(\log n)^2/\varepsilon}^{c_0(\log n)^3}\E\left[\sum_{|x|=\ell}na_x(1-a_x)^{n-1}b_x^{n^\theta-1}\ind{x\in\B_n^-\cap \D_n\cap \L_n,\mV(x)\geq-\alpha}\right]=&o_\varepsilon(1);\label{largegeneration1ex++}\\
\frac{1}{n^{1-\theta}}\sum_{\ell=\varepsilon(\log n)^2}^{(\log n)^2/\varepsilon}\E\left[\sum_{|x|=\ell}na_x(1-a_x)^{n-1}b_x^{n^\theta-1}\ind{x\in\D_n}\ind{\MV(x)\in[\log n-a_n,\log n+a_n],\mV(x)\geq-\alpha}\right]=&o_n(1).\label{badmaxV1ex++}
\end{align}
\textbf{Proof of \eqref{smallgeneration1ex++}.} Let 
\[
\E_{\eqref{smallgeneration1ex++}}(\ell):=\E\left[\sum_{|x|=\ell}na_x(1-a_x)^{n-1}b_x^{n^\theta-1}\ind{x\in\B_n^-\cap \D_n, \mV(x)\geq-\alpha}\right].
\]
To get \eqref{smallgeneration1ex++}, we need to bound %$\sum_{\ell=1}^{\varepsilon(\log n)^2}\E_{\eqref{smallgeneration1ex++}}(\ell)=
$\sum_{\ell=1}^{(\log n)^{1+\delta}}\E_{\eqref{smallgeneration1ex++}}(\ell)$ and $\sum_{\ell=(\log n)^{1+\delta}}^{\varepsilon(\log n)^2}\E_{\eqref{smallgeneration1ex++}}(\ell)$ with $\delta\in(0,1)$.

As $na_x(1-a_x)^{n-1}b_x^{n^\theta-1}\leq n^{1-\theta}e^{-V(x)}$, we see that by Many-to-One Lemma,
\begin{align*}
\E_{\eqref{smallgeneration1ex++}}%&\sum_{\ell=1}^{\varepsilon(\log n)^2}\E\left[\sum_{|x|=\ell}na_x(1-a_x)^{n-1}b_x^{n^\theta-1}\ind{x\in\B_n^-\cap \D_n, \mV(x)\geq-\alpha}\right]\\
\leq & n^{1-\theta}\E\left[\sum_{|x|=\ell} e^{-V(x)}\ind{\mV(x)\geq -\alpha, \MV(x)\geq\log n-a_n, \MV(x)-V(x)\in[\theta\log n-a_n,\theta\log n+a_n]},\right]\\
=&n^{1-\theta}\P(\mS_\ell\geq-\alpha, \MS_\ell\geq\log n-a_n, \MS_\ell-S_\ell\in[\theta\log n-a_n,\theta\log n+a_n]).
\end{align*}
For $1\leq\ell\leq (\log n)^{1+\delta}$ with $\delta\in(0,1/3)$, one sees that by \eqref{sumSbd},
\begin{align*}
\sum_{\ell=1}^{(\log n)^{1+\delta}}\E_{\eqref{smallgeneration1ex++}}(\ell)\leq&n^{1-\theta}\sum_{\ell=1}^{(\log n)^{1+\delta}}\P(\mS_\ell\geq-\alpha, \MS_\ell\geq\log n-a_n, \MS_\ell-S_\ell\in[\theta\log n-a_n,\theta\log n+a_n])\\
\leq & n^{1-\theta}\sum_{\ell=1}^{(\log n)^{1+\delta}}\P(S_\ell\geq (1-\theta)\log n-2a_n)\\
\leq &n^{1-\theta}e^{-c_{39}(\log n)^{1-\delta}}=o(n^{1-\theta}).
\end{align*}
For $\ell\geq (\log n)^{1+\delta}$, by considering the first time hitting $\MS_\ell$, one sees that
\begin{align*}
&\frac{1}{n^{1-\theta}}\sum_{\ell=(\log n)^{1+\delta}}^{\varepsilon(\log n)^2}\E_{\eqref{smallgeneration1ex++}}(\ell)\\%\le\sum_{\ell=(\log n)^{1+\delta}}^{\varepsilon(\log n)^2}\P(\mS_\ell\geq-\alpha, \MS_\ell\geq\log n-a_n, \MS_\ell-S_\ell\in[\theta\log n-a_n,\theta\log n+a_n])\\
\le &\sum_{\ell=(\log n)^{1+\delta}}^{\varepsilon(\log n)^2}\sum_{j=1}^{\ell-1}\P(\mS_\ell\geq-\alpha, \MS_{j-1}<S_j=\MS_\ell, S_j\geq \log n-a_n, S_j-S_\ell\in[\theta\log n-a_n,\theta\log n+a_n]),
\end{align*}
which by Markov property at time $j$, is less than
\begin{align*}
&\sum_{\ell=(\log n)^{1+\delta}}^{\varepsilon(\log n)^2}\sum_{j=1}^{\ell-1}\P(\mS_j\geq-\alpha, \MS_j=S_j\geq\log n-a_n)\P(\MS_{\ell-j}\leq 0, -S_{\ell-j}\in[\theta\log n-a_n,\theta\log n+a_n])\\
\leq & \sum_{\ell=(\log n)^{1+\delta}}^{\varepsilon(\log n)^2}\left(\sum_{j=1}^{(\log n)^{1+\delta}/2}\P(S_j\geq\log n-a_n)+\sum_{j=\ell-(\log n)^{1+\delta}/2}^{\ell}\P(-S_{\ell-j}\in[\theta\log n-a_n,\theta\log n+a_n])\right.\\
&+\left.\sum_{j=(\log n)^{1+\delta}/2}^{\ell-(\log n)^{1+\delta}/2}\P(\mS_j\geq-\alpha, \MS_j=S_j\geq\log n-a_n)\P(\MS_{\ell-j}\leq 0, -S_{\ell-j}\in[\theta\log n-a_n,\theta\log n+a_n])\right)
\end{align*}
Again by \eqref{sumSbd}, one has
\[
\sum_{j=1}^{(\log n)^{1+\delta}/2}\P(S_j\geq\log n-a_n)+\sum_{\ell-(\log n)^{1+\delta}/2}^{\ell}\P(-S_{\ell-j}\in[\theta\log n-a_n,\theta\log n+a_n])\leq e^{-c_{40}(\log n)^{1-\delta}}.
\]
On the other hand, by \eqref{mSMSSlargebde} and \eqref{mSSbde}, one has
\begin{align*}
&\sum_{j=(\log n)^{1+\delta}/2}^{\ell-(\log n)^{1+\delta}/2}\P(\mS_j\geq-\alpha, \MS_j=S_j\geq\log n-a_n)\P(\MS_{\ell-j}\leq 0, -S_{\ell-j}\in[\theta\log n-a_n,\theta\log n+a_n])\\
\leq &\sum_{j=(\log n)^{1+\delta}/2}^{\ell-(\log n)^{1+\delta}/2}\frac{c_{41}(1+\alpha)}{\sqrt{j}\log n}e^{-c_{42}\frac{(\log n)^2}{j}}\frac{1}{\ell-j}e^{-c_{43}\frac{(\log n)^2}{\ell-j}}\\
\leq & \frac{2}{\ell}e^{-c_{43}\frac{(\log n)^2}{\ell}}\sum_{j=(\log n)^{1+\delta}/2}^{\ell/2}\frac{c_{41}(1+\alpha)}{\sqrt{j}\log n}e^{-c_{42}\frac{(\log n)^2}{j}}+\frac{c_{41}(1+\alpha)}{\log n\sqrt{\ell/2}}e^{-c_{42}\frac{(\log n)^2}{\ell}}\sum_{j=\ell/2}^{\ell}\frac{1}{\ell-j}e^{-c_{43}\frac{(\log n)^2}{\ell-j}}\\
\leq &\frac{c_{44}(1+\alpha)}{\ell}e^{-c_{45}\frac{(\log n)^2}{\ell}}.
\end{align*}
As a consequence,
\begin{align*}
%&\sum_{\ell=(\log n)^{1+\delta}}^{\varepsilon(\log n)^2}\P(\mS_\ell\geq-\alpha, \MS_\ell\geq\log n-a_n, \MS_\ell-S_\ell\in[\theta\log n-a_n,\theta\log n+a_n])\\
\frac{1}{n^{1-\theta}}\sum_{\ell=(\log n)^{1+\delta}}^{\varepsilon(\log n)^2}\E_{\eqref{smallgeneration1ex++}}(\ell)\leq & \varepsilon(\log n)^2e^{-c_{40}(\log n)^{1-\delta}}+\sum_{\ell=(\log n)^{1+\delta}}^{\varepsilon(\log n)^2}\frac{c_{44}(1+\alpha)}{\ell}e^{-c_{45}\frac{(\log n)^2}{\ell}},
\end{align*}
which converges to $0$ as $n\to\infty$ then $\varepsilon\downarrow 0$.

\textbf{Proof of \eqref{largegeneration1ex++}.} Let 
\[
\E_{\eqref{largegeneration1ex++}}:=\sum_{\ell=(\log n)^2/\varepsilon}^{c_0(\log n)^3}\E\left[\sum_{|x|=\ell}na_x(1-a_x)^{n-1}b_x^{n^\theta-1}\ind{x\in\B_n^-\cap \D_n\cap \L_n,\mV(x)\geq-\alpha}\right]
\]
Note that $na_x(1-a_x)^{n-1}b_x^{n^\theta-1}\leq c_{46} n^{1-\theta}e^{-V(x)}(\frac{n^\theta}{H_x}\wedge \frac{H_x}{n^\theta})$ as $xe^{-x/2}\leq c_{46}(x\wedge \frac{1}{x})$. Therefore, by Many-to-One Lemma,
\begin{align*}
%&\sum_{\ell=(\log n)^2/\varepsilon}^{c_0(\log n)^3}\E\left[\sum_{|x|=\ell}na_x(1-a_x)^{n-1}b_x^{n^\theta-1}\ind{x\in\B_n^-\cap \D_n\cap \L_n,\mV(x)\geq-\alpha}\right]\\
&\E_{\eqref{largegeneration1ex++}}
\leq  c_{46} n^{1-\theta}\sum_{\ell=(\log n)^2/\varepsilon}^{c_0(\log n)^3}\E\left[\sum_{|x|=\ell}e^{-V(x)}(\frac{n^\theta}{H_x}\wedge \frac{H_x}{n^\theta})\ind{x\in\B_n^-\cap \D_n\cap \L_n,\mV(x)\geq-\alpha}\right]\\
=&c_{46} n^{1-\theta}\sum_{\ell=(\log n)^2/\varepsilon}^{c_0(\log n)^3}\E\left[(\frac{n^\theta}{H_\ell^S}\wedge \frac{H_\ell^S}{n^\theta})\ind{\mS_\ell\geq-\alpha, \MS_\ell\geq\log n-a_n, \max_{k\leq \ell}(\MS_k-S_k)\leq \log n, \MS_\ell-S_\ell\in[\theta\log n-a_n,\theta\log n+a_n]}\right]
\end{align*}
where $H_\ell^S:=\sum_{k=0}^\ell e^{S_k-S_\ell}$. Observe that
\begin{align*}
&\E\left[(\frac{n^\theta}{H_\ell^S}\wedge \frac{H_\ell^S}{n^\theta})\ind{\mS_\ell\geq-\alpha, \MS_\ell\geq\log n-a_n, \max_{k\leq \ell}(\MS_k-S_k)\leq \log n, \MS_\ell-S_\ell\in[\theta\log n-a_n,\theta\log n+a_n]}\right]\\
\leq &\sum_{x=-a_n}^{-1}\E\left[\frac{H_\ell^S}{n^{\theta}}\ind{\mS_\ell\geq-\alpha, \MS_\ell\geq\log n-a_n, \max_{k\leq \ell}(\MS_k-S_k)\leq \log n, \MS_\ell-S_\ell\in[\theta\log n+x,\theta\log n+x+1]}\right]\\
&+\sum_{x=0}^{a_n-1}\E\left[\frac{n^{\theta}}{H_\ell^S}\ind{\mS_\ell\geq-\alpha, \MS_\ell\geq\log n-a_n, \max_{k\leq \ell}(\MS_k-S_k)\leq \log n, \MS_\ell-S_\ell\in[\theta\log n+x,\theta\log n+x+1]}\right]
\end{align*}
Note that $H_\ell^S\geq e^{\MS_\ell-S_\ell}$. It follows that
\begin{align*}
&\sum_{x=0}^{a_n-1}\E\left[\frac{n^{\theta}}{H_\ell^S}\ind{\mS_\ell\geq-\alpha, \MS_\ell\geq\log n-a_n, \max_{k\leq \ell}(\MS_k-S_k)\leq \log n, \MS_\ell-S_\ell\in[\theta\log n+x,\theta\log n+x+1]}\right]\\
\leq &\sum_{x=0}^{a_n-1}\E\left[n^{\theta}e^{S_\ell-\MS_\ell}\ind{\mS_\ell\geq-\alpha, \MS_\ell\geq\log n-a_n, \max_{k\leq \ell}(\MS_k-S_k)\leq \log n, \MS_\ell-S_\ell\in[\theta\log n+x,\theta\log n+x+1]}\right]\\
\leq & \sum_{x=0}^{a_n-1}e^{-x}\E\left[\ind{\mS_\ell\geq-\alpha, \MS_\ell\geq\log n-a_n,\max_{k\leq \ell}(\MS_k-S_k)\leq \log n, \MS_\ell-S_\ell\in[\theta\log n+x,\theta\log n+x+1]}\right].
\end{align*}
On the other hand, as $H_\ell^S=\sum_{k=0}^\ell e^{S_k-\MS_\ell}e^{\MS_\ell-S_\ell}$,
\begin{align*}
&\sum_{x=-a_n}^{-1}\E\left[\frac{H_\ell^S}{n^{\theta}}\ind{\mS_\ell\geq-\alpha, \MS_\ell\geq\log n-a_n, \max_{k\leq \ell}(\MS_k-S_k)\leq \log n, \MS_\ell-S_\ell\in[\theta\log n+x,\theta\log n+x+1]}\right]\\
\leq & \sum_{x=-a_n}^{-1}e^x\E\left[\sum_{k=0}^\ell e^{S_k-\MS_\ell}\ind{  \mS_\ell\geq-\alpha, \MS_\ell\geq\log n-a_n, \max_{k\leq \ell}(\MS_k-S_k)\leq \log n, \MS_\ell-S_\ell\in[\theta\log n+x,\theta\log n+x+1]}\right]
\end{align*}
So, it suffices to prove that uniformly for $x\in[-a_n, a_n]$,
\begin{align}\label{positivex}
\sum_{\ell=(\log n)^2/\varepsilon}^{c_0(\log n)^3}\E\left[\ind{\mS_\ell\geq-\alpha, \MS_\ell\geq\log n-a_n,\max_{k\leq \ell}(\MS_k-S_k)\leq \log n, \MS_\ell-S_\ell\in[\theta\log n+x,\theta\log n+x+1]}\right]=&o_{n,\varepsilon}(1)\\
\sum_{\ell=(\log n)^2/\varepsilon}^{c_0(\log n)^3}\E\left[\sum_{k=0}^\ell e^{S_k-\MS_\ell}\ind{\mS_\ell\geq-\alpha, \MS_\ell\geq\log n-a_n, \max_{k\leq \ell}(\MS_k-S_k)\leq \log n, \MS_\ell-S_\ell\in[\theta\log n+x,\theta\log n+x+1]}\right]=&o_{n,\varepsilon}(1).\label{negativex}
\end{align}

First, we consider \eqref{positivex} and let $\E_{\eqref{positivex}}(\ell):=\E\left[\ind{\mS_\ell\geq-\alpha, \MS_\ell\geq\log n-a_n,\max_{k\leq \ell}(\MS_k-S_k)\leq \log n, \MS_\ell-S_\ell\in[\theta\log n+x,\theta\log n+x+1]}\right]$. By considering the first time hitting $\MS_\ell$, we have
\begin{equation}\label{positivexsum}
\sum_{\ell=(\log n)^2/\varepsilon}^{c_0(\log n)^3}\E_{\eqref{positivex}}(\ell)= \sum_{\ell=(\log n)^2/\varepsilon}^{c_0(\log n)^3}\sum_{j=1}^{\ell-1}\E_{\eqref{positivexsum}}(j,\ell),
\end{equation}
where 
\[
\E_{\eqref{positivexsum}}(j,\ell):=\E\left[\ind{\mS_\ell\geq-\alpha, \MS_\ell=S_j>\MS_{j-1}, \MS_\ell\geq\log n-a_n, \max_{k\leq \ell}(\MS_k-S_k)\leq \log n, \MS_\ell-S_\ell\in[\theta\log n+x,\theta\log n+x+1]}\right].
\]
By Markov property at time $j$, It is immediate that
\begin{multline}\label{Markovpositivex}
\E_{\eqref{positivexsum}}(j,\ell)\leq \P\left(\mS_j\geq-\alpha, \max_{k\leq j}(\MS_k-S_k)\leq \log n, \MS_j=S_j\geq\log n-a_n\right)\\
\times\P\left(\MS_{\ell-j}\leq 0, \mS_{\ell-j}\geq-\log n, -S_{\ell-j}\in[\log n+x,\log n+x+1]\right).
\end{multline}
Observe that for $j\leq (\log n)^{1+\delta}$ or $j\geq \ell-(\log n)^{1+\delta}$ with $\delta\in(0,1)$,
\begin{align*}
&\sum_{\ell=(\log n)^2/\varepsilon}^{c_0(\log n)^3}\left[\sum_{j=1}^{(\log n)^{1+\delta}}\E_{\eqref{positivexsum}}(j,\ell)+\sum_{j=\ell-(\log n)^{1+\delta}}^{\ell-1}\E_{\eqref{positivexsum}}(j,\ell)\right]\\
\leq &\sum_{\ell=(\log n)^2/\varepsilon}^{c_0(\log n)^3}\left[\sum_{j=1}^{(\log n)^{1+\delta}}\P(S_j\geq \log n-a_n)+\sum_{j=\ell-(\log n)^{1+\delta}}^{\ell-1}\P(-S_{\ell-j}\geq \theta\log n-a_n)\right].
\end{align*}
By \eqref{sumSbd}, we get that
\begin{equation}
\sum_{\ell=(\log n)^2/\varepsilon}^{c_0(\log n)^3}\left[\sum_{j=1}^{(\log n)^{1+\delta}}\E_{\eqref{positivexsum}}(j,\ell)+\sum_{j=\ell-(\log n)^{1+\delta}}^{\ell-1}\E_{\eqref{positivexsum}}(j,\ell)\right]\leq c_0(\log n)^3 e^{-c_{47}(\log n)^{1-\delta}}=o_n(1).
\end{equation}
For $(\log n)^{1+\delta}\leq j\leq (\log n)^2$, by \eqref{Markovpositivex}, \eqref{mSMSSlargebde} and \eqref{mSSbd}, one has
\begin{align*}
&\sum_{\ell=(\log n)^2/\varepsilon}^{c_0(\log n)^3}\sum_{j=(\log n)^{1+\delta}}^{(\log n)^2}\E_{\eqref{positivexsum}}(j,\ell)\\
\leq &\sum_{\ell=(\log n)^2/\varepsilon}^{c_0(\log n)^3}\sum_{j=(\log n)^{1+\delta}}^{(\log n)^2}\P(\mS_j\geq-\alpha, \MS_j=S_j\geq \log n-a_n)\P(\MS_{\ell-j}\leq0, -S_{\ell-j}-\theta\log n\in[x,x+1])\\
%\leq &\sum_{j=(\log n)^{1+\delta}}^{(\log n)^2}\P(\mS_j\geq-\alpha, \MS_j=S_j\geq \theta\log n-2a_n)\frac{\theta\log n+a_n}{(\ell-j)^{3/2}}\\
\leq &\sum_{\ell=(\log n)^2/\varepsilon}^{c_0(\log n)^3}\sum_{j=(\log n)^{1+\delta}}^{(\log n)^2}c_{48}\frac{1+\alpha}{\sqrt{j}\log n}e^{-c_{49}\frac{(\log n)^2}{j}}\frac{\theta\log n+a_n}{(\ell-j)^{3/2}}\leq \sum_{\ell=(\log n)^2/\varepsilon}^{c_0(\log n)^3}c_{50}(1+\alpha)\frac{\log n}{\ell^{3/2}},
\end{align*}
which is $o_\varepsilon(1)$ as $\varepsilon\downarrow0$.
For $(\log n)^2\leq j\leq \ell-(\log n)^2$, by \eqref{Markovpositivex}, \eqref{mSMSvalleybd} and \eqref{mSSbd}, one has
\begin{align*}
&\sum_{\ell=(\log n)^2/\varepsilon}^{c_0(\log n)^3}\sum_{j=(\log n)^{2}}^{\ell-(\log n)^2}\E_{\eqref{positivexsum}}(j,\ell)\\
&\sum_{\ell=(\log n)^2/\varepsilon}^{c_0(\log n)^3}\sum_{j=(\log n)^{2}}^{\ell-(\log n)^2}\P(\mS_j\geq-\alpha, \max_{k\leq j}(\MS_k-S_k)\leq \log n, \MS_j=S_j)\P(\MS_{\ell-j}\leq0, -S_{\ell-j}-\theta\log n\in[x,x+1])\\
\leq &\sum_{\ell=(\log n)^2/\varepsilon}^{c_0(\log n)^3}\sum_{j=(\log n)^{2}}^{\ell-(\log n)^2}c_{51}\frac{1+\alpha}{j}e^{-c_{52}\frac{j}{(\log n)^2}}\frac{\theta\log n+a_n}{(\ell-j)^{3/2}}\leq \sum_{\ell=(\log n)^2/\varepsilon}^{c_0(\log n)^3}c_{53}(1+\alpha)\frac{\log n}{\ell^{3/2}},
\end{align*}
which is also $o_\varepsilon(1)$.
For $\ell-(\log n)^2\leq j\leq \ell-(\log n)^{1+\delta}$, by \eqref{Markovpositivex}, \eqref{mSMSvalleybd} and \eqref{mSSbde}, one has
\begin{align*}
&\sum_{\ell=(\log n)^2/\varepsilon}^{c_0(\log n)^3}\sum_{j=\ell-(\log n)^2}^{\ell-(\log n)^{1+\delta}}\E_{\eqref{positivexsum}}(j,\ell)\\
\le&\sum_{\ell=(\log n)^2/\varepsilon}^{c_0(\log n)^3}\sum_{j=\ell-(\log n)^2}^{\ell-(\log n)^{1+\delta}}\P(\mS_j\geq-\alpha, \max_{k\leq j}(\MS_k-S_k)\leq \log n, \MS_j=S_j)\P(\MS_{\ell-j}\leq0, -S_{\ell-j}-\theta\log n\in[x,x+1])\\
\leq &\sum_{\ell=(\log n)^2/\varepsilon}^{c_0(\log n)^3}\sum_{j=\ell-(\log n)^2}^{\ell-(\log n)^{1+\delta}}c_{54}\frac{1+\alpha}{j}e^{-c_{55}\frac{j}{(\log n)^2}}\frac{1}{\ell-j}e^{-c_{56}\frac{(\log n)^2}{\ell-j}}\leq \sum_{\ell=(\log n)^2/\varepsilon}^{c_0(\log n)^3}c_{57}\frac{1+\alpha}{\ell}e^{-c_{58}\frac{\ell}{(\log n )^2}}
\end{align*}
which is $o_\varepsilon(1)$. We hence end up with 
\[
\sum_{\ell=(\log n)^2/\varepsilon}^{c_0(\log n)^3}\E_{\eqref{positivex}}(\ell)=o_n(1)+o_\varepsilon(1),
\]
which shows \eqref{positivex}. 

Let us turn to check \eqref{negativex}. Let 
\[
\E_{\eqref{negativex}}(\ell):=\sum_{k=0}^\ell \E\left[e^{S_k-\MS_\ell}\ind{\mS_\ell\geq-\alpha, \MS_\ell\geq\log n-a_n, \max_{k\leq \ell}(\MS_k-S_k)\leq \log n, \MS_\ell-S_\ell\in[\theta\log n+x,\theta\log n+x+1]}\right],
\]
and
\begin{equation}\label{negativexsum}
\E_{\eqref{negativexsum}}(j,k,\ell):=\E\left[e^{S_k-S_j}\ind{\MS_\ell=S_j>\MS_{j-1}}\ind{\mS_\ell\geq-\alpha, \MS_\ell\geq\log n-a_n, \max_{k\leq \ell}(\MS_k-S_k)\leq \log n, \MS_\ell-S_\ell\in[\theta\log n+x,\theta\log n+x+1]}\right].
\end{equation}
By considering the first time hitting $\MS_\ell$, one sees that 
\begin{align}
\sum_{\ell=(\log n)^2/\varepsilon}^{c_0(\log n)^3}\E_{\eqref{negativex}}(\ell)=&\sum_{\ell=(\log n)^2/\varepsilon}^{c_0(\log n)^3}\sum_{k=0}^{\ell}\sum_{j=1}^{\ell-1}\E_{\eqref{negativexsum}}(j,k,\ell)\nonumber\\
=&\sum_{\ell=(\log n)^2/\varepsilon}^{c_0(\log n)^3}\sum_{j=1}^{\ell-1}\sum_{k=0}^{j-1}\E_{\eqref{negativexsum}}(j,k,\ell)+\sum_{\ell=(\log n)^2/\varepsilon}^{c_0(\log n)^3}\sum_{j=1}^{\ell-1}\sum_{k=j}^{\ell}\E_{\eqref{negativexsum}}(j,k,\ell).\label{-xsumk}
\end{align}
For $k\geq j$, by Markov property at time $j$, 
\begin{multline}\label{Markovnegativex+}
\E_{\eqref{negativexsum}}(j,k,\ell)\leq \P(\mS_j\geq-\alpha, S_j=\MS_j\geq \log n-a_n, \max_{1\leq i\leq j}(\MS_i-S_i)\leq \log n)\\
\times\E\left[e^{S_{k-j}}\ind{\MS_{\ell-j}\leq0; -S_{\ell-j}\in[\theta\log n+x,\theta\log n+x+1]}\right].
\end{multline}
Similarly as above, we use different inequalities for different $j$ to bound the second sum on the right hand side of \eqref{-xsumk}. 
\begin{enumerate}
\item For $j\leq (\log n)^{1+\delta}$ with $\delta\in(0,1)$, by \eqref{Markovnegativex+}, \eqref{sumeSMSSbd} and \eqref{sumSbd}, one sees that
\begin{align*}
&\sum_{\ell=(\log n)^2/\varepsilon}^{c_0(\log n)^3} \sum_{j=1}^{(\log n)^{1+\delta}}\sum_{k=j}^\ell\E_{\eqref{negativexsum}}(j,k,\ell)\\
\leq &\sum_{\ell=(\log n)^2/\varepsilon}^{c_0(\log n)^3} \sum_{j=1}^{(\log n)^{1+\delta}}\P(S_j\geq \log n-a_n)\E\left[\sum_{k=0}^{\ell-j}e^{S_k}\ind{\MS_{\ell-j}\leq0; -S_{\ell-j}\in[\theta\log n+x,\theta\log n+x+1]}\right]\\
\leq & \sum_{\ell=(\log n)^2/\varepsilon}^{c_0(\log n)^3} \sum_{j=1}^{(\log n)^{1+\delta}}\P(S_j\geq \log n-a_n)c_{59}\frac{\theta\log n+a_n}{(\ell/2)^{3/2}}\leq c_{60}\sqrt{\varepsilon}e^{-c_{61}(\log n)^{1-\delta}},
\end{align*}
which is $o_n(1)$.
\item For $(\log n)^{1+\delta}\leq j\leq (\log n)^2$, by \eqref{Markovnegativex+}, \eqref{mSMSSlargebde} and \eqref{sumeSMSSbd}, one sees that
\begin{align*}
&\sum_{\ell=(\log n)^2/\varepsilon}^{c_0(\log n)^3}\sum_{j=(\log n)^{1+\delta}}^{(\log n)^2}\sum_{k=j}^\ell\E_{\eqref{negativexsum}}(j,k,\ell)\\
\leq&\sum_{\ell=(\log n)^2/\varepsilon}^{c_0(\log n)^3}\sum_{j=(\log n)^{1+\delta}}^{(\log n)^2}\P(\mS_j\geq-\alpha, S_j=\MS_j\geq \log n-a_n)\E\left[\sum_{k=0}^{\ell-j}e^{S_k}\ind{\MS_{\ell-j}\leq0; -S_{\ell-j}\in[\theta\log n+x,\theta\log n+x+1]}\right]\\
\leq &\sum_{\ell=(\log n)^2/\varepsilon}^{c_0(\log n)^3}\sum_{j=(\log n)^{1+\delta}}^{(\log n)^2}c_{62}\frac{1+\alpha}{\sqrt{j}\log n}e^{-c_{63}\frac{(\log n)^2}{j}}\frac{\theta\log n+a_n}{(\ell-j)^{3/2}}\leq \sum_{\ell=(\log n)^2/\varepsilon}^{c_0(\log n)^3}c_{64}(1+\alpha)\frac{\log n}{\ell^{3/2}},
\end{align*}
which is $o_\varepsilon(1)$.
\item For $(\log n)^2\leq j\leq \ell-(\log n)^2$, by \eqref{Markovnegativex+}, \eqref{mSMSvalleybd} and \eqref{sumeSMSSbd}, one sees that
\begin{align*}
&\sum_{\ell=(\log n)^2/\varepsilon}^{c_0(\log n)^3}\sum_{j=(\log n)^2}^{\ell-(\log n)^2}\sum_{k=j}^\ell\E_{\eqref{negativexsum}}(j,k,\ell)\\
\le&\sum_{\ell=(\log n)^2/\varepsilon}^{c_0(\log n)^3}\sum_{j=(\log n)^2}^{\ell-(\log n)^2}\P(\mS_j\geq-\alpha, S_j=\MS_j, \max_{1\leq i\leq j}(\MS_i-S_i)\leq \log n)\E\left[\sum_{k=0}^{\ell-j}e^{S_k}\ind{\MS_{\ell-j}\leq0; -S_{\ell-j}\in[\theta\log n+x,\theta\log n+x+1]}\right]\\
\leq &\sum_{\ell=(\log n)^2/\varepsilon}^{c_0(\log n)^3}\sum_{j=(\log n)^2}^{\ell-(\log n)^2}c_{65}\frac{1+\alpha}{j}e^{-c_{66}\frac{j}{(\log n)^2}}\frac{\theta\log n+a_n}{(\ell-j)^{3/2}}\leq \sum_{\ell=(\log n)^2/\varepsilon}^{c_0(\log n)^3}c_{67}(1+\alpha)\frac{\log n}{\ell^{3/2}},
\end{align*}
which is $o_\varepsilon(1)$.
\item For $\ell-(\log n)^2\leq j \leq \ell$, by \eqref{Markovnegativex+}, \eqref{mSMSvalleybd} and \eqref{2sumeSmSSbd}, one sees that
\begin{align*}
&\sum_{\ell=(\log n)^2/\varepsilon}^{c_0(\log n)^3}\sum_{j=\ell-(\log n)^2}^{\ell-1}\sum_{k=j}^\ell\E_{\eqref{negativexsum}}(j,k,\ell)\\
\le&\sum_{\ell=(\log n)^2/\varepsilon}^{c_0(\log n)^3}\sum_{\ell-j=1}^{(\log n)^2}\P(\mS_j\geq-\alpha, S_j=\MS_j, \max_{1\leq i\leq j}(\MS_i-S_i)\leq \log n)\E\left[\sum_{k=0}^{\ell-j}e^{S_k}\ind{\MS_{\ell-j}\leq0; -S_{\ell-j}\in[\theta\log n+x,\theta\log n+x+1]}\right]\\
\leq &\sum_{\ell=(\log n)^2/\varepsilon}^{c_0(\log n)^3}\sum_{j=\ell-(\log n)^2}^{\ell-1}c_{68}\frac{1+\alpha}{\ell}e^{-c_{69}\frac{j}{(\log n)^2}}\E\left[\sum_{k=0}^{\ell-j}e^{S_k}\ind{\MS_{\ell-j}\leq0; -S_{\ell-j}\in[\theta\log n+x,\theta\log n+x+1]}\right]\\
\leq &\sum_{\ell=(\log n)^2/\varepsilon}^{c_0(\log n)^3} c_{70}\frac{1+\alpha}{\ell}e^{-c_{71}\frac{\ell}{(\log n)^2}},
\end{align*}
which is $o_\varepsilon(1)$.
\end{enumerate}
Combining all these terms, we get that
\begin{equation}\label{-xk+sum}
\sum_{\ell=(\log n)^2/\varepsilon}^{c_0(\log n)^3}\sum_{j=1}^{\ell-1}\sum_{k=j}^{\ell}\E_{\eqref{negativexsum}}(j,k,\ell)=o_n(1)+o_\varepsilon(1).
\end{equation}
Next, let us bound $\sum_{\ell=(\log n)^2/\varepsilon}^{c_0(\log n)^3}\sum_{j=1}^{\ell-1}\sum_{k=0}^{j-1}\E_{\eqref{negativexsum}}(j,k,\ell)$. For $k<j$, Markov property at time $j$ implies that
\begin{multline*}
\E_{\eqref{negativexsum}}(j,k,\ell)\leq \E\left[e^{S_k-S_j}\ind{\mS_j\geq-\alpha, S_j>\MS_{j-1}, \MS_j\geq\log n-a_n, \max_{0\leq i\leq j}(\MS_i-S_i)\leq \log n}\right]\\
\times\P(\MS_{\ell-j}\leq0, -S_{\ell-j}-\theta\log n\in [x,x+1]).
\end{multline*}
If $\MS_k\leq \frac{1}{2}\log n$, then $\E_{\eqref{negativexsum}}(j,k,\ell)\leq e^{-\frac{1}{2}\log n+a_n}$. 
Therefore,
\begin{multline}
\E_{\eqref{negativexsum}}(j,k,\ell)\leq e^{-\frac{1}{2}\log n+a_n}+\E\left[e^{S_k-S_j}\ind{\mS_j\geq-\alpha, \MS_k\geq \frac{\log n}{2}, S_j>\MS_{j-1}, \MS_j\geq\log n-a_n, \max_{0\leq i\leq j}(\MS_i-S_i)\leq \log n}\right]\\
\times\P(\MS_{\ell-j}\leq0, -S_{\ell-j}-\theta\log n\in [x,x+1]).
\end{multline}
By Markov property at time $k$ and by the fact that $(S_{j}-S_{j-i})_{0\leq i\leq j}$ is distributed as $(S_i)_{0\leq i\leq j}$,
\begin{align*}
&\E\left[e^{S_k-S_j}\ind{\MS_k\geq \frac{1}{2}\log n, \mS_j\geq-\alpha, S_j>\MS_{j-1}, \MS_j\geq\log n-a_n, \max_{0\leq i\leq j}(\MS_i-S_i)\leq \log n}\right]\\
\leq &\E\left[\ind{\mS_K\geq-\alpha,\MS_k\geq \frac{1}{2}\log n,\max_{i\leq k}(\MS_i-S_i)\leq \log n}\E[e^{-S_{j-k}}\ind{S_{j-k}=\MS_{j-k}\geq x_0}]\vert_{x_0=\MS_k-S_k}\right]\\
=&\E\left[\ind{\mS_K\geq-\alpha,\MS_k\geq \frac{1}{2}\log n,\max_{i\leq k}(\MS_i-S_i)\leq \log n}\E[e^{-S_{j-k}}\ind{\mS_{j-k}\geq0, S_{j-k}\geq x_0}]\vert_{x_0=\MS_k-S_k}\right]
\end{align*}
which by \eqref{eSmSSbd} is less than $\frac{c_{72}}{(j-k)^{3/2}}\E\left[e^{(S_k-\MS_k)/2}\ind{\mS_K\geq-\alpha,\MS_k\geq \frac{1}{2}\log n,\max_{i\leq k}(\MS_i-S_i)\leq \log n}\right]$. As a result,
\begin{equation}\label{Markov-x-}
\sum_{\ell=(\log n)^2/\varepsilon}^{c_0(\log n)^3}\sum_{j=1}^{\ell-1}\sum_{k=0}^{j-1}\E_{\eqref{negativexsum}}(j,k,\ell)\leq o_n(1)+\sum_{\ell=(\log n)^2/\varepsilon}^{c_0(\log n)^3}\sum_{j=1}^{\ell-1}\sum_{k=1}^{j-1}\frac{c_{72}}{(j-k)^{3/2}}\E_{\eqref{Markov-x-}}(j,k,\ell),
\end{equation}
where 
\[
\E_{\eqref{Markov-x-}}(j,k,\ell):=\E\left[e^{(S_k-\MS_k)/2}\ind{\mS_K\geq-\alpha,\MS_k\geq \frac{1}{2}\log n,\max_{i\leq k}(\MS_i-S_i)\leq \log n}\right]\P(\MS_{\ell-j}\leq0, -S_{\ell-j}-\theta\log n\in [x,x+1]).
\]
For $j\leq \ell/2$, by \eqref{mSSbd}, we have
\begin{align*}
&\sum_{j=1}^{\ell/2}\sum_{k=1}^{j-1}\frac{c_{72}}{(j-k)^{3/2}}\E_{\eqref{Markov-x-}}(j,k,\ell)
\leq \sum_{j=1}^{\ell/2}\sum_{k=1}^{j-1}\frac{c_{73}}{(j-k)^{3/2}}\E\left[e^{(S_k-\MS_k)/2}\ind{\mS_K\geq-\alpha,\MS_k\geq \frac{1}{2}\log n,\max_{i\leq k}(\MS_i-S_i)\leq \log n}\right]\frac{\log n}{(\ell/2)^{3/2}}\\
\leq &\frac{\log n}{(\ell/2)^{3/2}}\sum_{k=1}^{\ell/2-1}\E\left[e^{(S_k-\MS_k)/2}\ind{\mS_K\geq-\alpha,\MS_k\geq \frac{1}{2}\log n,\max_{i\leq k}(\MS_i-S_i)\leq \log n}\right]\sum_{j=k+1}^{\ell/2}\frac{c_{73}}{(j-k)^{3/2}}\\
\leq & c_{74}\frac{\log n}{\ell^{3/2}}\left[\sum_{k=1}^{(\log n)^2}+\sum_{k=(\log n)^2}^{\ell/2}\E\left[e^{(S_k-\MS_k)/2}\ind{\mS_K\geq-\alpha,\MS_k\geq \frac{1}{2}\log n,\max_{i\leq k}(\MS_i-S_i)\leq \log n}\right]\right]
\end{align*}
By \eqref{eSMSMSbd} for $k\leq (\log n)^2$ and by \eqref{eSMSmSbd} for $k\geq (\log n)^2$, we see that
\begin{align*}
&\sum_{j=1}^{\ell/2}\sum_{k=1}^{j-1}\frac{c_{72}}{(j-k)^{3/2}}\E_{\eqref{Markov-x-}}(j,k,\ell)\\
\leq &c_{75}\frac{\log n}{\ell^{3/2}}\left[\sum_{k=1}^{(\log n)^2}\frac{1+\alpha}{\sqrt{k}\log n}+\sum_{k=(\log n)^2}^{\ell/2}(\frac{(1+\alpha)\log k}{k^{3/2}}+\frac{(1+\alpha)}{k}e^{-c_{76}\frac{k}{(\log n)^2}})\right]
\leq  c_{77} (1+\alpha)\frac{\log n}{\ell^{3/2}}.
\end{align*}
This yields that
\begin{equation}\label{Markov-x-sum1}
\sum_{\ell=(\log n)^2/\varepsilon}^{c_0(\log n)^3}\sum_{j=1}^{\ell/2}\sum_{k=1}^{j-1}\frac{c_{72}}{(j-k)^{3/2}}\E_{\eqref{Markov-x-}}(j,k,\ell)=o_\varepsilon(1).
\end{equation}
For $\ell/2\leq j\leq \ell-(\log n)^2$, by \eqref{mSSbd}, we have
\begin{align*}
&\sum_{j=\ell/2}^{\ell-(\log n)^2}\sum_{k=1}^{j-1}\frac{c_{72}}{(j-k)^{3/2}}\E_{\eqref{Markov-x-}}(j,k,\ell)\\
\leq &\sum_{j=\ell/2}^{\ell-(\log n)^2}\sum_{k=1}^{j-1}\frac{c_{78}}{(j-k)^{3/2}}\E\left[e^{(S_k-\MS_k)/2}\ind{\mS_K\geq-\alpha,\MS_k\geq \frac{1}{2}\log n,\max_{i\leq k}(\MS_i-S_i)\leq \log n}\right]\frac{\log n}{(\ell-j)^{3/2}}
%\leq & \sum_{j=\ell/2}^{\ell-(\log n)^2}\left[\sum_{k=1}^{(\log n)^2}+\sum_{k=(\log n)^2}^{j-1}\frac{c}{(j-k)^{3/2}}\E\left[e^{(S_k-\MS_k)/2}\ind{\mS_K\geq-\alpha,\MS_k\geq \frac{1}{2}\log n,\max_{i\leq k}(\MS_i-S_i)\leq \log n}\right]\right]c\frac{\log n}{\ell^{3/2}}.
\end{align*}
By \eqref{eSMSMSbd} for $k\leq (\log n)^2$ and by \eqref{eSMSmSbd} for $k\geq (\log n)^2$, we see that
\begin{align*}
&\sum_{j=\ell/2}^{\ell-(\log n)^2}\sum_{k=1}^{j-1}\frac{c_{72}}{(j-k)^{3/2}}\E_{\eqref{Markov-x-}}(j,k,\ell)\\
\leq &\sum_{j=\ell/2}^{\ell-(\log n)^2}\left[\sum_{k=1}^{(\log n)^2}\frac{c_{79}}{(j-k)^{3/2}}\frac{1+\alpha}{\sqrt{k}\log n}+\sum_{k=(\log n)^2}^{j-1}\frac{c_{80}}{(j-k)^{3/2}}(1+\alpha)(\frac{\log k}{k^{3/2}}+\frac{e^{-c_{81}\frac{k}{(\log n)^2}}}{k})\right]\frac{\log n}{(\ell-j)^{3/2}}\\
\leq &c_{82}(1+\alpha)(\frac{\log \ell}{\ell^{3/2}}+\frac{e^{-c_{83}\frac{\ell}{(\log n)^2}}}{\ell}).
\end{align*}
This leads to 
\begin{equation}\label{Markov-x-sum2}
\sum_{\ell=(\log n)^2/\varepsilon}^{c_0(\log n)^3}\sum_{j=\ell/2}^{\ell-(\log n)^2}\sum_{k=1}^{j-1}\frac{c}{(j-k)^{3/2}}\E_{\eqref{Markov-x-}}(j,k,\ell)=o_\varepsilon(1).
\end{equation}
For $\ell-(\log n)^2\leq j\leq \ell-1$, by \eqref{eSMSMSbd} for $k\leq (\log n)^2$ and by \eqref{eSMSmSbd} for $k\geq (\log n)^2$, one sees that
\begin{align*}
&\sum_{j=\ell-(\log n)^2}^{\ell-1}\sum_{k=1}^{j-1}\frac{c_{72}}{(j-k)^{3/2}}\E_{\eqref{Markov-x-}}(j,k,\ell)\\
\leq &\sum_{j=\ell-(\log n)^2}^{\ell-1}\left[\sum_{k=1}^{(\log n)^2}\frac{c_{84}}{(j-k)^{3/2}}\frac{1+\alpha}{\sqrt{k}\log n}+\sum_{k=(\log n)^2}^{j-1}\frac{c_{85}}{(j-k)^{3/2}}(1+\alpha)(\frac{\log k}{k^{3/2}}+\frac{e^{-c_{86}\frac{k}{(\log n)^2}}}{k})\right]\\
&\hspace{8cm}\times\P(\MS_{\ell-j}\leq0, -S_{\ell-j}-\theta\log n\in [x,x+1])\\
\leq & c_{87}(1+\alpha)(\frac{\log \ell}{\ell^{3/2}}+\frac{e^{-c_{88}\frac{\ell}{(\log n)^2}}}{\ell})\sum_{j=1}^{(\log n)^2}\P(\MS_j\leq0, -S_j-\theta\log n\in [x,x+1])
%&\leq &c(1+\alpha)(\frac{\log \ell}{\ell^{3/2}}+\frac{e^{-c\frac{\ell}{(\log n)^2}}}{\ell})
\end{align*}
which by \eqref{summSSbd} is less than
\[
c_{89}(1+\alpha)(\frac{\log \ell}{\ell^{3/2}}+\frac{e^{-c_{88}\frac{\ell}{(\log n)^2}}}{\ell}).
\]
Consequently, 
\begin{equation}\label{Markov-x-sum3}
\sum_{\ell=(\log n)^2/\varepsilon}^{c_0(\log n)^3}\sum_{j=\ell-(\log n)^2}^{\ell-1}\sum_{k=1}^{j-1}\frac{c_{72}}{(j-k)^{3/2}}\E_{\eqref{Markov-x-}}(j,k,\ell)=o_n(1)+o_\varepsilon(1).
\end{equation}
In view of \eqref{Markov-x-sum1}, \eqref{Markov-x-sum2} and \eqref{Markov-x-sum3}, we end up with 
\[
\sum_{\ell=(\log n)^2/\varepsilon}^{c_0(\log n)^3}\sum_{j=1}^{\ell-1}\sum_{k=1}^{j-1}\frac{c_{72}}{(j-k)^{3/2}}\E_{\eqref{Markov-x-}}(j,k,\ell)=o_n(1)+o_\varepsilon(1).
\]
This, combined with \eqref{Markov-x-}, \eqref{-xk+sum} and \eqref{-xsumk}, gives \eqref{negativex}. We thus conclude \eqref{largegeneration1ex++}.

\textbf{Proof of \eqref{badmaxV1ex++}.} Let
\[
\E_{\eqref{badmaxV1ex++}}(\ell):=\E\left[\sum_{|x|=\ell}na_x(1-a_x)^{n-1}b_x^{n^\theta-1}\ind{x\in\D_n}\ind{\MV(x)\in[\log n-a_n,\log n+a_n],\mV(x)\geq-\alpha}\right].
\]
We are going to show that 
\begin{equation}
\sum_{\ell=\varepsilon(\log n)^2}^{(\log n)^2/\varepsilon}\E_{\eqref{badmaxV1ex++}}(\ell)=o(n^{1-\theta}).
\end{equation}
Recall that $na_x(1-a_x)^{n-1}b_x^{n^\theta-1}\leq n^{1-\theta}e^{-V(x)}$. It then follows from Many-to-One Lemma that
\begin{align*}
&\sum_{\ell=\varepsilon(\log n)^2}^{(\log n)^2/\varepsilon}\E_{\eqref{badmaxV1ex++}}(\ell)
\leq n^{1-\theta}\sum_{\ell=\varepsilon(\log n)^2}^{(\log n)^2/\varepsilon}\E\left[\sum_{|x|=\ell}e^{-V(x)}\ind{x\in\D_n}\ind{\MV(x)\in[\log n-a_n,\log n+a_n],\mV(x)\geq-\alpha}\right]\\
=&n^{1-\theta}\sum_{\ell=\varepsilon(\log n)^2}^{(\log n)^2/\varepsilon}\P\left(\mS_\ell\geq-\alpha, \MS_\ell\in[\log n-a_n, \log n+a_n], \MS_\ell-S_\ell\in[\theta\log n-a_n,\theta\log n+a_n]\right).
\end{align*}
which by considering the first time hitting $\MS_\ell$ is equal to
\[
n^{1-\theta}\sum_{\ell=\varepsilon(\log n)^2}^{(\log n)^2/\varepsilon}\sum_{j=1}^{\ell-1}\P\left(\mS_\ell\geq-\alpha, \MS_{j-1}<S_j=\MS_\ell\in[\log n-a_n, \log n+a_n], S_j-S_\ell-\theta\log n\in[-a_n, a_n]\right).
\]
So by Markov property at time $j$, we get that
\begin{equation}\label{badMV1ex++sum}
\sum_{\ell=\varepsilon(\log n)^2}^{(\log n)^2/\varepsilon}\E_{\eqref{badmaxV1ex++}}(\ell)\leq n^{1-\theta}\sum_{\ell=\varepsilon(\log n)^2}^{(\log n)^2/\varepsilon}\sum_{j=1}^{\ell-1}\P_{\eqref{badMV1ex++sum}}(j,\ell),
\end{equation}
where
\[
\P_{\eqref{badMV1ex++sum}}(j,\ell):=\P(\mS_j\geq-\alpha,\MS_j=S_j\in[\log n-a_n, \log n+a_n])\P(\MS_{\ell-j}\leq0, -S_{\ell-j}-\theta\log n\in[-a_n, a_n]).
\]
We will divide th sum on the righe hand side of \eqref{badMV1ex++sum} into four parts: $\sum_{\ell=\varepsilon(\log n)^2}^{(\log n)^2/\varepsilon}\sum_{j=1}^{(\log n)^{1+\delta}}$  with $\delta\in(1/2,1)$, $\sum_{\ell=\varepsilon(\log n)^2}^{(\log n)^2/\varepsilon}\sum_{j=(\log n)^{1+\delta}}^{(\log n)^2}$, $\sum_{\ell=\varepsilon(\log n)^2}^{(\log n)^2/\varepsilon}\sum_{j=(\log n)^{2}}^{\ell-(\log n)^2}$ and $\sum_{\ell=\varepsilon(\log n)^2}^{(\log n)^2/\varepsilon}\sum_{j=\ell-(\log n)^2}^{\ell-1}$ and bound them separately.
\begin{enumerate}
\item For $1\leq j\leq (\log n)^{1+\delta}$ with $\delta\in(0,1)$, by \eqref{mSSbd} and \eqref{sumSbd}, one sees that
\begin{align*}
&\sum_{\ell=\varepsilon(\log n)^2}^{(\log n)^2/\varepsilon}\sum_{j=1}^{(\log n)^{1+\delta}}\P_{\eqref{badMV1ex++sum}}(j,\ell)\\
%&\sum_{j=1}^{(\log n)^{1+\delta}}\P(\mS_j\geq-\alpha,\MS_j=S_j\in[\log n-a_n, \log n+a_n])\P(\MS_{\ell-j}\leq0, -S_{\ell-j}\in[\theta\log n-a_n, \theta\log n+a_n])\\
\leq &\sum_{\ell=\varepsilon(\log n)^2}^{(\log n)^2/\varepsilon}\sum_{j=1}^{(\log n)^{1+\delta}}\P(S_j\geq \log n-a_n)c_{90}\frac{a_n\log n}{(\ell-j)^{3/2}}\leq \sum_{\ell=\varepsilon(\log n)^2}^{(\log n)^2/\varepsilon}c_{90}\frac{a_n\log n}{\ell^{3/2}}e^{-c_{91}(\log n)^{1+\delta}}=o_n(1).
\end{align*}
\item For $(\log n)^{1+\delta}\leq j\leq (\log n)^2$ with $\delta\in(1/2,1)$, by \eqref{mSMS-Sbd} and \eqref{mSSbd}, one gets that
\begin{align*}
&\sum_{\ell=\varepsilon(\log n)^2}^{(\log n)^2/\varepsilon}\sum_{j=(\log n)^{1+\delta}}^{(\log n)^2}\P_{\eqref{badMV1ex++sum}}(j,\ell)\\
\leq &\sum_{\ell=\varepsilon(\log n)^2}^{(\log n)^2/\varepsilon}\sum_{j=(\log n)^{1+\delta}}^{(\log n)^2}c_{92}(1+\alpha)^4\frac{a_n(\log n)^3}{j^3}\frac{a_n\log n}{(\ell-j)^{3/2}}\leq \sum_{\ell=\varepsilon(\log n)^2}^{(\log n)^2/\varepsilon}c_{92}(1+\alpha)^4\frac{a_n^2(\log n)^{2(1-\delta)}}{\ell^{3/2}}=o_n(1).
\end{align*}
\item For $(\log n)^2\leq j\leq \ell-(\log n)^2$, by \eqref{mSMSSintervalbd} and \eqref{mSSbd}, one sees that
\begin{align*}
&\sum_{\ell=\varepsilon(\log n)^2}^{(\log n)^2/\varepsilon}\sum_{j=(\log n)^2}^{\ell-(\log n)^2}\P_{\eqref{badMV1ex++sum}}(j,\ell)\\
\leq &\sum_{\ell=\varepsilon(\log n)^2}^{(\log n)^2/\varepsilon}\sum_{j=(\log n)^2}^{\ell-(\log n)^2}\P(\mS_j\geq-\alpha,S_j\in[\log n-a_n, \log n+a_n])\P(\MS_{\ell-j}\leq0, -S_{\ell-j}-\theta\log n\in[-a_n, a_n])\\
\leq &\sum_{\ell=\varepsilon(\log n)^2}^{(\log n)^2/\varepsilon}\sum_{j=(\log n)^2}^{\ell-(\log n)^2}c_{93}\frac{(1+\alpha)a_n}{j^{3/2}}\frac{a_n\log n}{(\ell-j)^{3/2}}\leq \sum_{\ell=\varepsilon(\log n)^2}^{(\log n)^2/\varepsilon}c_{94}(1+\alpha)\frac{a_n^2}{\ell^{3/2}}=o_n(1).
\end{align*}
\item For $\ell-(\log n)^2\leq j\leq \ell-1$, by \eqref{mSMSSintervalbd} and \eqref{summSSbd}, one gets that
\begin{align*}
&\sum_{\ell=\varepsilon(\log n)^2}^{(\log n)^2/\varepsilon}\sum_{j=\ell-(\log n)^2}^{\ell-1}\P_{\eqref{badMV1ex++sum}}(j,\ell)\\
\leq &\sum_{\ell=\varepsilon(\log n)^2}^{(\log n)^2/\varepsilon}c_{95}\frac{(1+\alpha)a_n}{\ell^{3/2}}\sum_{r=\theta\log n-a_n}^{\theta\log n+a_n}\sum_{j=1}^{(\log n)^2}\P(\MS_j\leq0, -S_j\in[r,r+1])\\
\leq &\sum_{\ell=\varepsilon(\log n)^2}^{(\log n)^2/\varepsilon}c_{96}(1+\alpha)\frac{a_n^2}{\ell^{3/2}}=o_n(1).
\end{align*}
\end{enumerate}
Going back to \eqref{badMV1ex++sum}, we deduce that $\sum_{\ell=\varepsilon(\log n)^2}^{(\log n)^2/\varepsilon}\E_{\eqref{badmaxV1ex++}}(\ell)$. This completes the proof of \eqref{badmaxV1ex++}.
\end{proof}

\begin{proof}[Proof of Lemma \ref{smallpart1exL}]

\textbf{Proof of \eqref{badvalley1ex}.}
Recall that for $x\in\B^+n\cap\D_n\cap\L_n$, one has $\p^\en(\eL_x(\tau_n)\geq n^\theta, E_x^{(n)}=1)=na_x(1-a_x)^{n-1}b_x^{n^\theta-1}=(1+o_n(1))n^{1-\theta}e^{-V(x)}f(\frac{n^\theta}{H_x})$ with $\cf(t)=te^{-t}$. It then follows that
\begin{align}\label{1stmoment1ex}
&\frac{1}{n^{1-\theta}}\e\left[\sum_{\ell=\varepsilon (\log n)^2}^{(\log n)^2/\varepsilon}\sum_{|z|=\ell}\ind{\eL_z(\tau_n)\geq n^\theta, E_z^{(n)}=1}\ind{z\in \B^+_n}\ind{z\in\D_n}\ind{\mV(z)\geq-\alpha, \gamma_n\leq \max_{x\leq z} H_z< n}\right]\nonumber\\
=&(1+o_n(1))\sum_{\ell=\varepsilon (\log n)^2}^{(\log n)^2/\varepsilon}\E\left[\sum_{|z|=\ell}e^{-V(z)}\cf(\frac{n^\theta}{H_z})\ind{z\in \B^+_n}\ind{z\in\D_n}\ind{\mV(z)\geq-\alpha, \gamma_n\leq \max_{x\leq z} H_z< n}\right]\nonumber\\
=&(1+o_n(1))\sum_{\ell=\varepsilon (\log n)^2}^{(\log n)^2/\varepsilon}\E\left[\cf(\frac{n^\theta}{H_\ell^S})\ind{\mS_\ell\geq-\alpha, \MS_\ell\geq \log n+a_n, \MS_\ell-S_\ell\in[\theta\log n-a_n,\theta\log n+a_n], \gamma_n\leq \max_{k\leq \ell}H_k^S\leq n}\right]
\end{align}
which is less than
\[
c\sum_{\ell=\varepsilon (\log n)^2}^{(\log n)^2/\varepsilon}\E\left[\cf(\frac{n^\theta}{H_\ell^S})\ind{\mS_\ell\geq-\alpha, \MS_\ell\geq \log n+a_n, \MS_\ell-S_\ell\in[\theta\log n-a_n,\theta\log n+a_n], \log n-r\log\log n-\log \ell\le \max_{k\leq \ell}(\MS_k-S_k)\leq \log n}\right],
\]
as $e^{\MS_k-S_k}\leq H_k^S\leq ke^{\MS_k-S_k}$. To conclude, we only need to show that for any $a_n=o(\log n)$, 
\begin{multline}\label{sumcvgeSMS}
\lim_{n\to\infty}\sum_{\ell=a(\log n)^2}^{A(\log n)^2}\E\left[\cf(\frac{n^\theta}{H_\ell^S})\ind{\mS_\ell\geq-\alpha, \MS_\ell\geq \log n+a_n, \MS_\ell-S_\ell\in[\theta\log n-a_n,\theta\log n+a_n], \max_{k\leq \ell}(\MS_k-S_k)\leq \log n+a_n}\right]\\
=\Ren(\alpha)\int_{a}^{A}\mathcal{G}(\frac{1}{\sqrt{u}},\frac{1}{\sqrt{u}},\frac{\theta}{\sqrt{u}})\frac{du}{u} 
\end{multline}
which follows immediately from \eqref{cvgeSMS} and \eqref{badend1ex}. By comparing the convergences for $a_n=0$ and $a_n=-(r+3)\log\log n$, we obtain what we want. 

\textbf{Proof of \eqref{badend1ex}.} Similarly as \eqref{1stmoment1ex}, we get that
\begin{align*}
&\frac{1}{n^{1-\theta}}\e\left[\sum_{\ell=\varepsilon (\log n)^2}^{(\log n)^2/\varepsilon}\sum_{|z|=\ell}\ind{\eL_z(\tau_n)\geq n^\theta, E_z^{(n)}=1}\ind{z\in \B^+_n}\ind{z\in\D_n\setminus\D_n^K}\ind{\mV(z)\geq-\alpha, \max_{x\leq z} H_z< n}\right]\\
=&(1+o_n(1))\sum_{\ell=\varepsilon (\log n)^2}^{(\log n)^2/\varepsilon}\E\left[\cf(\frac{n^\theta}{H_\ell^S})\ind{\mS_\ell\geq-\alpha, \MS_\ell\geq \log n+a_n, \MS_\ell-S_\ell\in[\theta\log n-a_n,\theta\log n+a_n]\setminus[\theta\log n-K,\theta\log n+K],  \max_{k\leq \ell}H_k^S\leq n}\right].
\end{align*}
Similarly as in the proof of \eqref{largegeneration1ex++}, one has
\begin{align*}
&\E\left[\cf(\frac{n^\theta}{H_\ell^S})\ind{\mS_\ell\geq-\alpha, \MS_\ell\geq \log n+a_n, \MS_\ell-S_\ell\in[\theta\log n-a_n,\theta\log n+a_n]\setminus[\theta\log n-K,\theta\log n+K],  \max_{k\leq \ell}H_k^S\leq n}\right]\\
\leq &
\sum_{\ell=\varepsilon (\log n)^2}^{(\log n)^2/\varepsilon}\sum_{j=1}^{\ell-1}\E\left[(\frac{n^\theta}{H_\ell^S}\wedge\frac{H_\ell^S}{n^\theta})\ind{\mS_\ell\geq-\alpha, \tau_\ell(\MS)=j, \MS_\ell\geq \log n+a_n, \MS_\ell-S_\ell\in[\theta\log n-a_n,\theta\log n+a_n]\setminus[\theta\log n-K,\theta\log n+K],  \max_{k\leq \ell}(\MS_k-S_k)\leq n}\right]\\
\leq &\sum_{\ell=\varepsilon (\log n)^2}^{(\log n)^2/\varepsilon}\sum_{x=-a_n}^{-K}e^x\sum_{j=1}^{\ell-1}\E\left[\sum_{k=0}^{\ell}e^{S_k-\MS_\ell}\ind{\mS_\ell\geq-\alpha, \tau_\ell(\MS)=j, \MS_\ell\geq \log n+a_n, \MS_\ell-S_\ell\in[\theta\log n+x,\theta\log n+x+1],  \max_{k\leq \ell}(\MS_k-S_k)\leq n}\right]\\
&+\sum_{\ell=\varepsilon (\log n)^2}^{(\log n)^2/\varepsilon}\sum_{x=K}^{a_n}e^{-x}\sum_{j=1}^{\ell-1}\E\left[\ind{\mS_\ell\geq-\alpha, \tau_\ell(\MS)=j, \MS_\ell\geq \log n+a_n, \MS_\ell-S_\ell\in[\theta\log n+x,\theta\log n+x+1],  \max_{k\leq \ell}(\MS_k-S_k)\leq n}\right]
\end{align*}
Using the same arguments as for \eqref{positivex} and \eqref{negativex}, one sees that
\[
\E\left[\cf(\frac{n^\theta}{H_\ell^S})\ind{\mS_\ell\geq-\alpha, \MS_\ell\geq \log n+a_n, \MS_\ell-S_\ell\in[\theta\log n-a_n,\theta\log n+a_n]\setminus[\theta\log n-K,\theta\log n+K],  \max_{k\leq \ell}H_k^S\leq n}\right]\leq c_{97}(1+\alpha)e^{-K},
\]
which is $o_K(1)$ as $K\to\infty$.
\end{proof}

\begin{proof}[Proof of Lemma \ref{1exvariance}]
Let us consider the quenched variance of $\Xi_n(\ell,\B_n^+\cap\D_n\cap\L_{\gamma_n},\alpha)$ which is
\begin{align}\label{variancebd}
&\V^\en(\Xi_n(\ell, \B^+_n\cap\D_n\cap\L_{\gamma_n},\alpha))=\e^\en\left[(\Xi_n(\ell,\B_n^+\cap\D_n\cap\L_{\gamma_n},\alpha)-\e^\en[\Xi_n(\ell,\B_n^+\cap\D_n\cap\L_{\gamma_n},\alpha)])^2\right]\nonumber\\
=&\sum_{|x|=\ell}na_x(1-a_x)^{n-1}b_x^{n^\theta-1}[1-na_x(1-a_x)^{n-1}b_x^{n^\theta-1}]\ind{x\in\B_n^+\cap\D_n\cap\L_{\gamma_n}, \mV(x)\geq-\alpha}+\Sigma_{\V}
%&+\sum_{|x|=|y|=\ell, x\neq y}[\p^\en(\eL_x(\tau_n)\geq n^\theta, E_x^{(n)}=1; \eL_y(\tau_n)\geq n^\theta, E_y^{(n)}=1)-\p^\en(\eL_x(\tau_n)\geq n^\theta, E_x^{(n)}=1)\p^\en(\eL_y(\tau_n)\geq n^\theta, E_y^{(n)}=1]\ind{x,y\in\B_n^+\cap\D_n\cap\L_{\gamma_n}, \mV(x)\geq-\alpha, \mV(y)\geq-\alpha}
\end{align}
where
\begin{multline}
\Sigma_{\V}:=\sum_{|x|=|z|=\ell, x\neq z}\ind{x,z\in\B_n^+\cap\D_n\cap\L_{\gamma_n}, \mV(x)\geq-\alpha, \mV(z)\geq-\alpha}\\
\times[\e^\en(\ind{\eL_x(\tau_n)\geq n^\theta, E_x^{(n)}=1}\ind{ \eL_z(\tau_n)\geq n^\theta, E_z^{(n)}=1})-n^2a_xa_z(1-a_x)^{n-1}b_x^{n^\theta-1}(1-a_z)^{n-1}b_z^{n^\theta-1}].
\end{multline}
On the one hand, for the first term on the right hand side of \eqref{variancebd}, as $\ell=\Theta((\log n)^2)$,
\begin{align*}
&\sum_{|x|=\ell}na_x(1-a_x)^{n-1}b_x^{n^\theta-1}[1-na_x(1-a_x)^{n-1}b_x^{n^\theta-1}]\ind{x\in\B_n^+\cap\D_n\cap\L_{\gamma_n}, \mV(x)\geq-\alpha}\\
\leq &\sum_{|x|=\ell}na_x(1-a_x)^{n-1}b_x^{n^\theta-1}\ind{x\in\B_n^+\cap\D_n\cap\L_{\gamma_n}, \mV(x)\geq-\alpha}
\end{align*}
whose expectation under $\E$ is $\Theta(\frac{n^{1-\theta}}{\ell})$ according to \eqref{cvgeSMS} and \eqref{1stmoment1ex}. For $x\neq z$, one sees that $\{E_x^{(n)}=E_z^{(n)}=1\}$ means that either $x$ and $z$ are visited in two different excursions or they are both visited in the same excursion. Let $a_{x,z}:=\p^\en_\rho(T_x\wedge T_z<T_{\rho^*})$. Then,
\begin{align*}
&\e^\en(\ind{\eL_x(\tau_n)\geq n^\theta, E_x^{(n)}=1}\ind{ \eL_z(\tau_n)\geq n^\theta, E_z^{(n)}=1})\\
=& n(n-1)a_xa_z(1-a_{x,z})^{n-2}(b_xb_z)^{n^\theta-1}+n(1-a_{x,z})^{n-1}\p^\en(\eL_x(\tau_1)\geq n^\theta, \eL_z(\tau_1)\geq n^\theta)\\
\leq & n^2 a_xa_z(1-a_{x,z})^{n-2}(b_xb_z)^{n^\theta-1}+n(1-a_{x,z})^{n-1}\e^\en\left[\frac{\eL_x(\tau_1)\eL_z(\tau_1)}{n^{2\theta}}\right]
\end{align*}
Let $u=x\wedge z$ be the latest common ancestor of $x$ and $z$. Say that $u_x$ is the child of $u$ such that $u_x\leq x$ and $u_z$ is the child of $u$ such that $u_z\leq z$. Then
\begin{align*}
\e^\en\left[\eL_x(\tau_1)\eL_z(\tau_1)\right]=&\e^\en\left[\eL_{u_x}(\tau_1)\eL_{u_z}(\tau_1)\right]e^{-V(x)-V(z)+V(u_x)+V(u_z)}\\
=&\e^\en\left[\eL_{u}(\tau_1)(\eL_{u}(\tau_1)+1)\right]e^{-V(x)-V(z)+2V(u)}\\
=& 2H_ue^{-V(u)}\times e^{-V(x)-V(z)+2V(u)}
\end{align*}
It follows that
\begin{align*}
\Sigma_{\V}\leq &\sum_{x\neq z, |z|=|x|=\ell} n^2 a_x a_z (b_xb_z)^{n^\theta-1}[(1-a_{x,z})^{n-2}-(1-a_x)^{n-1}(1-a_z)^{n-1}]\ind{x,z\in \B_n^+\cap\D_n\cap\L_{\gamma_n}, \mV(x)\geq-\alpha, \mV(z)\geq-\alpha}\\
&+\sum_{k=0}^{\ell-1}\sum_{|u|=k}\sum_{|x|=|z|=\ell, x\wedge z=u}2n^{1-2\theta}H_ue^{-V(u)}\times e^{-V(x)-V(z)+2V(u)}\ind{x,z\in \B_n^+\cap\D_n\cap\L_{\gamma_n},\mV(x)\geq-\alpha, \mV(z)\geq-\alpha}
\end{align*}
By Lemma 4.2 of \cite{AC18}, $(1-a_{x,z})^{n-2}-(1-a_x)^{n-1}(1-a_z)^{n-1}\leq na_z+na_x$. Moreover, $a_x\leq e^{-\MV(x)}\leq e^{-\log n-a_n}$ for $x\in\B_n^+$. Consequently,
\begin{align*}
& \sum_{x\neq z, |z|=|x|=\ell} n^2 a_x a_z (b_xb_z)^{n^\theta-1}[(1-a_{x,z})^{n-2}-(1-a_x)^{n-1}(1-a_z)^{n-1}]\ind{x,z\in \B_n^+\cap\D_n\cap\L_{\gamma_n},\mV(x)\geq-\alpha, \mV(z)\geq-\alpha }\\
\leq &n^{2-2\theta}\sum_{k=0}^{\ell-1}\sum_{|u|=k}\sum_{|x|=|z|=\ell, x\wedge z=u} e^{-V(x)-V(z)}[na_z+na_x]\ind{x,z\in \B_n^+\cap\D_n\cap\L_{\gamma_n},\mV(x)\geq-\alpha, \mV(z)\geq-\alpha}\\
\leq & 2e^{-a_n}n^{2-2\theta}\sum_{k=0}^{\ell-1}\sum_{|u|=k}\sum_{|x|=|z|=\ell, x\wedge z=u} e^{-V(x)-V(z)}\ind{x,z\in \B_n^+\cap\D_n\cap\L_{\gamma_n},\mV(x)\geq-\alpha, \mV(z)\geq-\alpha}
\end{align*}
As $x,z\in\L_{\gamma_n}$, $H_u\leq \gamma_n$. So,
\begin{align*}
\Sigma_{\V}\leq & \frac{2n^{2-2\theta}}{(\log n)^a}\sum_{k=0}^{\ell-1}\sum_{|u|=k}\sum_{|x|=|z|=\ell, x\wedge z=u} e^{-V(x)-V(z)}\ind{x,z\in \B_n^+\cap\D_n\cap\L_{\gamma_n},\mV(x)\geq-\alpha, \mV(z)\geq-\alpha}\\
&+2\frac{n^{2-2\theta}}{(\log n)^\gamma}\sum_{k=0}^{\ell-1}\sum_{|u|=k}\sum_{|x|=|z|=\ell, x\wedge z=u}e^{-V(x)-V(z)+V(u)}\ind{x,z\in \B_n^+\cap\D_n\cap\L_{\gamma_n},\mV(x)\geq-\alpha, \mV(z)\geq-\alpha}\\
\leq &\left(\frac{2n^{2-2\theta}}{(\log n)^a}e^\alpha+\frac{2n^{2-2\theta}}{(\log n)^\gamma}\right)\sum_{k=0}^{\ell-1}\sum_{|u|=k}\sum_{|x|=|z|=\ell, x\wedge z=u}e^{-V(x)-V(z)+V(u)}\ind{x,z\in \B_n^+\cap\D_n\cap\L_{\gamma_n},\mV(x)\geq-\alpha, \mV(z)\geq-\alpha},
\end{align*}
since $V(u)\geq\mV(x)\geq-\alpha$. Observe that
\begin{align*}
&\E\left[\sum_{k=0}^{\ell-1}\sum_{|u|=k}\sum_{|x|=|z|=\ell, x\wedge z=u}e^{-V(x)-V(z)+V(u)}\ind{x,z\in \B^+_n\cap\D_n\cap\L_{\gamma_n},\mV(x)\geq-\alpha, \mV(z)\geq-\alpha}\right]\\
\leq&\E\left[\sum_{k=0}^{\ell-1}\sum_{|u|=k}\sum_{\substack{u_z^*=u_x^*=u \\ u_z\neq u_x}}e^{-V(u_x)-\Delta V(u_z)}\sum_{z>u_z, |z|=\ell}e^{-[V(z)-V(u_z)]}\ind{\mV(z)\geq-\alpha}\sum_{x>u_x, |x|=\ell}e^{-[V(x)-V(u_x)]}\ind{\mV(x)\geq-\alpha}\right]\\
=&\E\left[\sum_{k=0}^{\ell-1}\sum_{|u|=k}\sum_{\substack{u_z^*=u_x^*=u \\ u_z\neq u_x}}e^{-V(u_x)-\Delta V(u_z)}\ind{\mV(u_x)\wedge\mV(u_z)\geq-\alpha}\P_{V(u_z)}(\mS_{\ell-1-k}\geq-\alpha)\P_{V(u_x)}(\mS_{\ell-1-k}\geq-\alpha)\right],
\end{align*}
where the last equality follows from Many-to-One Lemma. By \eqref{mSbd}, we deduce that
\begin{align*}
&\E\left[\sum_{k=0}^{\ell-1}\sum_{|u|=k}\sum_{\substack{u_z^*=u_x^*=u \\ u_z\neq u_x}}e^{-V(u_x)-\Delta V(u_z)}\ind{\mV(u_x)\wedge\mV(u_z)\geq-\alpha}\P_{V(u_z)}(\mS_{\ell-1-k}\geq-\alpha)\P_{V(u_x)}(\mS_{\ell-1-k}\geq-\alpha)\right]\\
\leq &c_{98}\E\left[\sum_{k=0}^{\ell-1}\sum_{|u|=k}(1+\alpha+V(u))^2e^{-V(u)}\ind{\mV(u)\geq-\alpha}\sum_{\substack{u_z^*=u_x^*=u \\ u_z\neq u_x}}e^{-\Delta V(u_x)-\Delta V(u_z)}\frac{(1+\Delta_+ V(u_z))(1+\Delta_+ V(u_x))}{\ell-k}\right]\\
\leq &\sum_{k=0}^{\ell-1}\frac{c_{99}}{\ell-k}\E[(1+\alpha+S_k)^2; \mS_k\geq-\alpha]\leq \sum_{k=0}^{\ell-1}c_{100}\frac{k+(1+\alpha)^2}{\ell-k}\leq c_{101}\ell^2.
\end{align*}
We therefore end up with
\[
\Sigma_{\V}\leq \frac{c_{102}n^{2-2\theta}}{(\log n)^{a\wedge r}}\ell^2\leq \frac{c_{103}n^{2-2\theta}}{(\log n)^{a\wedge r-4}}.
\]
which suffices to conclude Lemma \ref{1exvariance}.
\end{proof}

\appendix
\section{Appendix}
\subsection{Proof of Lemma \ref{sumGeo}}\label{A1}
\begin{proof}
We first prove \eqref{sumGeosmall}. Observe that as $b\in(0,1)$, for any $\lambda>0$, by Markov inequality,
\begin{align*}
\P\left(\sum_{i=1}^n\zeta_i\leq A\right)=&\P\left(e^{-\lambda(1-b) \sum_{i=1}^n\zeta_i}\geq e^{-\lambda(1-b) A}\right)\\
\leq & e^{\lambda (1-b)A}\E\left[e^{-\lambda(1-b) \zeta_1}\right]^n
\end{align*}
where $\E[e^{-\lambda(1-b)\zeta_1}]=1-\frac{a(e^{\lambda(1-b)}-1)}{e^{\lambda(1-b)}-b}$. We have $1-x\leq e^{-x}$ for any $x\in[0,1]$. It follows that
\begin{align*}
\P\left(\sum_{i=1}^n\zeta_i\leq A\right)\leq & \exp\{\lambda (1-b)A-na \frac{(e^{\lambda(1-b)}-1)}{e^{\lambda(1-b)}-b}\}\\
=&\exp\{\lambda (1-b)A-na \frac{(e^{\lambda(1-b)}-1)}{(e^{\lambda(1-b)}-1)+(1-b)}\}
\end{align*}
Since $0<1-b\leq \frac{e^{\lambda(1-b)}-1}{\lambda}$, one gets $\frac{(e^{\lambda(1-b)}-1)}{(e^{\lambda(1-b)}-1)+(1-b)}\geq \frac{\lambda}{\lambda+1}$ and then
\[
\P\left(\sum_{i=1}^n\zeta_i\leq A\right)\leq e^{-\lambda(\frac{na}{1+\lambda}-(1-b)A)}, \forall n\geq1.
\]
Let us turn to check \eqref{sumGeolarge} and \eqref{sumGeolarge+}. We only prove \eqref{sumGeolarge}, \eqref{sumGeolarge+} follows from similar arguments. Note that for any $s\in[1,\frac{1}{b})$, Markov inequality implies that
\begin{align*}
\P\left(\sum_{i=1}^n\zeta_i\geq n^\theta\right)\leq & s^{-n^\theta}\E\left[ s^{\sum_{i=1}^n\zeta_i}; \sum_{i=1}^n\ind{\zeta_i\geq1}\geq 1\right]
=\frac{\E[s^{\zeta_1}]^n-\P(\sum_{i=1}^n\ind{\zeta_i\geq1}=0)}{s^{n^\theta}}\\
=&\frac{1}{s^{n^\theta}}\left[(1-a+\frac{a(1-b)s}{1-bs})^n-(1-a)^n\right]\\
\leq & \frac{1}{s^{n^\theta}} \frac{na(1-b)s}{1-bs}(1-a+\frac{a(1-b)s}{1-bs})^{n-1},
\end{align*}
since $(1-a+x)^n-(1-a)^n\leq n x(1-a+x)^{n-1}$ for any $x>0$. Now take $s=\frac{1+\delta b}{(1+\delta)b}$ with some $\delta >0$. Apparently, $s\in [1,\frac{1}{b})$ and for any $\eta\in(0,1)$, there exists $M_\eta>1$ such that $\log(1+\frac{1-b}{(1+\delta)b})\geq (1-\eta/3)\frac{1-b}{(1+\delta)b}$ as long as $\delta b\geq M_\eta $. Consequently, for $\delta\geq M_\eta/b>0$,
\begin{align*}
\P\left(\sum_{i=1}^n\zeta_i\geq n^\theta\right)\leq & na \frac{1+\delta b}{\delta b} (1+\frac{1-b}{(1+\delta)b})^{-n^\theta}(1+\frac{a}{\delta b})^{n-1}\\
\leq &2 (na) e^{-(1-\eta/3)\frac{n^\theta(1-b)}{(1+\delta )b}+n\frac{a}{\delta b}}.
\end{align*}
Now we take $\eta\in(0,1)$ such that $n^\theta(1-b)>na(1+\eta)$ and $\delta=\max\{\frac{M_\eta}{b}, \frac{2}{\eta-\eta^2}\}$ so that
\[
\frac{na}{\delta b}\leq \frac{n^\theta(1-b)}{(1+\eta)\delta b}= \frac{n^\theta(1-b)}{(1+\delta)b}\frac{1+\delta}{\delta(1+\eta)}\leq (1-\eta/2 ) \frac{n^\theta(1-b)}{(1+\delta)b}
\]
This yields that
\[
\P\left(\sum_{i=1}^n\zeta_i\geq n^\theta\right)\leq 2(na)e^{-\frac{\eta}{6(1+\delta)b}n^\theta(1-b)},
\]
where $(1+\delta)b\leq 1+M_\eta+\frac{2}{\eta-\eta^2}$. We hence conclude \eqref{sumGeolarge} with $c_\eta=\frac{\eta}{6(1+M_\eta+\frac{2}{\eta-\eta^2})}\in(0,\infty)$.
\end{proof}

\subsection{Results on one-dimensional random walks}\label{A2}
We state some facts and inequalities on centred random walk $(S_n)_{n\geq0}$ introduced in the Many-to-One Lemma. The proofs are postponed in Section \ref{A3}.

Let $\xi_n:=S_n-S_{n-1}$ for any $n\geq1$. Note that $\E[\xi_1]=0$, $\sigma^2=\E[\xi_1^2]<\infty$. Moreover, by \eqref{Integrability},
\[
\E[e^{-\delta_0\xi_1}+e^{(1+\delta_0)\xi_1}]<\infty.
\]
We start with some well known inequalities (see \cite{AC18} for instance). Recall that $\mS_n=\min_{0\leq k\leq n}S_k$ and $\MS_n=\max_{0\leq k\leq n}S_k$. Note that the inequalities in the following hold also for the random walk $(-S_n)_{n\geq0}$. For any $\alpha\geq0$ and $n\geq1$, we have
\begin{equation}\label{mSbd}
\P(\mS_n\geq-\alpha)\leq \frac{C_4(1+\alpha)}{\sqrt{n}}\textrm{ and } \P(\MS_n\leq \alpha)\leq \frac{C_4(1+\alpha)}{\sqrt{n}}.
\end{equation}
For any $\alpha\geq0$, $b\geq a\geq-\alpha$ and for any $n\geq1$,
\begin{equation}\label{mSSbd}
\P(\mS_n\geq-\alpha, S_n\in[a,b])=\P_\alpha(\mS_n\ge0, S_n\in[\alpha+a,\alpha+b])\leq \frac{C_5(1+\alpha)(1+b+\alpha)(1+b-a)}{n^{3/2}}.
\end{equation}
We define the renewal function associated with the strict descending ladder process as follows:
\begin{equation}\label{renewalf}
\Ren(u):=\sum_{k=0}^\infty\P(S_k<\mS_{k-1}, S_k\geq-u), \forall u\geq0.
\end{equation}
It is known from Renewal theorem that 
\begin{equation}\label{cvgrenewalf}
\frac{1}{u}\Ren(u)\rightarrow\cb_\Ren \textrm{ as }u\to\infty.
\end{equation}
Moreover there exist $0<C_6<C_7<\infty$ such that for any $u\geq0$,
\[
C_6(1+u)\leq \Ren(u)\leq C_7(1+u).
\]
Recall that there exists some positive constant $\cb_+$ such that $\P(\mS_n\geq0)\sim \frac{\cb_+}{\sqrt{n}}$ as $n\to\infty$. According to Lemma 2.1 of \cite{AS14},
\begin{equation}\label{prodc}
\cb_\Ren\cb_+=\sqrt{\frac{2}{\pi\sigma^2}}.
\end{equation}
\begin{fact}
\begin{enumerate}
\item For any $u,\alpha\geq0$ and for any $n\geq1$,
\begin{equation}\label{mSMSbd}
\P_u(\mS_n\geq-\alpha, S_n=\MS_n)\leq \frac{C_8(1+\alpha+u)}{n}.
\end{equation}
\item For any $n\geq1$ and $A>0$, $\alpha\geq0$,
\begin{equation}\label{mSMSlargebd}
\P(\mS_n\geq-\alpha, S_n=\MS_n\geq A)\leq \frac{C_9(1+\alpha)}{A\sqrt{n}}.
\end{equation}
\item For any $B>0$ fixed, there exists $c(B)>0$ such that for any $n\geq1$ and $-B\sqrt{n}\leq -\alpha\le 0<a<b\leq B\sqrt{n}$, 
\begin{equation}\label{mSMSSintervalbd}
\P(S_n\geq-\alpha, S_n\in[a,b])\leq \frac{c(B)(1+\alpha)(b-a)}{n^{3/2}}.
\end{equation}
\item For $A>0$ sufficiently large and any $\lambda>0$, $\alpha\geq0$ and $n\geq1$,
\begin{equation}\label{eSMSmSbd}
\E[e^{\lambda(S_n-\MS_n)}; \max_{1\leq k\leq n}(\MS_k-S_k)\leq A, \mS_n\geq-\alpha]\leq C_{10}(1+\alpha)[\frac{\log n}{n^{3/2}}+\frac{1}{n}e^{-C_{11}n/A^2}].
\end{equation}
\item For any $A\geq1$, $\lambda>0$, $\alpha\ge0$ and $n\geq1$,
\begin{equation}\label{eSMSMSbd}
\E[e^{\lambda(S_n-\MS_n)}; \MS_n\geq A, \mS_n\geq-\alpha]\leq \frac{C_{12}(1+\alpha)}{A\sqrt{n}}.
\end{equation}
\item For $\alpha\ge0$ and $A\geq1$ sufficiently large,
\begin{equation}\label{mSMSvalleybd}
\P(\mS_n\geq-\alpha, \MS_n=S_n,\max_{1\leq k\leq n}(\MS_k-S_k)\leq A)\leq C_{13}\frac{1+\alpha}{n}e^{-C_{14} \frac{n}{A^2}}.
\end{equation}
\item As $x\to\infty$,
\begin{equation}\label{sumeSmSbd}
\E_x\left[\sum_{n=0}^\infty e^{-S_n/4}\ind{\mS_n\geq0}\right]=o_x(1)\Ren(x).
\end{equation}
\end{enumerate}
\end{fact}

According to \cite{Afa93}, conditioned on $\{\mS_n\geq 0\}$, the rescaled path $(\frac{S_{\lfloor nt\rfloor}}{\sqrt{n}}; 0\leq t\leq 1)$ and $\sum_{i=0}^n e^{-S_n}$ converge jointly in law to a Brownian meander $(m_t, t\in[0,1])$ and a positive random variable $\Hinf_\infty$ which is independent of the Brownian meander. One can refer to \cite{AC18} for more details. Let us state (A.12) of \cite{AC18} here.

\begin{fact}\label{mSMScvg}
Let $\alpha\geq0$, $a,b>0$ fixed and $a_n=o(\sqrt{n})$, $b_n=o(\sqrt{n})$. For any uniformly continuous and bounded function $g:[1,\infty)\rightarrow\R_+$, we have
\begin{equation}
\lim_{n\rightarrow\infty}n\E\left[g(\sum_{j=1}^n e^{S_j-S_n})\ind{\mS_n\geq-\alpha, S_n>\MS_{n-1}, \max_{1\leq i\leq n}(\MS_i-S_i)\leq a\sqrt{n}+a_n, S_n\geq b\sqrt{n}+b_n}\right]=\mathcal{C}_{a,b}\Ren(\alpha)\E[g(\Hinf_\infty)].
\end{equation}
where $\Ren$ is the renewal function and $\mathcal{C}_{a,b}$ is  defined in (3.20) of \cite{AC18}.
\end{fact}

The previous two Facts can be found in \cite{AC18}. The following lemmas state some inequalities that will be proved in Appendix \ref{A3}.

\begin{lem}\label{lemA3}
Let $\alpha\geq0$. There exists $\varepsilon_0\in(0,1)$ such that for $m$ sufficiently large and for any $1\leq r\leq\varepsilon_0 m$, we have
\begin{equation}\label{mSSbde}
\P(\mS_m\geq-\alpha, S_m\in[r,r+1])\leq C_{15}\frac{1+\alpha}{m}e^{-C_{16}\frac{r^2}{m}},
\end{equation}
and 
\begin{equation}\label{mSSlargebde}
\P(\mS_m\geq-\alpha, S_m\geq r)\leq C_{17}\frac{1+\alpha}{r}e^{-C_{18}\frac{r^2}{m}}.
\end{equation}
Moreover,
\begin{equation}\label{mSMSSlargebde}
\P(\mS_m\geq-\alpha, \MS_m=S_m\geq r)\leq C_{19}\frac{1+\alpha}{\sqrt{m} r}e^{-C_{20}\frac{r^2}{m}}.
\end{equation}
\end{lem}
%{\color{blue}Only check them for $r\geq A\sqrt{m}$.}
\begin{lem}
\begin{enumerate}
%\item For any $n\geq1$ and $\alpha>0$,
%\begin{equation}\label{mSS2bd}
%\E[S_n^2; \mS_n\geq -\alpha]\leq c(1+\alpha)\sqrt{n}.
%\end{equation}
\item For $\delta\in[0,1)$ and $A\geq1$ sufficiently large,
\begin{equation}\label{sumSbd}
\sum_{k=1}^{A^{1+\delta}}\P(S_k\geq A)\leq e^{-C_{21} A^{1-\delta}}.
\end{equation}
\item Let $\alpha\geq0$, for any $n\geq1$ and $r\geq0$,
\begin{equation}\label{mSMS-Sbd}
\P(\mS_n\geq-\alpha, \MS_n=S_n\in[r,r+1])\leq C_{22}(1+\alpha)^4 \frac{(1+r)^3}{n^3}
\end{equation}
\item Let $\eta>0$, $\alpha\geq0$. For $r$ sufficiently large, one has
\begin{equation}\label{summSSbd}
\sum_{1\leq k\leq \eta r^2}\P(\mS_k\geq-\alpha, S_k\in[r,r+1])\leq C_{23}(1+\alpha)\eta.
%\begin{cases}
%\frac{c_{20}(1+\alpha)\eta n}{r^2}\textrm{ if }\eta\in(0,\epsilon_0];\\
%c'_{20}(1+\alpha)^2 (\frac{\epsilon}{a^2}+\frac{b}{\sqrt{\epsilon}})\textrm{ if }\eta>\epsilon_0,
%\end{cases}
\end{equation}
%where $\epsilon_0$ is the same as in Lemma \ref{lemA3}. 
Moreover,
\begin{equation}\label{2sumeSmSSbd}
\sum_{k=1}^{\eta r^2}\E\left[\sum_{i=0}^{k}e^{-S_i}; \mS_k\geq 0, S_k\in[r,r+1]\right]\leq C_{24}\eta.
\end{equation}
%{\color{blue}In fact, we only need to take $S_i\leq 2\log n$ and use Markov property at time $k$.}
\item For any $x\geq0$ and $n\geq1$,
\begin{equation}\label{sumeSMSSbd}
\E\left[\sum_{k=0}^n e^{S_k}; \MS_n\leq 0, S_n\in[-x-1,-x]\right]\leq C_{25}\frac{1+x}{n^{3/2}}.
\end{equation}
\item For any $A\geq 0$, $\alpha\geq0$ and $n\geq1$, 
\begin{equation}\label{eSmSSbd}
\E_\alpha[e^{-S_n}; \mS_n\geq0, S_n\geq A]\leq \frac{C_{26}(1+\alpha)}{n^{3/2}}e^{-A/2}.
\end{equation}
\item For any $\alpha, A>0$ and $n\geq1$,
\begin{equation}\label{eSmSSsmallbd}
\E_\alpha[e^{S_n-A}; \mS_n\geq0, S_n\leq A]\leq C_{27}\frac{(1+\alpha)(1+A)}{n^{3/2}}.
\end{equation}
\item There exists $c\in\R_+^*$ such that for any $A>0$,
\begin{equation}\label{sumeSmSSsmallbd}
\sum_{n\ge 0}\E[e^{S_n-A}; \mS_n\ge 0, S_n\le A]<C_{28}.
\end{equation}
\item For $\alpha\ge0$, $a, b, c>0$, $K\geq1$, $n\leq Ar^2$ with $A>0$,  
\begin{equation}\label{mSMSMS-Sbd}
\P(\mS_n\geq-\alpha, \MS_n\geq a r, \max_{k\leq n}(\MS_k-S_k)\leq b r, \MS_n-S_n\in[cr-K,cr+K])\leq C_{29}(1+\alpha)\frac{(1+K^2)r}{n^{3/2}}
\end{equation}
\item For $a,b,\eta>0$ and $r\gg1$ sufficiently large, 
\begin{equation}\label{summSMSMS-Sbd}
\sum_{k=1}^{\eta r^2}\P(\mS_k\geq0, \MS_k\geq ar, \MS_k-S_k\in[br, br+1])\leq C_{30}(a,b)\eta^{3/2}.
\end{equation}
\end{enumerate}
\end{lem}

The following lemma focus on asymptotic results that we need. 
\begin{lem}
Let $\alpha\ge0$. Then the following convergences hold.
\begin{enumerate}
\item For any continuous and bounded function $g:[0,\infty)\rightarrow\mathbb{R}_+$, the following convergence holds uniformly for $x,y$ in any compact set of $(0,\infty)$ and for $z=o(\sqrt{n})$, $h>0$,
\begin{multline}\label{cvgsumeSmSSn}
\E\left[g(\sum_{i=1}^n e^{-S_i})\ind{\mS_n\geq-\alpha, \MS_n\leq x\sqrt{n}, S_n\in[y\sqrt{n}+z, y\sqrt{n}+z+h)}\right]=\frac{\cb_+h\Ren(\alpha)}{\sigma n}\E_\alpha[g(e^\alpha \Hinf_\infty-1)]\mathcal{C}_0(x-y,y)+\frac{o_n(1)}{n},
\end{multline}
where 
\begin{equation}\label{defconstantC}
\mathcal{C}_0(a,b)=\varphi(\frac{b}{\sigma})\P(\overline{R}_1-R_1\leq \frac{a}{\sigma}\vert R_1=\frac{b}{\sigma}),
\end{equation}
and $\cb_+=\lim_{n\rightarrow\infty}\sqrt{n}\P(\mS_n\geq0)$.
\item Let $a,b>0$ be fixed constants. For $F(x,y)=\frac{x}{y}e^{-x/y}$ with $x\in\R$ and $y\geq1$ and for $a_n=o(\sqrt{n})$, $a'_n=o(\sqrt{n})$ and fixed $K>0$, we have
\begin{multline}\label{cvgeSMS}
\lim_{n\to\infty}n\E\left[F(e^{b\sqrt{n}-(\MS_n-S_n)},\sum_{i=0}^n e^{S_i-\MS_n})\ind{\mS_n\geq-\alpha, \MS_n\geq a\sqrt{n}+a_n, \max_{0\leq k\leq n}(\MS_k-S_k)\leq a\sqrt{n}+a'_n, \MS_n-S_n\in[b\sqrt{n}-K, b\sqrt{n}+K]} \right]\\
=\mathcal{G}(a,b)\Ren(\alpha)\int_{-K}^K\E[(F(e^{-s}, \Hinf_\infty+\Hinf^{(-)}_\infty-1)]ds,
\end{multline}
where 
\begin{equation}\label{defconstantg}
\mathcal{G}(a,b):=\int_0^1 \mathcal{C}_{\frac{a}{\sqrt{u}},\frac{a}{\sqrt{u}}}\frac{\cb_-}{\sigma}\mathcal{C}_0(\frac{a-b}{\sqrt{1-u}},\frac{b}{\sqrt{1-u}})\ind{a>b}\frac{du}{u(1-u)}
\end{equation}
 with $\mathcal{C}_{a,b}$ defined in (3.20) of \cite{AC18}, $\cb_-:=\lim{n\to\infty}\sqrt{n}\P(\MS_n\le 0)$ and $\Hinf^{(-)}_\infty:=\sum_{k=0}^\infty e^{-\zeta_k^{(-)}}$ with $(\zeta_k^{(-)})_{k\geq0}$ the Markov chain obtained from the reflected walk $-S$. Moreover, this convergences holds uniformly for $a,b$ in any compact set of $(0,\infty)$.
\end{enumerate}
\end{lem}

The following result is a direct consequence of \eqref{cvgsumeSmSSn}.
\begin{cor}
Let $\alpha\ge0$ and $a,b>0$. For $a_n=o(\sqrt{n})$ and $b_n=o(\sqrt{n})$, the following convergence holds.
\begin{align}
\lim_{n\to\infty}n\E\left[e^{S_n-b\sqrt{n}-b_n}; \mS_n\geq-\alpha, \MS_n-S_n\leq a\sqrt{n}+a_n, S_n\leq b\sqrt{n}+b_n \right]&=\frac{\cb_+\Ren(\alpha)}{\sigma} \mathcal{C}_0(a,b);\label{eSmScvg}\\
\lim_{n\to\infty}n\E\left[e^{S_n-b\sqrt{n}-b_n}; \mS_n\geq-\alpha, \MS_n \leq (a+b)\sqrt{n}+a_n, S_n\leq b\sqrt{n}+b_n \right]&= \frac{\cb_+\Ren(\alpha)}{\sigma}\mathcal{C}_0(a,b).\label{eSmSMScvg}
\end{align}
where $\mathcal{C}_0(a,b)=\varphi(\frac{b}{\sigma})\P(\overline{R}_1-R_1\leq \frac{a}{\sigma}\vert R_1=\frac{b}{\sigma})$ as in \eqref{Mexconstant}.
\end{cor}
\subsection{Proofs of \eqref{mSSbde} - \eqref{cvgeSMS}}\label{A3}

\begin{proof}[Proof of \eqref{mSSbde}] This is given in Lemma B6 of \cite{AD14} when $\alpha=0$ and the increments are bounded. Let us prove the general case.

For $1\le r\le A\sqrt{m}$ with $A>10$ fixed, by \eqref{mSSbd}, it is clear that
\[
\P(\mS_m\ge-\alpha, S_m\in [r,r+1])\leq \frac{C_5(1+\alpha)(1+r+\alpha)}{m^{3/2}}\leq C_{31}\frac{1+\alpha}{m}e^{-C_{32}r^2/m}.
\]
It suffices to show \eqref{mSSbde} for $A\sqrt{m}\le r\le \epsilon_0m$. For any $x\in\R$, let
\[
T_x^+:=\inf\{k\ge0: S_k\ge x\},\textrm{ and } T_x^-:=\inf\{k\ge0: S_k<x\}.
\]
Then it is known that for any $0\le x\le y$, 
\[
\P_x(T_y^+<T_0^-)\leq C_{33}\frac{x+1}{y+1}.
\]
Recall that the increments of $S$ are $\xi_k, k\geq0$ which have finite exponential moments. Therefore, one has
\begin{align*}
&\P(\mS_m\ge-\alpha, S_m\in[r,r+1])\le  \P(\max_{k\leq m}\xi_k\ge r/2)+\P(\mS_m\ge-\alpha, S_m\in[r,r+1], S_{T^+_{\sqrt{m}}}\le \sqrt{m}+\frac{r}{2})\\
\le & C_{34} me^{-\delta_0 r/2}+\sum_{j=1}^{m}\P_\alpha(\mS_m\ge 0, S_m\in[r+\alpha,r+\alpha+1], T_{\sqrt{m}+\alpha}^+=j, S_j\in[\sqrt{m}+\alpha,r/2+\sqrt{m}+\alpha]).
\end{align*}
By Markov property at $T^+_{\sqrt{m}}$, 
\begin{align*}
&\sum_{j=1}^{m}\P_\alpha(\mS_m\ge 0, S_m\in[r+\alpha,r+\alpha+1], T_{\sqrt{m}+\alpha}^+=j, S_j\in[\sqrt{m}+\alpha,r/2+\sqrt{m}+\alpha])\\
\le&\sum_{j=1}^{m-1}\P_\alpha(\mS_j\ge0, T_{\sqrt{m}+\alpha}^+=j)\max_{\sqrt{m}+\alpha\le x\le \sqrt{m}+\alpha+r/2}\P(S_{m-j}\in [r+\alpha-x, r+\alpha-x+1])\\
\le &\P_\alpha(T_{\sqrt{m}+\alpha}^+<T_0^-)\max_{1\leq j\le m}\max_{r/3\le x\le r}\P(S_j\in[x,x+1])\le C_{35}\frac{\alpha+1}{\sqrt{m}+\alpha+1}\max_{1\leq j\le m}\max_{r/3\le x\le r}\P(S_j\in[x,x+1]).
\end{align*}
On the one hand, for $j\ge Kr$ with $K\ge1$ fixed and $r\gg1$, by Chernoff's bound,
\[
\max_{1\le j< Kr}\max_{r/3\le x\le r}\P(S_j\in [x,x+1])\le \max_{1\le j< Kr}\P(S_j\ge r/3)\le e^{-C_{36} r}.
\] 
On the other hand, for $Kr\le j\le m$, we use the following change of measure 
\[
\P^{(t)}((S_1,\cdots,S_j)\in\cdot)=E[e^{tS_j-j\phi_S(t)}; (S_1,\cdot,S_j)\in\cdot]
\]
with $\phi_S(t):=\log \E[e^{t\xi_1}]$. The probability $\P^{(t)}$ is well defined when $\phi_S(t)<\infty$. The corresponding expectation is denoted by $E^{(t)}$. It hence follows that for $t\in(-\delta_0/2,\delta_0/2)$,
\begin{align*}
\P(S_j\in[x,x+1])=&\E^{(t)}[e^{-tS_j+j\phi_S(t)}; S_j\in[x,x+1]]\\
\le & e^{-tx+j\phi_S(t)}\P^{(t)}(S_j\in[x,x+1])\le e^{-tr/3+C_{37} j t^2}\P^{(t)}(S_j\in[x,x+1]),
\end{align*}
as $\phi_S(t)\le C_{37} t^2$ for $|t|\le\delta_0/2$. Let us take $t=t_j=\frac{r}{6C_{37}j}$ so that $e^{-tr/3+C_{37} j t^2}\leq e^{-\frac{r^2}{36 C_{37} j}}$. Moreover, as under $\P^{(t)}$, $(S_k)$ is a random walk with i.i.d. increments and $\E^{(t)}[e^{s S_1}]<\infty$ for $s\in(0,\delta_0/2)$, Berry-Esseen theorem shows that there exists $C$ such that for $Kr\le j\le m$,
\[
\P^{(t)}(S_j\in[x,x+1])\leq \frac{C}{\sqrt{j}}.
\]
As a result, 
\[
\max_{Kr\le j\le m}\max_{r/3\le x\le r}\P(S_j\in [x,x+1])\le \max_{Kr\le j\le m}\frac{C}{\sqrt{j}}e^{-\frac{r^2}{36C_{37}j}}\le \frac{C}{\sqrt{m}}e^{-\frac{r^2}{36C_{37}m}},
\]
as long as $r\geq  A\sqrt{m}$ with $A\ge \sqrt{18C_{37}}$. We thus end up with
\[
\P(\mS_m\ge-\alpha, S_m\in[r,r+1])\le C_{34} me^{-\delta_0 r/2}+C_{35}\frac{\alpha+1}{\sqrt{m}+\alpha+1}\left(e^{-C_{36} r}\vee\frac{C}{\sqrt{m}}e^{-\frac{r^2}{36C_{37}m}}\right)%\leq \frac{C_{38}(1+\alpha)}{m}e^{-C_{38}r^2/m}.
\]
which suffices to obtain \eqref{mSSbde}.
\end{proof}

\begin{proof}[Proof of \eqref{mSSlargebde}]
Observe that by \eqref{mSSbde} and Chernoff's bound,
\begin{align*}
\P(\mS_m\geq-\alpha, S_m\geq r)\leq &\sum_{t=r}^{\epsilon_0 m}\P(\mS_m\ge-\alpha, S_m\in[t,t+1])+\P(S_m\ge \epsilon_0 m)\\
\leq &\sum_{t=r}^{\epsilon_0 m}C_{15}\frac{1+\alpha}{m}e^{-C_{16}\frac{t^2}{m}}+e^{-C_{38}m}\\
\leq &C_{17} \frac{1+\alpha}{r}e^{-C_{18}r^2/m}.
\end{align*}
\end{proof}

\begin{proof}[Proof of \eqref{mSMSSlargebde}]
Note that $(S_m-S_{m-i})_{0\leq i\leq m/2}$ is an independent copy of $(S_i)_{0\leq i\leq m}$. So, by \eqref{mSSlargebde} and \eqref{mSbd},
\begin{align*}
&\P(\mS_m\ge-\alpha, S_m=\MS_m\geq r)\\
\le &\P(\mS_{m/2}\ge-\alpha, S_{m/2}\ge r/2)\P(\mS_{m/2}\ge0)+\P(\mS_{m/2}\ge-\alpha)\P(\mS_{m/2}\ge0, S_{m/2}\ge r/2)\\
\le & C_{19} \frac{1+\alpha}{\sqrt{m} r}e^{-C_{20}r^2/m}.
\end{align*}
\end{proof}

\begin{proof}[Proof of \eqref{sumSbd}]
Because of \eqref{expmomS}, for $\lambda\in(0, 1+\delta_0)$ and $k\geq1$,
\[
\P(S_k\geq A)\leq e^{-\lambda A}\E[e^{\lambda S_k}]=e^{-\lambda A+k \phi_S(\lambda)},
\]
where $\phi_S(\lambda)=\log \E[e^{\lambda S_1}]$. Note that $\phi_S'(0)=\E[S_1]=0$ and $\phi_S(\lambda)\leq C_{37} \lambda^2$ for $\lambda\in (0,\delta_0/2)$ small. By taking $\lambda=\frac{1}{2C_{37}A^{\delta}}$ with $A$ sufficiently large, we have
\begin{align*}
\sum_{1\leq k\leq A^{1+\delta}}\P(S_k\geq A)\leq &\sum_{1\leq k\leq A^{1+\delta}}e^{-\lambda A+k  \phi_S(\lambda)}\\
\leq & \sum_{1\leq k\leq A^{1+\delta}}e^{-\lambda A+C_{37} k\lambda^2}\leq A^{1+\delta}e^{-\frac{A^{1-\delta}}{4C_{37}}}
\end{align*}
which suffices to conclude \eqref{sumSbd} for $\delta\in(0,1)$. In particular, for $\delta=0$, we can take $C_{37}>1/\delta_0$ so that \eqref{sumSbd} holds.
\end{proof}

\begin{proof}[Proof of \eqref{mSMS-Sbd}]
Observe that by Markov property at time $n/2$,
\begin{align*}
&\P(\mS_n\geq-\alpha, \MS_n=S_n\in[ r,r+1])\leq  \E[\ind{\mS_{n/2}\geq-\alpha, S_{n/2}\leq r+1}\P(\MS_{n/2}=S_{n/2}\in [r-x,r-x+1)\vert_{x=S_{n/2}}]\\
=& \E[\ind{\mS_{n/2}\geq-\alpha,S_{n/2}\leq r+1}\P(\mS_{n/2}\geq0, S_{n/2}\in [r-x,r-x+1])\vert_{x=S_{n/2}}]
\end{align*}
which by \eqref{mSSbd}, is bounded by
\begin{equation*}
C_{39}\E[\ind{\mS_{n/2}\geq-\alpha,S_{n/2}\leq r+1}\frac{(2+r-S_{n/2})}{n^{3/2}}]%=c\sum_{t=-\alpha}^{r}\frac{2+r-t}{n^{3/2}}\P(\mS_{n/2}\geq-\alpha, S_{n/2}\in [t,t+1])\\
\leq  C_{39}\frac{(2+r+\alpha)}{n^{3/2}}\P(\mS_{n/2}\geq-\alpha,S_{n/2}\leq r+1)
\end{equation*}
which by \eqref{mSSbd} implies that
\[
\P(\mS_n\geq-\alpha, \MS_n=S_n\in[ r,r+1])\leq C_{40}\frac{(1+\alpha)(2+r+\alpha)^3}{n^{3}}\leq C_{40} (1+\alpha)^4\frac{(1+r)^3}{n^3}.
\]
This completes the proof of \eqref{mSMS-Sbd}.
\end{proof}

\begin{proof}[Proof of \eqref{summSSbd}]
By use of \eqref{sumSbd} and \eqref{mSSbde}, we see that for $r\geq \eta^{-2}$ sufficiently large,
\begin{align*}
\sum_{1\leq k\le\eta r^2}\P(\mS_k\geq-\alpha, S_k\in[r,r+1])\leq &\sum_{k=1}^{r^{3/2}}\P(S_k\geq r)+\sum_{k=r^{3/2}}^{\eta r^2}\P(\mS_k\geq-\alpha, S_k\in[r,r+1])\\
\leq & e^{-C_{21}r^{1/2}}+\sum_{k=r^{3/2}}^{\eta r^2}C_{15}\frac{1+\alpha}{k}e^{-C_{16}\frac{r^2}{k}}\leq C_{41}(1+\alpha)\eta
\end{align*}
as $\sum_{k=r^{3/2}}^{\eta r^2}\frac{1}{k}e^{-C_{16}\frac{r^2}{k}}\leq \int_{r^{3/2}}^{\eta r^2+1}\frac{2}{x}e^{-C_{16}r^2/x}dx\leq \int_{\frac{1}{2\eta}}^{\sqrt{r}}\frac{1}{t}e^{-C_{16}t}dt\leq \frac{2\eta}{C_{16}}$.
\end{proof}

\begin{proof}[Proof of \eqref{2sumeSmSSbd}]
It is immediate that
\begin{align*}
\sum_{1\leq k\le\eta r^2}\E\left[\sum_{i=0}^{k}e^{-S_i}; \mS_k\geq 0, S_k\in[r,r+1]\right]\le& \sum_{1\leq k\le\eta r^2}k\P(\mS_k\geq0, S_k\in[r,r+1])\\
\le &r^{3/2}\sum_{k=1}^{r^{3/2}}\P(S_k\geq r)+\sum_{k=r^{3/2}}^{\eta r^2}k\P(\mS_k\geq-\alpha, S_k\in[r,r+1])
\end{align*}
which by \eqref{sumSbd} and \eqref{mSSbde}, is bounded by $C_{42}\eta$.
\end{proof}

\begin{proof}[Proof of \eqref{sumeSMSSbd}]
In fact, we only need to check that
\[
\sum_{k=1}^{n-1}\E[e^{-S_k}; \mS_n\geq0, S_n\in[x,x+1]]\leq C_{43}\frac{1+x}{n^{3/2}}.
\]
By Markov property time at time $k$ and then by \eqref{mSSbd}, one sees that
\begin{align*}
\sum_{k=1}^{n-1}\E[e^{-S_k}; \mS_n\geq0, S_n\in[x,x+1]]=&\sum_{k=1}^{n-1}\E[e^{-S_k}\ind{S_k\ge0}\P_{S_k}(\mS_{n-k}\ge0, S_{n-k}\in[x,x+1])]\\
\leq &\sum_{k=1}^{n-1}\frac{C_5(2+x)}{(n-k)^{3/2}}\E[(1+S_k)e^{-S_k}\ind{S_k\ge0}]\\
\leq &\sum_{k=1}^{n-1}\frac{C_{44}(1+x)}{(n-k)^{3/2}}\sum_{t=0}^\infty \frac{(1+t)^2e^{-t}}{k^{3/2}}\leq C_{45}(1+x)\sum_{k=1}^{n-1}\frac{1}{k^{3/2}(n-k)^{3/2}},
\end{align*}
which is less than $C_{46}(1+x)n^{-3/2}$.
\end{proof}

\begin{proof}[Proof of \eqref{eSmSSbd}]
Observe that by \eqref{mSSbd},
\begin{align*}
\E_\alpha[e^{-S_n}; \mS_n\geq0, S_n\geq A]\leq &\sum_{t=A}^\infty e^{-t}\P_\alpha(\mS_n\geq0, S_n\in [t,t+1])\\
\leq &\sum_{t=A}^\infty e^{-t}\frac{C_5(1+\alpha)(2+t)}{n^{3/2}}\leq C_{47} \frac{1+\alpha}{n^{3/2}}e^{-A/2}.
\end{align*}
\end{proof}

\begin{proof}[Proof of \eqref{eSmSSsmallbd}]
Note that by \eqref{mSSbd},
\begin{align*}
\E_\alpha[e^{S_n-A}; \mS_n\geq0, S_n\leq A]\leq &\sum_{t=0}^Ae^{t+1-A}\P_\alpha(\mS_n\geq0, S_n\in [t,t+1])\\
\leq &C_5\frac{1+\alpha}{n^{3/2}}\sum_{t=0}^A (2+t)e^{t+1-A}\leq C_{48}\frac{(1+\alpha)(1+A)}{n^{3/2}}.
\end{align*}
\end{proof}

\begin{proof}[Proof of \eqref{sumeSmSSsmallbd}]
In fact, by setting $\tau^-:=\inf\{k\geq0: S_k<0\}$ and $\Ren^-(dx)$ the renewal measure associated with the weak ascending ladder process of $(S_n)_{n\geq0}$, we have
\begin{align*}
\sum_{n\ge 0}\E[e^{S_n-A}; \mS_n\ge 0, S_n\le A]=&\E\left[\sum_{n=0}^{\tau^--1}e^{S_n-A}\ind{S_n\le A}\right]\\
=&\int_0^A e^{x-A}\Ren^-(dx)\leq C_{49},
\end{align*}
because there exists a constant $\cb_\Ren^->0$ such that for any $h>0$, $\Ren^-([x,x+h])\sim \cb_\Ren^- h$ as $x\to\infty$.
\end{proof}

\begin{proof}[Proof of \eqref{mSMSMS-Sbd}]
Let 
\[
\P_{\eqref{mSMSMS-Sbd}}:=\P(\mS_n\geq-\alpha, \MS_n\geq a r, \max_{k\leq n}(\MS_k-S_k)\leq b r, \MS_n-S_n\in[cr-K,cr+K]).
\]
By considering the first time hitting $\MS_n$ and by Markov property,
\begin{align*}
\P_{\eqref{mSMSMS-Sbd}}=&\sum_{j=1}^{n-1}\P(\mS_n\geq-\alpha, S_j=\MS_n\geq a r, S_j-S_n\in[cr-K,cr+K])\\
\leq & \sum_{j=1}^{n-1}\P(\mS_j\geq-\alpha, S_j=\MS_j\geq ar)\P(\MS_{n-j}\le 0, -S_{n-j}\in[cr-K,cr+K])
\end{align*} 
which by \eqref{mSMSlargebd} and by \eqref{mSSbd} for $(-S_n)_{n\geq0}$ is bounded by
\[
\sum_{j=1}^{n-1}\frac{C_{50}(1+\alpha)}{\sqrt{j} ar}\frac{(1+cr+K)(1+2K)}{(n-j)^{3/2}}.
\]
which is bounded by $\frac{C_{29}(1+\alpha)(1+K)^2(1+r)}{n^{3/2}}$ as $n\leq Ar^2$.
\end{proof}

\begin{proof}[Proof of \eqref{summSMSMS-Sbd}]
By considering the first time hitting $\MS_n$ and by Markov property, we have
\begin{align*}
&\sum_{k=1}^{\eta r^2}\P(\mS_k\geq0, \MS_k\geq ar, \MS_k-S_k\in[br, br+1])\\
\leq &\sum_{k=1}^{\eta r^2}\sum_{j=1}^{k-1}\P(\mS_k\ge 0, S_j=\MS_j\ge ar)\P(\MS_{k-j}\le 0, -S_{k-j}\in [br,br+1])\\
\leq &\sum_{j=1}^{\eta r^2}\P(\mS_k\ge 0, S_j=\MS_j\ge ar)\sum_{k=1}^{\eta r^2}\P(\MS_k\leq 0, -S_k\in[br,br+1]),
\end{align*}
which by \eqref{mSMSlargebd} and by \eqref{summSSbd} for $(-S_n)_{n\geq0}$ is bounded by
\[
\sum_{j=1}^{\eta r^2}\frac{C_{51}}{\sqrt{j} ar}\frac{\eta}{b^2}\leq C_{30}(a,b)\eta^{3/2}.
\]
\end{proof}

\begin{proof}[Proof of \eqref{cvgsumeSmSSn}]
Let
\[
\E_{\eqref{cvgsumeSmSSn}}:=\E\left[g(\sum_{i=1}^n e^{-S_i})\ind{\mS_n\geq-\alpha, \MS_n\leq x\sqrt{n}, S_n\in[y\sqrt{n}+z, y\sqrt{n}+z+h]}\right].
\]
Let $\delta\in(0,1/2)$.
\[
\E_{\eqref{cvgsumeSmSSn}}=\E\left[g(\sum_{i=1}^n e^{-S_i})\ind{\mS_n\geq-\alpha, \mS_{[n^\delta, n]}\geq n^{\delta/6}, \MS_{n^\delta}\leq n^{\delta}, \MS_n\leq x\sqrt{n}, S_n\in[y\sqrt{n}+z, y\sqrt{n}+z+h]}\right]+Error_\eqref{cvgsumeSmSSn}
\]
where
\begin{align*}
Error_\eqref{cvgsumeSmSSn}\le &||g||_\infty\E\left[\ind{\mS_n\geq-\alpha, \mS_{[n^\delta, n]}\leq n^{\delta/6}, S_n\in[y\sqrt{n}+z, y\sqrt{n}+z+h]}\right]\\
&+||g||_\infty\E\left[\ind{\mS_n\geq-\alpha, \MS_{n^\delta}\geq n^{\delta}, S_n\in[y\sqrt{n}+z, y\sqrt{n}+z+h]}\right]
\end{align*}
First, let us check that $Error_\eqref{cvgsumeSmSSn}=o_n(\frac{1}{n})$. On the one hand, by \eqref{sumSbd}, 
\begin{align*}
\E\left[\ind{\mS_n\geq-\alpha, \MS_{n^\delta}\geq n^{\delta}, S_n\in[y\sqrt{n}+z, y\sqrt{n}+z+h]}\right]\leq &\P(\MS_{n^\delta}\ge n^{\delta})\\
\leq & \sum_{k=1}^{n^\delta}\P(S_k\geq n^\delta)\le e^{-C_{21} n^\delta}=o_n(\frac{1}{n}).
\end{align*}
On the other hand, by considering the first time hitting $\mS_{[n^\delta, n]}$, 
\begin{align*}
&\E\left[\ind{\mS_n\geq-\alpha, \mS_{[n^\delta, n]}\leq n^{\delta/6}, S_n\in[y\sqrt{n}+z, y\sqrt{n}+z+h]}\right]\\
\leq & \sum_{j=n^\delta}^{n-1} \E\left[\ind{\mS_j\geq-\alpha, S_j\leq n^{\delta/6}}\P(\mS_{n-j}\geq 0, S_{n-j}\in[y\sqrt{n}+z-t, y\sqrt{n}+z-t+h])\vert_{t=S_j}\right]\\
\leq & \sum_{j=n^\delta}^{n-\sqrt{n}}\frac{C_5(1+\alpha)(1+\alpha+n^{\delta/6})^2}{j^{3/2}}\frac{C_5(1+h)(1+y\sqrt{n}+z+\alpha+h)}{(n-j)^{3/2}}\\
&+\sum_{j=n-\sqrt{n}}^{n-1}\frac{C_5(1+\alpha)(1+\alpha+n^{\delta/6})^2}{j^{3/2}}\P(S_{n-j}\geq y\sqrt{n}+z-n^{\delta/6})
\end{align*}
where the last inequality follows from \eqref{mSSbd}. By \eqref{sumSbd},
\[
\E\left[\ind{\mS_n\geq-\alpha, \mS_{[n^\delta, n]}\leq n^{\delta/6}, S_n\in[y\sqrt{n}+z, y\sqrt{n}+z+h]}\right]\leq C_{52}(1+\alpha)^4 n^{-\delta/6-1}+C_{52}(1+\alpha)^3 n^{\delta/3-3/2}e^{-c_{53}\sqrt{n}}
\]
which is $o_n(\frac{1}{n})$. It remains to prove the convergence of
\begin{equation}
\E_{\eqref{cvgsumeSmSSn}}^+:=\E\left[g(\sum_{i=1}^n e^{-S_i})\ind{\mS_n\geq-\alpha, \mS_{[n^\delta, n]}\geq n^{\delta/6}, \MS_{n^\delta}\leq n^{\delta}, \MS_n\leq x\sqrt{n}, S_n\in[y\sqrt{n}+z, y\sqrt{n}+z+h)}\right].
\end{equation}
As $\mS_{[n^\delta, n]}\geq n^{\delta/6}$ and $g$ is uniformly continuous on any compact set of $[0,\infty)$,
\begin{equation}\label{unifcong}
g(\sum_{i=1}^n e^{-S_i})=g(\sum_{i=1}^{n^\delta}e^{-S_i})+o_n(1).
\end{equation}
In fact, we need to work on $\{\sum_{i=1}^n e^{-S_i}\leq K\}$ with $K>0$ fixed. It is easy to check that 
\begin{align*}
&\E[(\sum_{i=1}^n e^{-S_i})\ind{\mS_n\geq-\alpha, S_n\in[y\sqrt{n}+z, y\sqrt{n}+z+h)}]\\
\leq &\sum_{i=1}^{n-1}\E_\alpha[e^{\alpha-S_i}\ind{\mS_i\ge0}\P_{S_i}(\mS_{n-i}\geq0, S_{n-i}\in[y\sqrt{n}+z, y\sqrt{n}+z+h))]+e^{-y\sqrt{n}-z}\\
\leq &\sum_{i=1}^{n-1}\frac{C_{54}(1+\alpha)(1+y\sqrt{n}+z+h)e^\alpha}{i^{3/2}(n-i)^{3/2}}\leq \frac{C_{55}}{n}.
\end{align*}
So, 
\[
\E_\eqref{cvgsumeSmSSn}=\E\left[g(\sum_{i=1}^n e^{-S_i})\ind{\mS_n\geq-\alpha, \MS_n\leq x\sqrt{n}, S_n\in[y\sqrt{n}+z, y\sqrt{n}+z+h), \sum_{i=1}^n e^{-S_i}\leq K}\right]+\frac{o_K(1)}{n}.
\]
Let us work directly with \eqref{unifcong}. By \eqref{mSSbd}, it is clear that 
\begin{align*}
&o_n(1)\P\left(\mS_n\geq-\alpha, \mS_{[n^\delta, n]}\geq n^{\delta/6}, \MS_{n^\delta}\leq n^{\delta}, \MS_n\leq x\sqrt{n}, S_n\in[y\sqrt{n}+z, y\sqrt{n}+z+h)\right)\\
\leq & o_n(1)\P\left(\mS_n\ge-\alpha, S_n\in[y\sqrt{n}+z, y\sqrt{n}+z+h)\right)\\
\leq & o_n(1) \frac{C_5(1+\alpha)(1+h)(1+y\sqrt{n}+z+\alpha+h)}{n^{3/2}}=o_n(\frac{1}{n}).
\end{align*}
It then follows that
\begin{align}\label{EA26}
&\E_{\eqref{cvgsumeSmSSn}}^+=\E\left[g(\sum_{i=1}^{n^\delta} e^{-S_i})\ind{\mS_n\geq-\alpha, \mS_{[n^\delta, n]}\geq n^{\delta/6}, \MS_{n^\delta}\leq n^{\delta}, \MS_n\leq x\sqrt{n}, S_n\in[y\sqrt{n}+z, y\sqrt{n}+z+h)}\right]+o_n(\frac{1}{n})\nonumber\\
=&\E\left[g(\sum_{i=1}^{n^\delta} e^{-S_i})\ind{\mS_n\geq-\alpha, \MS_{n^\delta}\leq n^{\delta}, \MS_n\leq x\sqrt{n}, S_n\in[y\sqrt{n}+z, y\sqrt{n}+z+h)}\right]+o_n(\frac{1}{n})\nonumber\\
=&\E\left[g(\sum_{i=1}^{n^\delta} e^{-S_i})\ind{\mS_{n^\delta}\geq-\alpha, \MS_{n^\delta}\leq n^{\delta}}\P_{S_{n^\delta}}\left(\mS_{n-n^\delta}\geq-\alpha, \MS_{n-n^\delta}\leq x\sqrt{n}, S_{n-n^\delta}-y\sqrt{n}-z\in[0,h)\right)\right]+o_n(\frac{1}{n}).
\end{align}
where the last equality is obtained by Markov property at time $n^\delta$. For $t=S_{n^\delta}\in[-\alpha, n^\delta]$, one sees that
\begin{align*}
&\P_{t}\left(\mS_{n-n^\delta}\geq-\alpha, \MS_{n-n^\delta}\leq x\sqrt{n}, S_{n-n^\delta}-y\sqrt{n}-z\in[0,h)\right)\\
=&\P_{t+\alpha}\left( \MS_{n-n^\delta}\leq x\sqrt{n}+\alpha\vert\mS_{n-n^\delta}\geq 0, S_{n-n^\delta}-y\sqrt{n}-z+\alpha\in[0,h)\right)\\
&\times\P_{t+\alpha}(\mS_{n-n^\delta}\ge0, S_{n-n^\delta}-y\sqrt{n}-z \in[\alpha,\alpha+h)).
\end{align*}
By (5.3) of \cite{CC13}, 
\[
\P_{t+\alpha}(\mS_{n-n^\delta}\ge0, S_{n-n^\delta}-y\sqrt{n}-z \in[\alpha,\alpha+h))=\frac{\cb_+}{\sigma n}\Ren(t+\alpha)(\psi(\frac{y}{\sigma})h+o_n(1)),
\]
where the constant $\cb_+=\lim_{n\rightarrow \infty}\sqrt{n}\P(\mS_n\ge0)\in(0,\infty)$. Moreover, in the spirit of Theorem 2.4 of \cite{CC13}, we can say that 
\[
\P_{t+\alpha}\left( \MS_{n-n^\delta}\leq x\sqrt{n}+\alpha\vert\mS_{n-n^\delta}\geq 0, S_{n-n^\delta}-y\sqrt{n}-z+\alpha\in[0,h)\right)\rightarrow \P(\overline{R}_1\leq \frac{x}{\sigma}\vert R_1=\frac{y}{\sigma})
\]
uniformly for $(x,y)$ in a compact set of $(0,\infty)^2$. In fact, in Theorem 2.4 of \cite{CC13}, the \textit{Hypothesis} 2.2 is needed for the density of increments. However, in this work, as we consider $\{S_n\in[y,y+h]\}$ instead of $\{S_n=y\}$, the \textit{Hypothesis} 2.2 is not necessary. As a result, 
\begin{multline*}
\P_{t}\left(\mS_{n-n^\delta}\geq-\alpha, \MS_{n-n^\delta}\leq x\sqrt{n}, S_{n-n^\delta}-y\sqrt{n}-z\in[0,h)\right)\vert_{t=S_{n^\delta}}\\
=\frac{\cb_+}{n\sigma}\Ren(S_{n^\delta}+\alpha)\psi(\frac{y}{\sigma})h\P(\overline{R}_1\leq \frac{x}{\sigma}\vert R_1=\frac{y}{\sigma})(1+o_n(1)).
\end{multline*}
Plugging it into \eqref{EA26} yields that
\begin{align*}
\E_{\eqref{cvgsumeSmSSn}}^+=&\E\left[g(\sum_{i=1}^{n^\delta} e^{-S_i})\ind{\mS_{n^\delta}\geq-\alpha, \MS_{n^\delta}\leq n^{\delta}}\frac{\cb_+}{n\sigma}\Ren(S_{n^\delta}+\alpha)\psi(\frac{y}{\sigma})h\P(\overline{R}_1\leq \frac{x}{\sigma}\vert R_1=\frac{y}{\sigma})(1+o_n(1))\right]+o_n(\frac{1}{n})\\
=&\E_\alpha\left[g(\sum_{i=1}^{n^\delta} e^{\alpha-S_i})\ind{\mS_{n^\delta}\geq 0}\Ren(S_{n^\delta})\right]\frac{\cb_+}{n\sigma}\psi(\frac{y}{\sigma})h\P(\overline{R}_1\leq \frac{x}{\sigma}\vert R_1=\frac{y}{\sigma})(1+o_n(1))+o_n(\frac{1}{n})\\
=&\Ren(\alpha)\E_\alpha\left[g(\sum_{i=1}^{n^\delta}e^{\alpha-\zeta_i})\right]\frac{\cb_+}{n\sigma}\psi(\frac{y}{\sigma})h\P(\overline{R}_1\leq \frac{x}{\sigma}\vert R_1=\frac{y}{\sigma})(1+o_n(1))+o_n(\frac{1}{n})
\end{align*}
where $(\zeta_i)_{i\geq0}$ is a Markov chain taking values in $\R_+$, satisfying $\P_\alpha(\zeta_0=\alpha)=1$,  with transition probability $p(x,dy)=\ind{y>0}\frac{\Ren(y)}{\Ren(x)}\P_x(S_1\in dy)$. It is known that for any $\delta\in(0,1/2)$ small, $\P_\alpha$-a.s., $\zeta_n\geq n^{1/2-\delta}$ for $n\gg1$. So $\Hinf_\infty=\sum_{i=0}^\infty e^{-\zeta_i}$ is a positive random variable taking values in $\R_+$. It is obvious that
\[
\sum_{i=1}^{n^\delta}e^{\alpha-\zeta_i}\rightarrow \sum_{i=1}^{\infty}e^{\alpha-\zeta_i}=e^{\alpha}\Hinf_\infty-1.
\]
As $g$ is bounded, one obtains \eqref{cvgsumeSmSSn} by dominated convergence.
\end{proof}

\begin{proof}[Proof of \eqref{cvgeSMS}]
Let 
\[
\E_{\eqref{cvgeSMS}}:=\E\left[F(e^{b\sqrt{n}-(\MS_n-S_n)},\sum_{i=0}^n e^{S_i-\MS_n})\ind{\mS_n\geq-\alpha, \MS_n\geq a\sqrt{n}+a_n, \max_{0\leq k\leq n}(\MS_k-S_k)\leq a\sqrt{n}+a'_n, \MS_n-S_n\in[b\sqrt{n}-K, b\sqrt{n}+K)} \right].
\]
By considering the first time hitting $\MS_n$, we have
\begin{multline}
\E_{\eqref{cvgeSMS}}=\sum_{j=1}^{n-1}\E\left[F(e^{b\sqrt{n}-(\MS_n-S_n)},\sum_{i=0}^n e^{S_i-\MS_n})\ind{\mS_n\geq-\alpha, \MS_{j-1}<S_j, S_j=\MS_n\geq a\sqrt{n}+a_n}\right.\\
\left.\times\ind{\max_{0\leq k\leq n}(\MS_k-S_k)\leq a\sqrt{n}+a'_n, \MS_n-S_n\in[b\sqrt{n}-K, b\sqrt{n}+K)} \right].
\end{multline}
By Markov property at time $j$, this is equal to $\sum_{j=1}^{n-1}\sum_{\ell=-K/h}^{K/h-1}\E_{\eqref{cvgeSMS}}(j,n,\ell)$ where
\begin{multline}
\E_{\eqref{cvgeSMS}}(j,n,\ell):=\E\left[F(e^{b\sqrt{n}+R_{n-j}},\sum_{i=0}^j e^{S_i-S_j}+\sum_{k=1}^{n-j}e^{R_k})\ind{\mS_j\geq-\alpha, \MS_{j-1}<S_j, S_j\geq a\sqrt{n}+a_n, \max_{0\leq k\leq j}(\MS_k-S_k)\leq a\sqrt{n}+a'_n}\right.\\
\left.\times\ind{\max_{k\leq n-j}R_k\le0, \min_{0\leq k\leq n-j}(-R_k)\leq (a\sqrt{n}+a'_n)\wedge(\alpha+S_j), -R_{n-j}\in[b\sqrt{n}+\ell h, b\sqrt{n}+\ell h+h)} \right]
\end{multline}
with $(R_{k})_{k\geq0}$ is an independent copy of the random walk $(S_k)_{k\geq0}$. %In fact, according to \eqref{badend1ex} and its proof in Section \ref{lems}, we only need to take the sum $\sum_{j=1}^{n-1}\sum_{\ell=-K}^K$ with $K\geq1/h$ fixed. 
First, let us prove that for $n\gg1$,
\[
\sum_{j\le\epsilon n\textrm{ or } j\ge n-\epsilon n}\sum_{\ell=-K/h}^{K/h-1}n\E_{\eqref{cvgeSMS}}(j,n,\ell)=o_\epsilon(1)
\]
For $j\le \epsilon n$, similarly to \eqref{Markovpositivex}, by \eqref{sumSbd}, \eqref{mSMSlargebd} and \eqref{mSSbd} one has
\begin{align*}
&\sum_{j\le\epsilon n}\sum_{\ell=-K/h}^{K/h-1}\E_{\eqref{cvgeSMS}}(j,n,\ell)%\leq &\sum_{j\le\epsilon n\textrm{ or } j\ge n-\epsilon n}\P(\mS_j\geq-\alpha, S_j\geq a\sqrt{n}+a_n)\P(\MS_{n-j}\le0, -S_{n-j}\in [c\sqrt{n}+\ell h,c\sqrt{n}+\ell h+h])\\
\leq  \sum_{j\leq n^{3/4}}\P(S_j\geq a\sqrt{n}+a_n)\\
&\qquad+\sum_{j=n^{3/4}}^{\delta n}\P(\mS_j\geq-\alpha, \MS_j=S_j\geq a\sqrt{n}+a_n)\P(\MS_{n-j}\le0, -S_{n-j}\in [b\sqrt{n}-K,b\sqrt{n}+K])\\
\leq &e^{-C_{56}n^{1/4}}+\sum_{j=n^{3/4}}^{\epsilon n}\frac{C_{57}(1+\alpha)}{\sqrt{j}(a\sqrt{n}+a_n)}\frac{(1+b\sqrt{n}+K)}{(n-j)^{3/2}}=\frac{o_n(1)+o_\epsilon(1)}{n}.
\end{align*}
For $j\ge n-\epsilon n$, by \eqref{mSMSbd} and \eqref{summSSbd}, one has
\begin{align*}
&\sum_{j=n-\epsilon n}^n\sum_{\ell=-K/h}^{K/h-1}\E_{\eqref{cvgeSMS}}(j,n,\ell)\\
\leq&\sum_{j=1}^{\epsilon n}\P(\mS_{n-j}\geq-\alpha, \MS_{n-j}=S_{n-j}\geq a\sqrt{n}+a_n)\P(\MS_{j}\le0, -S_{j}\in [b\sqrt{n}-K,b\sqrt{n}+K])\\
\leq &\frac{C_{58}(1+\alpha)(1+2K)}{n}\epsilon=\frac{o_\epsilon(1)}{n}.
\end{align*}
Thus, it remains to study $\sum_{j=\epsilon n}^{n-\epsilon n}\sum_{\ell=-K/h}^{K/h-1}n\E_{\eqref{cvgeSMS}}(j,n,\ell)$. Recall that $F(x,y)=\frac{x}{y}e^{-x/y}$ with $x>0$ and $y\geq1$. So, for any fixed $h>0$
\[
\sup_{x>0, y\ge1}|F(xe^h,y)-F(x,y)|\leq 2(e^h-1)\textrm{ and } \sup_{x>0, y\ge1}|F(x,y+h)-F(x,y)|\leq 2 h.
\]
Therefore, on $-R_{n-j}\in[b\sqrt{n}+\ell h, b\sqrt{n}+\ell h+h)$,
\[
F(e^{b\sqrt{n}+R_{n-j}},\sum_{i=0}^j e^{S_i-S_j}+\sum_{k=1}^{n-j}e^{R_k})=F(e^{-\ell h},\sum_{i=0}^j e^{S_i-S_j}+\sum_{k=1}^{n-j}e^{R_k})+o_h(1).
\]
Moreover, let $(S_k^{(-)})_{k\geq0}$ be the random walk distributed as the reflected walk $-S$, and independent of $S$. Observe that for $\ell\in[-K,K]$ with $K\geq1/h$ fixed integer
\begin{multline*}
\E_{\eqref{cvgeSMS}}(j,n,\ell)=\E\left[\ind{\mS_j\geq-\alpha, \MS_{j-1}<S_j, S_j\geq a\sqrt{n}+a_n, \max_{0\leq k\leq j}(\MS_k-S_k)\leq a\sqrt{n}+a'_n}\right.\\
\left.\times\E[(F(e^{-\ell h},t+\sum_{k=1}^{n-j}e^{-S_k^{(-)}})+o_h(1))\ind{\mS_{n-j}^{(-)}\ge0, \MS_{n-j}^{(-)}\le (a\sqrt{n}+a'_n)\wedge(\alpha+s), S_{n-j}^{(-)}-c\sqrt{n}\in[\ell h,\ell h+h)}]\vert_{t=\sum_{i=0}^j e^{S_i-S_j},s=S_j}\right].
\end{multline*}
By use of \eqref{cvgsumeSmSSn} for $S^{(-)}$, one sees that $\epsilon n\leq j\leq n-\epsilon n$ with $\epsilon\in(0,1/2)$, for $n\gg1$,
\begin{multline*}
(n-j)\E[(F(e^{-\ell h},t+\sum_{k=1}^{n-j}e^{-S_k^{(-)}})+o_h(1))\ind{\mS_{n-j}^{(-)}\ge0, \MS_{n-j}^{(-)}\le (a\sqrt{n}+a'_n)\wedge(\alpha+s), S_{n-j}^{(-)}-b\sqrt{n}\in[\ell h,\ell h+h)}]\\
=\frac{\cb_-h}{\sigma}\E[(F(e^{-\ell h}, t+\Hinf^{(-)}_\infty-1)+o_n(1))]\mathcal{C}_0(a\wedge \frac{s}{\sqrt{n}}-b,b)+o_n(1),
\end{multline*}
where $\cb_-:=\lim_{n\to\infty}\sqrt{n}\P(\MS_n\le0)$ and $\Hinf^{(-)}_\infty:=\sum_{k=0}^\infty e^{-\zeta_k^{(-)}}$ with $(\zeta_k^{(-)})_{k\geq0}$ the Markov chain obtained from the reflected walk. It follows that
\begin{multline*}
n\E_{\eqref{cvgeSMS}}(j,n,\ell)=\frac{n}{n-j}\E\left[\ind{\mS_j\geq-\alpha, \MS_{j-1}<S_j, S_j\geq a\sqrt{n}+a_n, \max_{0\leq k\leq j}(\MS_k-S_k)\leq a\sqrt{n}+a'_n}\right.\\
\left.\times\left(\frac{\cb_-h}{\sigma}\E[F(e^{-\ell h}, \sum_{i=0}^je^{S_i-S_j}+\Hinf^{(-)}_\infty-1)+o_n(1)]\mathcal{C}_0(a-b,b)+o_n(1)\right)\right]
\end{multline*}
which by Fact \ref{mSMScvg} is equal to
\[
\frac{n}{(n-j)j}\mathcal{C}_{\frac{a\sqrt{n}}{\sqrt{j}},\frac{a\sqrt{n}}{\sqrt{j}}}\Ren(\alpha)\frac{\cb_-h}{\sigma}\E[(F(e^{-\ell h}, \Hinf_\infty+\Hinf^{(-)}_\infty-1)]\mathcal{C}_0(a-b,b)+o_n(\frac{1}{n}).
\]
This leads to
\begin{align*}
&\sum_{j=\epsilon n}^{n-\epsilon n}\sum_{\ell=-K/h}^{K/h-1}n\E_{\eqref{cvgeSMS}}(j,n,\ell)\\
=& \sum_{j=\epsilon n}^{n-\epsilon n}\sum_{\ell=-K}^K\frac{n}{j(n-j)}\mathcal{C}_{\frac{a\sqrt{n}}{\sqrt{j}},\frac{a\sqrt{n}}{\sqrt{j}}}\Ren(\alpha)\frac{\cb_-h}{\sigma}\E[(F(e^{-\ell h}, \Hinf_\infty+\Hinf^{(-)}_\infty-1)]\mathcal{C}_0(a-b,b)+o_n(1)\\
=&\int_{\epsilon}^{1-\epsilon}\mathcal{C}_{\frac{a}{\sqrt{t}},\frac{a}{\sqrt{t}}}\frac{dt}{t(1-t)}\frac{\cb_-\Ren(\alpha)}{\sigma}\mathcal{C}_0(a-b,b)\int_{-K}^{K}\E[(F(e^{-s}, \Hinf_\infty+\Hinf^{(-)}_\infty-1)]ds+o_n(1)+o_h(1)Kh
\end{align*}
Letting $n\to\infty$ then letting $h\to0$ and $\epsilon\to0$, we conclude \eqref{cvgeSMS}.
\end{proof}

\bibliographystyle{alpha}
\bibliography{thbiblio.bib}

\begin{thebibliography}{LPP96}

\bibitem[AC18]{AC18}
P.~Andreoletti and X.~Chen.
\newblock Range and critical generations of a random walk on galton-watson
  trees.
\newblock {\em Ann. Inst. Henri Poincar\'e}, 54(1):466--513, 2018.

\bibitem[AD14]{AD14}
P.~Andreoletti and P.~Debs.
\newblock Spread of visited sites of a random walk along the generations of a
  branching process.
\newblock {\em Electron. J. Probab.}, 19(42):22 pp, 2014.

\bibitem[AD20]{AD18+}
P.~Andreoletti and R.~Diel.
\newblock The heavy range of randomly biased walks on trees.
\newblock {\em Stoch. Proc. Appl.}, 130(2):962--999, 2020.

\bibitem[Afa93]{Afa93}
V.~I. Afanasyev.
\newblock A limit theorem for a critical branching process in a random
  environment.
\newblock {\em Diskret. Mat.}, 5:45--58, 1993.

\bibitem[Aid13]{Aid13}
E.~Aid\'ekon.
\newblock Convergence in law of the minimum of a branching random walk.
\newblock {\em Ann. probab.}, 41:1362--1426, 2013.

\bibitem[AS14]{AS14}
E.~Aidekon and Z.~Shi.
\newblock The seneta-heyde scaling for the branching random walk.
\newblock {\em The Annals of Probability}, 42(3):959--993, 2014.

\bibitem[BK04]{BigKyp}
J.~D. Biggins and A.E. Kyprianou.
\newblock Measure change in multitype branching.
\newblock {\em Adv. Appl. Probab.}, 36:544--581, 2004.

\bibitem[BM19]{BM19}
P.~Boutaud and P.~Maillard.
\newblock A revisited proof of the Seneta-Heyde norming for branching random
  walks under optimal assumptions.
\newblock {\em Electron. J. Probab.}, 24(99):22pp, 2019.

\bibitem[CC13]{CC13}
F.~Caravenna and L.~Chaumont.
\newblock An invariance principle for random walk bridges conditioned to stay
  positive.
\newblock {\em Electron. J. Probab.}, 18(60):32 pp, 2013.

\bibitem[CdR]{CdR19+}
X.~Chen and L.~de~Raph\'elis.
\newblock The most visited edges of randomly biased random walks on a supercritical Galton-Watson tree under the quenched law.
\newblock Preprint.

\bibitem[Far11]{Faraud}
G.~Faraud.
\newblock A central limit theorem for random walk in a random environment on
  marked galton-watson trees.
\newblock {\em Electron. J. Probab.}, 16(6):174--215, 2011.

\bibitem[FHS12]{FHS11}
G.~Faraud, Y.~Hu, and Z.~Shi.
\newblock Almost sure convergence for stochastically biased random walks on
  trees.
\newblock {\em Probab. Theory Relat. Fields}, 154:621--660, 2012.

\bibitem[HS07]{HS07}
Y.~Hu and Z.~Shi.
\newblock Slow movement of random walk in random environment on a regular tree.
\newblock {\em Ann. probab.}, 35(5):1978--1997, 2007.

\bibitem[HS15]{HuShi16+}
Y.~Hu and Z.~Shi.
\newblock The most visited sites of biased random walks on trees.
\newblock {\em Electron. J. Probab.}, 20(62):14pp, 2015.

\bibitem[HS16]{HuShi16}
Y.~Hu and Z.~Shi.
\newblock The slow regime of randomly biased walks on trees.
\newblock {\em Ann. probab.}, 44(6):3893--3933, 2016.

\bibitem[LP92]{LyonPema}
R.~Lyons and R.~Pemantle.
\newblock Random walk in a random environment and first-passage percolation on
  trees.
\newblock {\em Ann. probab.}, 20:125--136, 1992.

\bibitem[LPP95]{LPP95}
R.~Lyons, R.~Pemantle, and Y.~Peres.
\newblock Ergodic theory on galton-watson trees: Speed of random walk and
  dimension of harmonic measure.
\newblock {\em Ergodic Theory Dynam. Systems}, \textbf{15}:\ 593--619, 1995.

\bibitem[LPP96]{LPP96}
R.~Lyons, R.~Pemantle, and Y.~Peres.
\newblock Biased random walks on galton-watson trees.
\newblock {\em Probab. Theory Related Fields}, 106:249--264, 1996.

\bibitem[Lyo90]{Lyons90}
R.~Lyons.
\newblock Random walks and percolation on trees.
\newblock {\em Ann. Probab.}, 18:931--958, 1990.

\bibitem[Lyo92]{Lyons92}
R.~Lyons.
\newblock Random walks, capacity and percolation on trees.
\newblock {\em Ann. Probab.}, 20:2043--2088, 1992.

\bibitem[Lyo97]{Lyons97}
R.~Lyons.
\newblock A simple path to biggins' martingale convergence for branching random
  walk.
\newblock In {\em Classical and Modern Branching Processes (Minneapolis, MN,
  1994). IMA Vol. Math. Appl.}, volume~84, pages 217--221. Springer, New York,
  1997.

\end{thebibliography}
\end{document}